\documentclass[12pt]{amsart}

\usepackage[colorlinks=true, pdfstartview=FitV, linkcolor=blue, citecolor=blue, urlcolor=blue, breaklinks=true]{hyperref}
\usepackage{
  amsmath,
  amsfonts,
  amssymb,
  amsthm,
  amscd,
  comment,
  paralist,
  etoolbox,
  mathtools,
  enumitem,
  mathdots
}
\usepackage[usenames,dvipsnames]{xcolor}
\usepackage{mathrsfs} 
\usepackage[labelformat=simple]{subcaption}

\usepackage[all]{xy}
\usepackage{tikz}
\usetikzlibrary{arrows,patterns}

\newcommand\Z{\mathbb{Z}}
\newcommand\Q{\mathbb{Q}}
\newcommand\R{\mathbb{R}}
\newcommand\N{\mathbb{N}}
\newcommand\kk{\Bbbk}

\newcommand\bb{\mathbf{b}}
\newcommand\bB{\mathbf{B}}
\newcommand\bc{\mathbf{c}}
\newcommand\bF{\mathbf{F}}
\newcommand\one{\mathbf{1}}

\newcommand\cB{\mathcal{B}}
\newcommand\cC{\mathcal{C}}
\newcommand\cH{\mathcal{H}}
\newcommand\tcH{\tilde{\mathcal{H}}}

\newcommand\fg{\mathfrak{g}}
\newcommand\fsl{\mathfrak{sl}}
\newcommand\fh{\mathfrak{h}}

\newcommand\gr{\mathrm{gr}}

\newcommand\op{\mathrm{op}}
\newcommand\rR{\mathrm{R}}
\newcommand\rL{\mathrm{L}}

\newcommand\sQ{\mathsf{Q}}
\newcommand\sx{\mathsf{x}}
\newcommand\sy{\mathsf{y}}
\newcommand\sz{\mathsf{z}}

\newcommand{\md}{\textup{-mod}}

\newcommand{\pmd}{\textup{-pmod}}

\newcommand\pr[1]{\prescript{r}{}{#1}}
\newcommand\pl[1]{\prescript{\ell}{}{#1}}

\DeclareMathOperator{\End}{End}

\DeclareMathOperator{\Hom}{Hom}

\DeclareMathOperator{\Ind}{Ind}
\DeclareMathOperator{\Mor}{Mor}
\DeclareMathOperator{\Res}{Res}
\DeclareMathOperator{\Sym}{Sym}

\DeclareMathOperator{\tr}{tr}
\DeclareMathOperator{\id}{id}
\DeclareMathOperator{\Ob}{Ob}

\newcommand\bluedot[1]{\filldraw[blue] #1 circle (2pt)}
\newcommand\dotlabel[1]{$\scriptstyle{#1}$}
\newcommand\regionlabel[1]{$\scriptstyle{#1}$}

\tikzset{anchorbase/.style={>=stealth,baseline={([yshift=-0.5ex]current bounding box.center)}}}

\newtheorem{theo}{Theorem}[section]
\newtheorem{prop}[theo]{Proposition}
\newtheorem{lem}[theo]{Lemma}
\newtheorem{cor}[theo]{Corollary}
\newtheorem{conj}[theo]{Conjecture}

\theoremstyle{definition}
\newtheorem{defin}[theo]{Definition}
\newtheorem{rem}[theo]{Remark}

\numberwithin{equation}{section}
\allowdisplaybreaks

\newtoggle{comments}
\newtoggle{details}
\newtoggle{detailsnote}

\toggletrue{detailsnote}   

\iftoggle{comments}{%
  \newcommand{\acomments}[1]{
    \ \\
    {\color{red}
      \textbf{AS:} #1
    }
    \ \\
    }
  \newcommand{\mcomments}[1]{
    \ \\
    {\color{red}
      \textbf{MM:} #1
    }
    \ \\
    }
}{%
  \newcommand{\acomments}[1]{}
  \newcommand{\mcomments}[1]{}
}

\iftoggle{details}{%
  \newcommand{\details}[1]{
      \ \\
      {\color{OliveGreen}
        \textbf{Details:} #1
      }
      \\
  }
}{%
  \newcommand{\details}[1]{}
}

\title[Higher level Heisenberg categorification]{Degenerate cyclotomic Hecke algebras and higher level Heisenberg categorification}

\author{Marco Mackaay}
\address[M.M.]{Departamento de Matem\'{a}tica, Universidade do Algarve and CAMGSD, Instituto Superior T\'{e}cnico}
\email{mmackaay@ualg.pt}

\author{Alistair Savage}
\address[A.S.]{Department of Mathematics and Statistics, University of Ottawa}
\urladdr{\href{http://alistairsavage.ca}{alistairsavage.ca}, \textrm{\textit{ORCiD}:} \href{https://orcid.org/0000-0002-2859-0239}{orcid.org/0000-0002-2859-0239}}
\email{alistair.savage@uottawa.ca}


\begin{document}

\begin{abstract}
  We associate a monoidal category $\cH^\lambda$ to each dominant integral weight $\lambda$ of $\widehat{\fsl}_p$ or $\fsl_\infty$.  These categories, defined in terms of planar diagrams, act naturally on categories of modules for the degenerate cyclotomic Hecke algebras associated to $\lambda$.  We show that, in the $\fsl_\infty$ case, the level $d$ Heisenberg algebra embeds into the Grothendieck ring of $\cH^\lambda$, where $d$ is the level of $\lambda$.  The categories $\cH^\lambda$ can be viewed as a graphical calculus describing induction and restriction functors between categories of modules for degenerate cyclotomic Hecke algebras, together with their natural transformations.  As an application of this tool, we prove a new result concerning centralizers for degenerate cyclotomic Hecke algebras.
\end{abstract}

\maketitle
\thispagestyle{empty}

\tableofcontents

%
\section{Introduction}
%

The Heisenberg algebra plays a fundamental role in many areas of mathematics and physics.  The universal enveloping algebra of the infinite-dimensional Heisenberg Lie algebra is the associative algebra with generators $p_n^\pm$, $n \in \N_+$, and $c$, and relations
\[
  p_n^+ p_m^+ = p_m^+ p_n^+,\quad
  p_n^- p_m^- = p_m^- p_n^-,\quad
  p_n^+ p_m^- = p_m^- p_n^+ + \delta_{n,m} c,\quad
  p_n^\pm c = c p_n^\pm,\quad
  n,m \in \N_+.
\]
(See Section~\ref{subsec:Heisenberg-algebra} for a more detailed treatment.)  On any irreducible representation, the central element $c$ acts by a constant.  For a positive integer $d$, the (associative) Heisenberg algebra $\fh_d$ of level $d$ is the quotient of this algebra by the ideal generated by $c-d$.

In \cite{Kho14}, Khovanov introduced a diagrammatic monoidal category $\cH$ that acts naturally on categories of modules for symmetric groups.  He proved that the Grothendieck ring of $\cH$ contains the level one Heisenberg algebra $\fh_1$ and conjectured that the two are actually equal.  This work has inspired an active area of research into Heisenberg categorification.  Replacing group algebras of symmetric groups by Hecke algebras of type $A$ led to the $q$-deformed categorification of \cite{LS13}, while replacing them by wreath product algebras led to categorifications of (quantum) lattice Heisenberg algebras in \cite{CL12, RS17}.

One remarkable feature of Khovanov's category is that degenerate affine Hecke algebras $H_n$ appear naturally in the endomorphism spaces of certain objects.  The group algebras of symmetric groups are level one cyclotomic quotients of these degenerate affine Hecke algebras, wherein the polynomial generators of $H_n$ are mapped to the Jucys--Murphy elements.  It is thus natural to conjecture that suitably modified versions of Khovanov's category should act on categories of modules for more general degenerate cyclotomic Hecke algebras.  These modified categories should categorify higher level Heisenberg algebras and should encode much of the representation theory of degenerate affine Hecke algebras and their cyclotomic quotients.

On the other hand, cyclotomic Hecke algebras and their degenerate versions have appeared in the context of categorification before.  In particular, Ariki's categorification theorem (\cite{Ari96,BK10}) relates the representation theory of these algebras to highest weight irreducible representations of affine Lie algebras of type $A$.  In addition, Brundan--Kleshchev (\cite{BK09,BK09b}) and Rouquier (\cite{Rou08}) have related cyclotomic Hecke algebras to the quiver Hecke algebras appearing in the Khovanov--Lauda--Rouquier categorification of quantum $\fsl_n$ (\cite{KL09,Rou08}).

In the current paper we define a family of diagrammatic $\kk$-linear monoidal categories $\cH^\lambda$ depending on a dominant integral weight $\lambda$ of $\fsl_\infty$ (when $\kk$ is of characteristic zero) or $\widehat{\fsl}_p$ (when $\kk$ is of characteristic $p > 0$).  When $\lambda$ is a fundamental weight (i.e.\ is of level one), the category $\cH^\lambda$ reduces to Khovanov's Heisenberg category $\cH$ (see Remark~\ref{rem:level1-Khovanov}).  One of our main results is Theorem~\ref{theo:main}, which states that, if $\kk$ is a field of characteristic zero, then we have an injective ring homomorphism
\[
  \fh_d \hookrightarrow K_0(\cH^\lambda),
\]
where $K_0(\cH^\lambda)$ denotes the Grothendieck ring of $\cH^\lambda$.  We conjecture that this homomorphism is also surjective.  In level one, this reduces to Khovanov's conjecture \cite[Conj.~1]{Kho14}.

Our next main result is Theorem~\ref{theo:action-functor}, which states that the category $\cH^\lambda$ acts naturally on categories of modules for the degenerate cyclotomic Hecke algebras $H_n^\lambda$, $n \in \N$, corresponding to $\lambda$.  The action arises from a collection of functors from $\cH^\lambda$ to categories of bimodules over the $H_n^\lambda$.  If $\kk$ is a field of characteristic zero, we prove in Theorem~\ref{theo:Fn-full} that these functors are full.  Thus, the categories $\cH^\lambda$ give a graphical calculus for bimodules over degenerate cyclotomic Hecke algebras.  As an application of the resulting computational tool, we prove in Corollary~\ref{cor:centralizers} that the centralizer of $H_n^\lambda$ in $H_{n+k}^\lambda$ under the natural inclusion of algebras $H_n^\lambda \otimes H_k \hookrightarrow H_{n+k}^\lambda$ is generated by $H_k$ and the center of $H_n^\lambda$.  This result, which seems to be new, is a generalization of a result of Olshanski on centralizers for group algebras of symmetric groups.

The constructions of the current paper suggest a number of applications and further research directions.  For example, we expect truncations of the $\cH^\lambda$ to be related to categorified quantum groups.  There should also exist $q$-deformations of the categories $\cH^\lambda$.  In addition, it would be interesting to examine the traces of the $\cH^\lambda$ and use them to construct diagrammatic pairings on spaces related to symmetric functions.  We expand on these ideas for further research in Section~\ref{sec:further-directions}.

We note that higher level Heisenberg algebras are also categorified by the categories $\cH_B$ of \cite{RS17} for suitable choices of the Frobenius algebra $B$ on which these categories depend.  The approach of the current paper is quite different, but there are relations between the two.  First, there is a natural filtration on the categories $\cH^\lambda$ introduced here.  The associated graded category is closely related to some special cases of the categories of \cite{RS17}.  See Proposition~\ref{prop:associated-graded-HB}.  In addition, we expect that the approaches of \cite{RS17} and the current paper can be unified, as we explain in Section~\ref{subsec:wreath-generalization}.  One aspect of the categories introduced in the current paper that differs substantially from previous Heisenberg categories is the introduction of \emph{dual dots} in planar diagrams.  These dots, which correspond algebraically to the duals of the elements $x_i$ of degenerate cyclotomic Hecke algebras under the natural trace form on those algebras, behave quite differently than the usual dots do.  This makes some computations substantially more difficult, requiring new techniques relying on a filtration on the category.  See Remark~\ref{rem:dualdots}.

The paper is organized as follows.  In Section~\ref{sec:diagrammatics} we introduce our main object of study, the diagrammatic categories $\cH^\lambda$, and study their morphism spaces in detail.  In Section~\ref{sec:cat-filtration} we define a filtration on these morphism spaces.  This filtration is used in Section~\ref{sec:Groth-group} to show that the higher level Heisenberg algebra maps to the Grothendieck ring of $\cH^\lambda$.  In Section~\ref{sec:DCHA} we recall the definition of the degenerate cyclotomic Hecke algebras and prove various results concerning them.  We use these results in Section~\ref{sec:action} to define the action of the category $\cH^\lambda$ on the category of modules for these algebras and use the action to prove injectivity of the map $\fh_d \to K_0(\cH^\lambda)$.  Then, in Section~\ref{sec:action-properties}, we prove some properties of the action.  In particular, we prove that the functors involved are full in the characteristic zero case, allowing us to deduce new facts about centralizers of degenerate cyclotomic Hecke algebras.  Finally, in Section~\ref{sec:further-directions} we discuss some further directions of research naturally suggested by the current paper.

\subsection*{Notation and conventions}

Throughout the paper, $\kk$ will denote a commutative ring.  In some places, we will assume that $\kk$ is a field of characteristic zero.  For the benefit of the reader, we will specify the assumptions on $\kk$ at the beginning of each section.  We let $I$ denote the image of $\Z$ in $\kk$ under the natural ring homomorphism $n \mapsto n \cdot 1$.  We identify $I$ with $\Z/p\Z$, where $p \ge 0$ is the characteristic of $\kk$.  We let $\N$ denote the set of nonnegative integers, and let $\N_+$ denote the set of positive integers.

\subsection*{Acknowledgments}

We would like to thank A.~Licata for helpful conversations and for sharing his preliminary notes on the topic of the current paper.  We would also like to thank J.~Brundan for useful discussions concerning degenerate affine Hecke algebras.  The second author was supported by a Discovery Grant from the Natural Sciences and Engineering Research Council of Canada.

\iftoggle{detailsnote}{
  \subsection*{Hidden details} For the interested reader, the tex file of the \href{https://arxiv.org/abs/1705.03066}{arXiv version} of this paper includes hidden details of some straightforward computations and arguments that are omitted in the pdf file.  These details can be displayed by switching the \texttt{details} toggle to true in the tex file and recompiling.
}{}

%
\section{The diagrammatic category} \label{sec:diagrammatics}
%

Throughout this section, $\kk$ is an arbitrary commutative ring.  Let $\fg = \widehat{\fsl}_p$ if the characteristic $p$ of $\kk$ is greater than zero, and let $\fg = \fsl_\infty$ if $p=0$.  Recall that $I \cong \Z/p\Z$ is the canonical image of $\Z$ in $\kk$.  For $i \in I$, let $\omega_i$ denote the $i$-th fundamental weight of $\fg$, and let $P_+ = \bigoplus_{i \in I} \N \omega_i$ denote the dominant weight lattice of $\fg$.  Fix $\lambda = \sum_{i \in I} \lambda_i \omega_i \in P_+$, $\lambda \ne 0$, and set $d = \sum_i \lambda_i$.

\subsection{Definition}

Let $\tcH^\lambda$ be the additive $\kk$-linear strict monoidal category defined as follows.  The set of objects is generated by two objects $\sQ_+$ and $\sQ_-$. Thus each object is a formal direct sum of objects of the form $\sQ_{\epsilon} := \sQ_{\epsilon_1} \dotsb \sQ_{\epsilon_k}$, where $\epsilon = (\epsilon_1,\dotsc,\epsilon_k)$ is a sequence whose entries are either $+$ or $-$ and we denote the tensor product by juxtaposition.  We denote the unit object by $\one :=\sQ_{\varnothing}$.

The space of morphisms between two objects $Q_\epsilon$ and $Q_{\epsilon'}$ is the $\kk$-algebra generated by suitable planar diagrams modulo local relations.  The diagrams consist of oriented compact one-manifolds immersed into the plane strip $\R \times [0,1]$, modulo isotopies fixing the boundary, and modulo certain local relations.  Strands are allowed to carry dots.  Only double intersections are allowed and the dots do not lie on them. The dots are allowed to move freely along the strands of the diagram as long as they do not cross double points.  The endpoints of the diagrams are considered to be located at $\{1,\dotsc,m\}\times\{0\}$ and $\{1,\dotsc,k\}\times \{1\}$, where $m$ and $k$ are the lengths of $\epsilon$ and $\epsilon'$ respectively.  (This is for the purposes of composition---we will not always draw our diagrams with these positions apparent.) The orientation of the one-manifold at the endpoints must agree with the signs in $\epsilon$ and $\epsilon'$. For example
\[
  \begin{tikzpicture}[anchorbase]
    \draw[->] (0,0) .. controls (0,1) and (1,1) .. (1,2);
    \draw[->] (1,0) .. controls (1,1) and (3,1) .. (3,0);
    \draw[->] (2,0) .. controls (2,1.5) and (-1,1.5) .. (-1,0);
    \draw[->] (2,2) arc(180:360:1);
    \bluedot{(0.915,1.5)};
    \bluedot{(-.3,0.97)};
    \draw (7,1) .. controls (7,1.5) and (6.3,1.5) .. (6.1,1);
    \draw (7,1) .. controls (7,.5) and (6.3,.5) .. (6.1,1);
    \draw (6,0) .. controls (6,.5) .. (6.1,1);
    \draw (6.1,1) .. controls (6,1.5) .. (6,2) [->];
    \bluedot{(6,0.3)};
    \draw[->] (4.3,1) arc(180:-180:.6);
    \bluedot{(3.88,1.5)};
    \bluedot{(2.5,1.14)};
  \end{tikzpicture}
\]
is an element of $\Hom_{\tcH^\lambda}(\sQ_{-+++-+},\sQ_{+-++})$.

Composition of morphisms is given by vertical stacking of diagrams (followed by rescaling of the vertical axis).  The monoidal structure on morphisms is given by horizontal juxtaposition of diagrams.

An endomorphism of $\one$ is a diagram with no endpoints.  A dot labeled $j \in \N$ will denote a strand with $j$ dots:
\[
  \begin{tikzpicture}[anchorbase]
    \draw[->] (0,0) to (0,2);
    \bluedot{(0,1)} node[anchor=east,color=black] {\dotlabel{j}};
  \end{tikzpicture}
  \ =\
  \begin{tikzpicture}[anchorbase]
    \draw[->] (1.5,0) to (1.5,2);
    \bluedot{(1.5,.5)};
    \bluedot{(1.5,.8)};
    \bluedot{(1.5,1.1)};
    \bluedot{(1.5,1.4)};
    \draw (2.5,1) node {$\Bigg\}\ j$ dots};
  \end{tikzpicture}
\]

Define $c_0=1$ and, for any $1\leq s\leq d$, let
\begin{equation} \label{eq:c-def}
  c_s := (-1)^s \det \left(
  \begin{tikzpicture}[anchorbase]
    \draw [->](0,0) arc (0:360:0.3);
    \bluedot{(-0.6,0)} node [anchor=east,color=black] {\dotlabel{d+j-i}};
  \end{tikzpicture}
  \right)_{i,j = 1,\dotsc,s}.
\end{equation}
Then, for $j \in \{0,\dotsc,d-1\}$, define the \emph{dual dots}
\begin{equation} \label{eq:dot-vee-def}
  \begin{tikzpicture}[anchorbase]
    \draw[->] (0,0) -- (0,1);
    \bluedot{(0,0.5)} node [anchor=east,color=black] {\dotlabel{j^\vee}};
  \end{tikzpicture}
  \ = \sum_{i=j}^{d-1}
  \begin{tikzpicture}[anchorbase]
    \draw[->] (0,0) -- (0,1);
    \bluedot{(0,0.5)} node [anchor=east,color=black] {\dotlabel{i-j}};
  \end{tikzpicture}
  \ \ c_{d-1-i}.
\end{equation}
(The reason for this definition of the dual dots will become apparent in Lemma~\ref{lem:circle-duality} below.)

The local relations are as follows:

\noindent\begin{minipage}{0.5\linewidth}
  \begin{equation} \label{rel:braid}
    \begin{tikzpicture}[anchorbase]
      \draw[->] (0,0) -- (1,1);
      \draw[->] (1,0) -- (0,1);
      \draw[->] (0.5,0) .. controls (0,0.5) .. (0.5,1);
    \end{tikzpicture}
    \ =\
    \begin{tikzpicture}[anchorbase]
      \draw[->] (0,0) -- (1,1);
      \draw[->] (1,0) -- (0,1);
      \draw[->] (0.5,0) .. controls (1,0.5) .. (0.5,1);
    \end{tikzpicture}
  \end{equation}
\end{minipage}%
\begin{minipage}{0.5\linewidth}
  \begin{equation} \label{rel:s-squared}
    \begin{tikzpicture}[anchorbase]
      \draw[->] (0,0) .. controls (0.5,0.5) .. (0,1);
      \draw[->] (0.5,0) .. controls (0,0.5) .. (0.5,1);
    \end{tikzpicture}
    \ =\
    \begin{tikzpicture}[anchorbase]
      \draw[->] (0,0) --(0,1);
      \draw[->] (0.5,0) -- (0.5,1);
    \end{tikzpicture}
  \end{equation}
\end{minipage}\par\vspace{\belowdisplayskip}

\begin{equation}  \label{rel:dotslide}
  \begin{tikzpicture}[anchorbase]
    \draw[->] (0,0) to (1,1);
    \draw[->] (1,0) to (0,1);
    \bluedot{(0.25,0.25)};
  \end{tikzpicture}
  \ -\
  \begin{tikzpicture}[anchorbase]
    \draw[->] (0,0) to (1,1);
    \draw[->] (1,0) to (0,1);
    \bluedot{(.75,.75)};
  \end{tikzpicture}
  \ =\
  \begin{tikzpicture}[anchorbase]
    \draw[->] (0,0) to (0,1);
    \draw[->] (0.5,0) to (0.5,1);
  \end{tikzpicture}
  \ =\
  \begin{tikzpicture}[anchorbase]
    \draw[->] (0,0) to (1,1);
    \draw[->] (1,0) to (0,1);
    \bluedot{(0.25,0.75)};
  \end{tikzpicture}
  \ -\
  \begin{tikzpicture}[anchorbase]
    \draw[->] (0,0) to (1,1);
    \draw[->] (1,0) to (0,1);
    \bluedot{(.75,.25)};
  \end{tikzpicture}
\end{equation}

\noindent\begin{minipage}{0.5\linewidth}
  \begin{equation} \label{rel:down-up-double-crossing}
    \begin{tikzpicture}[anchorbase]
      \draw[<-] (0,0) --(0,1);
      \draw[->] (0.5,0) -- (0.5,1);
    \end{tikzpicture}
    \ =\
    \begin{tikzpicture}[anchorbase]
      \draw[<-] (0,0) .. controls (0.5,0.5) .. (0,1);
      \draw[->] (0.5,0) .. controls (0,0.5) .. (0.5,1);
    \end{tikzpicture}
    \ + \sum_{j=0}^{d-1}\
    \begin{tikzpicture}[anchorbase]
      \draw[->] (0,1) -- (0,0.9) arc (180:360:.25) -- (0.5,1);
      \draw[<-] (0,0) -- (0,0.1) arc (180:0:.25) -- (0.5,0);
      \bluedot{(0.45,0.75)} node [anchor=west,color=black] {\dotlabel{j}};
      \bluedot{(0.05,0.25)} node [anchor=east,color=black] {\dotlabel{j^\vee}};
    \end{tikzpicture}
  \end{equation}
\end{minipage}%
\begin{minipage}{0.5\linewidth}
  \begin{equation} \label{rel:up-down-double-crossing}
    \begin{tikzpicture}[anchorbase]
      \draw[->] (0,0) .. controls (0.5,0.5) .. (0,1);
      \draw[<-] (0.5,0) .. controls (0,0.5) .. (0.5,1);
    \end{tikzpicture}
    \ =\
    \begin{tikzpicture}[anchorbase]
      \draw[->] (0,0) --(0,1);
      \draw[<-] (0.5,0) -- (0.5,1);
    \end{tikzpicture}
  \end{equation}
\end{minipage}\par\vspace{\belowdisplayskip}

\noindent\begin{minipage}{0.5\linewidth}
  \begin{equation} \label{rel:cc-bubble}
    \begin{tikzpicture}[anchorbase]
      \draw [->](0,0) arc (0:360:0.3);
      \bluedot{(-0.6,0)}  node [anchor=east,color=black] {\dotlabel{j}};
    \end{tikzpicture}
    \ = \
    \begin{cases}
      0 & \text{if } j < d-1, \\
      1 & \text{if } j=d-1, \\
      \sum_{i \in I} i \lambda_i & \text{if } j=d.
    \end{cases}
  \end{equation}
\end{minipage}%
\begin{minipage}{0.5\linewidth}
  \begin{equation} \label{rel:left-curl-zero}
    \begin{tikzpicture}[scale=0.5,anchorbase]
      \draw (0,0) .. controls (0,.5) and (.7,.5) .. (.9,0);
      \draw (0,0) .. controls (0,-.5) and (.7,-.5) .. (.9,0);
      \draw (1,-1) .. controls (1,-.5) .. (.9,0);
      \draw[->] (.9,0) .. controls (1,.5) .. (1,1);
    \end{tikzpicture}
    \ = 0
  \end{equation}
\end{minipage}\par\vspace{\belowdisplayskip}

\noindent We allow the dots to move freely along strands, in particular over cups and caps.  (Note that the second equality of \eqref{rel:dotslide} follows from the first, together with \eqref{rel:s-squared}.  However, we include it for ease of reference.)

\begin{rem} \label{rem:dualdots}
  Note that the dual dots do \emph{not} satisfy relation \eqref{rel:dotslide} in general.  In fact, the formulas for sliding a dual dot through a crossing are rather complicated.  This introduces a computational difficulty into the graphical calculus.  We will handle this difficulty by using a natural filtration on the morphism spaces of $\tcH^\lambda$ and performing some computations in the associated graded category.  See Section~\ref{sec:cat-filtration}.
\end{rem}

We let $\cH^\lambda$ be the idempotent completion (sometimes called the Karoubi envelope) of $\tcH^\lambda$.  Thus, the objects of $\cH^\lambda$ are direct sums of pairs $(\sQ_\epsilon, e)$, where $e \colon \sQ_\epsilon \to \sQ_\epsilon$ is an idempotent morphism.  Morphisms $(\sQ_\epsilon, e) \to (\sQ_{\epsilon'}, e')$ in $\cH^\lambda$ are morphisms $\alpha \colon \sQ_\epsilon \to \sQ_{\epsilon'}$ in $\tcH^\lambda$ such that $fe=f=e'f$.  Equivalently,
\[
  \Hom_{\cH^\lambda}((\sQ_\epsilon,e),(\sQ_{\epsilon'},e'))
  = e' \Hom_{\tcH^\lambda} (\sQ_\epsilon,\sQ_{\epsilon'}) e.
\]
The composition and monoidal structure of $\cH^\lambda$ are induced by those of $\tcH^\lambda$.  Identifying $\sQ_\epsilon$ with $(\sQ_\epsilon,\id_{\sQ_\epsilon})$, we view $\tcH^\lambda$ as a full subcategory of $\cH^\lambda$.

\subsection{Morphism spaces}

We now investigate the morphism spaces of $\cH^\lambda$ in some detail.

\begin{lem} \label{lem:higherNilHecke}
 We have
  \begin{equation}  \label{eq:mult_dotslide}
    \sum_{a+b=t-1}
    \begin{tikzpicture}[anchorbase]
      \draw[->] (0,0) to (0,1);
      \draw[->] (0.5,0) to (0.5,1);
      \bluedot{(0,0.5)} node[color=black,anchor=east] {\dotlabel{a}};
      \bluedot{(0.5,0.5)} node[color=black,anchor=west] {\dotlabel{b}};
    \end{tikzpicture}
    \ =\
    \begin{tikzpicture}[anchorbase]
      \draw[->] (0,0) to (1,1);
      \draw[->] (1,0) to (0,1);
      \bluedot{(0.25,0.25)} node[color=black,anchor=south east] {\dotlabel{t}};
    \end{tikzpicture}
    \ -\
    \begin{tikzpicture}[anchorbase]
      \draw[->] (0,0) to (1,1);
      \draw[->] (1,0) to (0,1);
      \bluedot{(.75,.75)} node[color=black,anchor=south east] {\dotlabel{t}};;
    \end{tikzpicture}
    \ =\
    \begin{tikzpicture}[anchorbase]
      \draw[->] (0,0) to (1,1);
      \draw[->] (1,0) to (0,1);
      \bluedot{(0.25,0.75)} node[color=black,anchor=south west] {\dotlabel{t}};;
    \end{tikzpicture}
    \ -\
    \begin{tikzpicture}[anchorbase]
      \draw[->] (0,0) to (1,1);
      \draw[->] (1,0) to (0,1);
      \bluedot{(.75,.25)} node[color=black,anchor=south west] {\dotlabel{t}};;
    \end{tikzpicture}\ .
  \end{equation}
\end{lem}

In \eqref{eq:mult_dotslide}, and throughout the paper, the notation $\sum_{a+b=t-1}$ means that we sum over all $a,b \in \N$ satisfying the condition $a+b=t-1$.

\begin{proof}
 The equations in \eqref{eq:mult_dotslide} are obtained by applying the ones in~\eqref{rel:dotslide} $t$ times.
\end{proof}

\begin{lem} \label{lem:det-expansion}
  Suppose $a_k$, $k \in \Z$, are elements of a commutative ring such that
  \[
    a_{-1} = 1,
    \quad \text{and} \quad
    a_k = 0 \quad \text{for } k < -1.
  \]
  Then, for $m \ge 1$, we have
  \[
    \sum_{s=0}^m (-1)^s a_{m-1-s} \det (a_{j-i})_{i,j=1,\dotsc,s} = 0.
  \]
\end{lem}

\begin{proof}
  If we compute the determinant $\det (a_{j-i})_{i,j=1,\dotsc,m}$ by repeatedly expanding along the first column, we get
  \begin{align*}
    \det (a_{j-i})_{i,j=1,\dotsc,m}
    &= a_0 \det (a_{j-i})_{i,j=1,\dotsc,m-1} - \det
    \begin{pmatrix}
      a_1 & a_2 & \cdots & a_{m-1} \\
      1 & a_0 & \cdots & a_{m-3} \\
      \vdots & \vdots & \ddots & \vdots \\
      0 & 0 & \cdots & a_0
    \end{pmatrix}
    \\
    &= a_0 \det (a_{j-i})_{i,j=1,\dotsc,m-1} - a_1 \det (a_{j-i})_{i,j=1,\dotsc,m-2} + \dotsb \\
    &= \sum_{s=0}^{m-1} (-1)^{m-1-s} a_{m-1-s} \det (a_{j-i})_{i,j=1,\dotsc,s}.
  \end{align*}
  The result follows.
\end{proof}

\begin{lem} \label{lem:circle-det-expand}
  For $m \in \N$, we have
  \begin{equation} \label{eq:det-expand}
    \sum_{t=0}^m c_t\
    \begin{tikzpicture}[anchorbase]
      \draw [->](0,0) arc (180:360:0.5);
      \draw (1,0) arc (0:180:0.5);
      \bluedot{(0.8,-0.4)} node [anchor=north,color=black] {\dotlabel{d-1+m-t}};
    \end{tikzpicture}
    \ = \delta_{m,0}.
  \end{equation}
\end{lem}

\begin{proof}
  The case $m=0$ follows immediately from \eqref{rel:cc-bubble}.  The case $m > 0$ follows from Lemma~\ref{lem:det-expansion} with
  \[
    a_k =\
    \begin{tikzpicture}[anchorbase]
      \draw [->](0,0) arc (180:360:0.5);
      \draw (1,0) arc (0:180:0.5);
      \bluedot{(0.8,-0.4)} node [anchor=north,color=black] {\dotlabel{d+k}};
    \end{tikzpicture}\ . \qedhere
  \]
\end{proof}

\begin{lem} \label{lem:circle-duality}
  For any $i,j \in \{0,\dotsc,d-1\}$, we have
  \[
    \begin{tikzpicture}[anchorbase]
      \draw [->](0,0) arc (180:360:0.5);
      \draw (1,0) arc (0:180:0.5);
      \bluedot{(0.8,0.4)} node [anchor=west,color=black] {\dotlabel{i}};
      \bluedot{(0.8,-0.4)} node [anchor=north,color=black] {\dotlabel{j^\vee}};
    \end{tikzpicture}
    \ = \delta_{i,j}.
  \]
\end{lem}

\begin{proof}
  We have
  \[
    \begin{tikzpicture}[anchorbase]
      \draw [->](0,0) arc (180:360:0.5);
      \draw (1,0) arc (0:180:0.5);
      \bluedot{(0.8,0.4)} node [anchor=west,color=black] {\dotlabel{i}};
      \bluedot{(0.8,-0.4)} node [anchor=north,color=black] {\dotlabel{j^\vee}};
    \end{tikzpicture}
    \ = \sum_{k=j}^{d-1} c_{d-1-k}\
    \begin{tikzpicture}[anchorbase]
      \draw [->](0,0) arc (180:360:0.5);
      \draw (1,0) arc (0:180:0.5);
      \bluedot{(0.8,-0.4)} node [anchor=north,color=black] {\dotlabel{k+i-j}};
    \end{tikzpicture}
  \]
  If $i \le j$, the result follows immediately from \eqref{rel:cc-bubble}.  Therefore, assume $i > j$.  Setting $m=i-j$ and $t=d-1-k$ in the above sum, it remains to prove that
  \begin{equation} \label{eq:circle-duality-sum}
    \sum_{t=0}^{d-1-j} c_t\
    \begin{tikzpicture}[anchorbase]
      \draw [->](0,0) arc (180:360:0.5);
      \draw (1,0) arc (0:180:0.5);
      \bluedot{(0.8,-0.4)} node [anchor=north,color=black] {\dotlabel{d-1+m-t}};
    \end{tikzpicture}
    \ = 0.
  \end{equation}
  By \eqref{rel:cc-bubble}, the terms in \eqref{eq:circle-duality-sum} with $t > m$ are equal to zero.  Since $i \le d-1$, we have $m \le d-1-j$.  Therefore, the result follows from Lemma~\ref{lem:circle-det-expand}.
\end{proof}

\begin{lem}\label{lem:idempotent}
  The equality \eqref{rel:down-up-double-crossing} is a decomposition of the identity morphism of $\sQ_- \sQ_+$ into $d+1$ orthogonal idempotents.  More precisely, the morphisms
  \[
    \begin{tikzpicture}[anchorbase]
      \draw[<-] (0,0) .. controls (0.5,0.5) .. (0,1);
      \draw[->] (0.5,0) .. controls (0,0.5) .. (0.5,1);
    \end{tikzpicture}
    \quad \text{and} \quad
    \begin{tikzpicture}[anchorbase]
      \draw[->] (0,1) -- (0,0.9) arc (180:360:.25) -- (0.5,1);
      \draw[<-] (0,0) -- (0,0.1) arc (180:0:.25) -- (0.5,0);
      \bluedot{(0.45,0.75)} node [anchor=west,color=black] {\dotlabel{j}};
      \bluedot{(0.05,0.25)} node [anchor=east,color=black] {\dotlabel{j^\vee}};
    \end{tikzpicture}
    \ ,\quad j=0,\dotsc,d-1,
  \]
  are orthogonal idempotents.
\end{lem}

\begin{proof}
  By \eqref{rel:up-down-double-crossing},
  \begin{equation} \label{eq-lem-proof:up-down-double-crossing}
    \begin{tikzpicture}[anchorbase]
      \draw[<-] (0,0) .. controls (0.5,0.5) .. (0,1);
      \draw[->] (0.5,0) .. controls (0,0.5) .. (0.5,1);
    \end{tikzpicture}
  \end{equation}
  is an idempotent.  It also follows immediately from Lemma~\ref{lem:circle-duality} that the
  \begin{equation} \label{eq-lem-proof:cup-cap-idempotents}
    \begin{tikzpicture}[anchorbase]
      \draw[->] (0,1) -- (0,0.9) arc (180:360:.25) -- (0.5,1);
      \draw[<-] (0,0) -- (0,0.1) arc (180:0:.25) -- (0.5,0);
      \bluedot{(0.45,0.75)} node [anchor=west,color=black] {\dotlabel{j}};
      \bluedot{(0.05,0.25)} node [anchor=east,color=black] {\dotlabel{j^\vee}};
    \end{tikzpicture}
   \ ,\quad j=0,\dotsc,d-1,
  \end{equation}
  are orthogonal idempotents.  Finally, since, by \eqref{rel:down-up-double-crossing}, we have that \eqref{eq-lem-proof:up-down-double-crossing} is equal to the identity minus the sum of the elements in \eqref{eq-lem-proof:cup-cap-idempotents}, it follows that \eqref{eq-lem-proof:up-down-double-crossing} is orthogonal to each of the idempotents in \eqref{eq-lem-proof:cup-cap-idempotents}.
\end{proof}

\begin{cor} \label{cor:Q+Q-commutation-relation}
  In $\tcH^\lambda$, we have an isomorphism
  \[
    \left[
      \begin{tikzpicture}[anchorbase]
        \draw[<-] (0,0) -- (0.5,0.5);
        \draw[->] (0.5,0) -- (0,0.5);
      \end{tikzpicture}
      \
      \begin{tikzpicture}[anchorbase]
        \draw[<-] (0,0) -- (0,0.2) arc (180:0:.25) -- (0.5,0);
        \bluedot{(0.01,0.25)} node [anchor=east,color=black] {\dotlabel{j^\vee}};
      \end{tikzpicture}
      \ ,\ j=0,\dotsc,d-1
    \right]^T
    \colon
    \sQ_- \sQ_+ \cong \sQ_+ \sQ_- \oplus \one^{\oplus d}.
  \]
\end{cor}

\begin{proof}
  This follows immediately from Lemma~\ref{lem:idempotent}.
\end{proof}

\begin{cor} \label{cor:tcH-basic-objects}
  Every object of $\tcH^\lambda$ is isomorphic to a direct sum of objects of the form $\sQ_+^n \sQ_-^m$, $n,m \in \N$.
\end{cor}

\begin{proof}
  By definition, objects of $\tcH^\lambda$ are sums of products of $\sQ_+$ and $\sQ_-$.  The result then follows from Corollary~\ref{cor:Q+Q-commutation-relation} by induction.
\end{proof}

We now deduce some other consequences of the local relations in $\cH^\lambda$.

\begin{lem} \label{lem:easy_cons}
  We have
  \begin{equation} \label{eq:dotted_leftcurl}
    \begin{tikzpicture}[anchorbase]
      \draw (0,0) .. controls (0,.5) and (.7,.5) .. (.9,0);
      \draw (0,0) .. controls (0,-.5) and (.7,-.5) .. (.9,0);
      \draw (1,-1) .. controls (1,-.5) .. (.9,0);
      \draw[->] (.9,0) .. controls (1,.5) .. (1,1);
      \bluedot{(0,0)} node[color=black,anchor=east] {\dotlabel{t}};
    \end{tikzpicture}
    \ =\
    \begin{cases}
      0 & \text{if } t < d, \\
      \ \begin{tikzpicture}[anchorbase]
        \draw[->] (0,0) to (0,0.5);
      \end{tikzpicture}
      & \text{if } t=d, \\
      \sum_{a=0}^{t-d}
      \begin{tikzpicture}[anchorbase]
        \draw [->](0,0) arc (180:360:0.5);
        \draw (1,0) arc (0:180:0.5);
        \bluedot{(0.8,-0.4)} node [anchor=north,color=black] {\dotlabel{d-1+a}};
        \draw[->] (1.5,-0.5) to (1.5,0.5);
        \bluedot{(1.5,0)} node[color=black,anchor=west] {\dotlabel{t-d-a}};
      \end{tikzpicture}
      & \text{if } t>d,
    \end{cases}
  \end{equation}

  \begin{equation} \label{eq:dots-vs-right-curls}
    \begin{tikzpicture}[anchorbase]
      \draw[->] (0,0) to (0,1);
      \bluedot{(0,0.5)} node[color=black,anchor=east] {\dotlabel{d}};
    \end{tikzpicture}
    \ =\
    \begin{tikzpicture}[anchorbase]
      \draw (0.5,0.5) .. controls (0.5,0.75) and (0.15,0.75) .. (0.05,0.5);
      \draw (0.5,0.5) .. controls (0.5,.25) and (0.15,.25) .. (0.05,0.5);
      \draw (0,0) .. controls (0,.25) .. (0.05,0.5);
      \draw[->] (0.05,0.5) .. controls (0,0.75) .. (0,1);
    \end{tikzpicture}
    \ -\
    \sum_{j=0}^{d-1}
    \begin{tikzpicture}[anchorbase]
      \draw[->] (0,0) to (0,1);
      \bluedot{(0,0.5)} node[color=black,anchor=east] {\dotlabel{j}};
    \end{tikzpicture}
    \ \ c_{d-j}.
  \end{equation}
\end{lem}

\begin{proof}
  By \eqref{eq:mult_dotslide} and \eqref{rel:left-curl-zero}, we have
  \[
    \begin{tikzpicture}[anchorbase]
      \draw (0,0) .. controls (0,.5) and (.7,.5) .. (.9,0);
      \draw (0,0) .. controls (0,-.5) and (.7,-.5) .. (.9,0);
      \draw (1,-1) .. controls (1,-.5) .. (.9,0);
      \draw[->] (.9,0) .. controls (1,.5) .. (1,1);
      \bluedot{(0,0)} node[color=black,anchor=east] {\dotlabel{t}};
    \end{tikzpicture}
    \ = \sum_{b=0}^{t-1}\
    \begin{tikzpicture}[anchorbase]
      \draw [->](0,0) arc (180:360:0.5);
      \draw (1,0) arc (0:180:0.5);
      \bluedot{(0.8,-0.4)} node [anchor=north,color=black] {\dotlabel{b}};
      \draw[->] (1.5,-0.5) to (1.5,0.5);
      \bluedot{(1.5,0)} node[color=black,anchor=west] {\dotlabel{t-1-b}};
    \end{tikzpicture}\ .
  \]
  Then \eqref{eq:dotted_leftcurl} follows from \eqref{rel:cc-bubble}.

  Now, to all terms in \eqref{rel:down-up-double-crossing}, add an upward pointing strand to the left, then a right cap joining the tops of the two leftmost strands, and a right cup with $d$ dots joining the bottoms of the two rightmost strands.  This gives
  \begin{multline*}
    \begin{tikzpicture}[anchorbase]
      \draw[->] (0,0) to (0,1.1) arc(180:0:.3) to (0.6,0.9) arc(180:360:.3) to (1.2,2);
      \bluedot{(1.2,1)} node[color=black,anchor=east] {\dotlabel{d}};
    \end{tikzpicture}
    \ =\
    \begin{tikzpicture}[anchorbase]
      \draw[->] (-0.5,-0.5) .. controls (-0.5,1) and (-0.4,2) .. (0,2) .. controls (0.4,2) and (0.6,1.3) .. (0.6,1) .. controls (0.6,0.5) and (0,.3) .. (0,0) .. controls (0,-0.4)  and (1,-0.4) .. (1,0) .. controls (1,.3) and (0,0.5) .. (0,1) .. controls (0,1.3) and (1,1.6) .. (1,2);
      \bluedot{(1,0)} node[color=black,anchor=east] {\dotlabel{d}};
    \end{tikzpicture}
    \ + \sum_{j=0}^{d-1} \
    \begin{tikzpicture}[anchorbase]
      \draw[->] (0,0) .. controls (0,1) and (0,2) .. (0.3,2) .. controls (0.6,2) and (0.5,1.2) .. (1,1.2) .. controls (1.5,1.2) and (1.5,1.6) .. (1.5,2);
      \bluedot{(1.48,1.6)} node[color=black,anchor=west] {\dotlabel{j}};
      \draw[->] (1.5,0.5) arc(0:360:0.5);
      \bluedot{(1.4,0.17)} node[color=black,anchor=west] {\dotlabel{j^\vee}};
      \bluedot{(1.4,0.83)} node[color=black,anchor=west] {\dotlabel{d}};
    \end{tikzpicture}
    \stackrel{\eqref{eq:dotted_leftcurl}}{=}\
    \begin{tikzpicture}[anchorbase]
      \draw (2,1) .. controls (2,1.5) and (1.3,1.5) .. (1.1,1);
      \draw (2,1) .. controls (2,.5) and (1.3,.5) .. (1.1,1);
      \draw (1,0) .. controls (1,.5) .. (1.1,1);
      \draw[->] (1.1,1) .. controls (1,1.5) .. (1,2);
    \end{tikzpicture}
    \ + \sum_{j=0}^{d-1}
    \begin{tikzpicture}[anchorbase]
      \draw[->] (0,0) -- (0,2);
      \bluedot{(0,1)} node[color=black,anchor=east] {\dotlabel{j}};
    \end{tikzpicture}
    \left( \sum_{i=j}^{d-1} c_{d-1-i}\
      \begin{tikzpicture}[anchorbase]
        \draw [->](0,0) arc (180:360:0.5);
        \draw (1,0) arc (0:180:0.5);
        \bluedot{(0.8,-0.4)} node [anchor=north,color=black] {\dotlabel{d+i-j}};
      \end{tikzpicture}
    \right)
    \\
    =\
    \begin{tikzpicture}[anchorbase]
      \draw (2,1) .. controls (2,1.5) and (1.3,1.5) .. (1.1,1);
      \draw (2,1) .. controls (2,.5) and (1.3,.5) .. (1.1,1);
      \draw (1,0) .. controls (1,.5) .. (1.1,1);
      \draw[->] (1.1,1) .. controls (1,1.5) .. (1,2);
    \end{tikzpicture}
    \ + \sum_{j=0}^{d-1}
    \begin{tikzpicture}[anchorbase]
      \draw[->] (0,0) -- (0,2);
      \bluedot{(0,1)} node[color=black,anchor=east] {\dotlabel{j}};
    \end{tikzpicture}
    \left( \sum_{t=0}^{d-1-j} c_t\
      \begin{tikzpicture}[anchorbase]
        \draw [->](0,0) arc (180:360:0.5);
        \draw (1,0) arc (0:180:0.5);
        \bluedot{(0.8,-0.4)} node [anchor=north,color=black] {\dotlabel{2d-1-j-t}};
      \end{tikzpicture}
    \right)
    \ =\
    \begin{tikzpicture}[anchorbase]
      \draw (2,1) .. controls (2,1.5) and (1.3,1.5) .. (1.1,1);
      \draw (2,1) .. controls (2,.5) and (1.3,.5) .. (1.1,1);
      \draw (1,0) .. controls (1,.5) .. (1.1,1);
      \draw[->] (1.1,1) .. controls (1,1.5) .. (1,2);
    \end{tikzpicture}
    \ -\
    \sum_{j=0}^{d-1}
    \begin{tikzpicture}[anchorbase]
      \draw[->] (0,0) to (0,2);
      \bluedot{(0,1)} node[color=black,anchor=east] {\dotlabel{j}};
    \end{tikzpicture}
    \ \ c_{d-j},
  \end{multline*}
  where, in the third equality, we have made the substitution $t=d-1-i$, and the last equality follows from Lemma~\ref{lem:circle-det-expand} with $m=d-j$.
\end{proof}

\begin{rem} \label{rem:level1-Khovanov}
  If $\lambda=\omega_i$ is a fundamental weight, \eqref{eq:dots-vs-right-curls} implies that
  \[
    \begin{tikzpicture}[anchorbase]
      \draw[->] (0,0) to (0,1);
      \bluedot{(0,0.5)};
    \end{tikzpicture}
    \ =\
    \begin{tikzpicture}[anchorbase]
      \draw (0.5,0.5) .. controls (0.5,0.75) and (0.15,0.75) .. (0.05,0.5);
      \draw (0.5,0.5) .. controls (0.5,.25) and (0.15,.25) .. (0.05,0.5);
      \draw (0,0) .. controls (0,.25) .. (0.05,0.5);
      \draw[->] (0.05,0.5) .. controls (0,0.75) .. (0,1);
    \end{tikzpicture}
    \ -\
    \begin{tikzpicture}[anchorbase]
      \draw[->] (0,0) to (0,1);
    \end{tikzpicture}
    \ c_1
    =\
    \begin{tikzpicture}[anchorbase]
      \draw (0.5,0.5) .. controls (0.5,0.75) and (0.15,0.75) .. (0.05,0.5);
      \draw (0.5,0.5) .. controls (0.5,.25) and (0.15,.25) .. (0.05,0.5);
      \draw (0,0) .. controls (0,.25) .. (0.05,0.5);
      \draw[->] (0.05,0.5) .. controls (0,0.75) .. (0,1);
    \end{tikzpicture}
    \ +\
    \begin{tikzpicture}[anchorbase]
      \draw[->] (0,0) to (0,1);
    \end{tikzpicture}
    \
    \begin{tikzpicture}[anchorbase]
      \draw [->](0,0) arc (0:360:0.3);
      \bluedot{(-0.6,0)};
    \end{tikzpicture}
    =\
    \begin{tikzpicture}[anchorbase]
      \draw (0.5,0.5) .. controls (0.5,0.75) and (0.15,0.75) .. (0.05,0.5);
      \draw (0.5,0.5) .. controls (0.5,.25) and (0.15,.25) .. (0.05,0.5);
      \draw (0,0) .. controls (0,.25) .. (0.05,0.5);
      \draw[->] (0.05,0.5) .. controls (0,0.75) .. (0,1);
    \end{tikzpicture}
    \ + i\
    \begin{tikzpicture}[anchorbase]
      \draw[->] (0,0) to (0,1);
    \end{tikzpicture}
    \ .
  \]
  Thus, all dots can be eliminated.  In particular, the category $\tcH^{\omega_0}$ reduces to the category $\cH'$ defined by Khovanov in \cite{Kho14}.  By Proposition~\ref{prop:dot-shift-functor} below, it then follows that $\tcH^{\omega_i}$ is isomorphic to Khovanov's category for all $i \in I$.
\end{rem}

In the sequel we will use the term \emph{bubble} for clockwise or counterclockwise circles with or without dots.

\begin{lem}[Bubble reduction] \label{lem:bubble-conversion}
  Any clockwise bubble is equal to a linear combination of counterclockwise bubbles.
\end{lem}

\begin{proof}
  For any $t \in \N_+$, we have
  \begin{multline*}
    \sum _{a+b=t-1}
    \begin{tikzpicture}[anchorbase]
      \draw [->](0,0) arc (180:360:0.5);
      \draw (1,0) arc (0:180:0.5);
      \bluedot{(0.8,-0.4)} node[color=black,anchor=west] {\dotlabel{a}};
      \draw [<-](1.5,0) arc (180:360:0.5);
      \draw (2.5,0) arc (0:180:0.5);
      \bluedot{(2.3,-0.4)} node[color=black,anchor=west] {\dotlabel{b}};
    \end{tikzpicture}
    \ \stackrel{\eqref{eq:mult_dotslide}}{=}\
    \begin{tikzpicture}[anchorbase]
      \draw[-<] (1,1) .. controls (1,2) and (1.5,2) .. (2,1) .. controls (2.5,0) and (3,0) ..
      (3,1);
      \draw (3,1) .. controls (3,2) and (2.5,2) .. (2,1) .. controls (1.5,0) and (1,0) .. (1,1);
      \bluedot{(1.02,1.3)} node[color=black,anchor=east] {\dotlabel{t}};
    \end{tikzpicture}
    \ -\
    \begin{tikzpicture}[anchorbase]
      \draw[-<] (4,1) .. controls (4,2) and (4.5,2) .. (5,1) .. controls (5.5,0) and (6,0) ..
      (6,1);
      \draw (6,1) .. controls (6,2) and (5.5,2) .. (5,1) .. controls (4.5,0) and (4,0) .. (4,1);
      \bluedot{(5.18,1.3)} node[color=black, anchor=south] {\dotlabel{t}};
    \end{tikzpicture}
    \ \stackrel{\eqref{rel:left-curl-zero}}{=}\
    \begin{tikzpicture}[anchorbase]
      \draw[-<] (1,1) .. controls (1,2) and (1.5,2) .. (2,1) .. controls (2.5,0) and (3,0) ..
      (3,1);
      \draw (3,1) .. controls (3,2) and (2.5,2) .. (2,1) .. controls (1.5,0) and (1,0) .. (1,1);
      \bluedot{(1.02,1.3)} node[color=black,anchor=east] {\dotlabel{t}};
    \end{tikzpicture}
    \\
    \ \stackrel{\eqref{eq:dots-vs-right-curls}}{=} \
    \begin{tikzpicture}[anchorbase]
      \draw [->](0,0) arc (180:360:0.5);
      \draw (1,0) arc (0:180:0.5);
      \bluedot{(0.8,-0.4)} node[color=black,anchor=west] {\dotlabel{d+t}};
    \end{tikzpicture}
    + \ \sum_{j=0}^{d-1} c_{d-j}\
    \begin{tikzpicture}[anchorbase]
      \draw [->](0,0) arc (180:360:0.5);
      \draw (1,0) arc (0:180:0.5);
      \bluedot{(0.8,-0.4)} node[color=black,anchor=west] {\dotlabel{t+j}};
    \end{tikzpicture}
    \ .
  \end{multline*}
  Now, for $s \in \N$, take $t=d+s$.  By \eqref{rel:cc-bubble}, the above allows us to write a clockwise bubble with $s$ dots in terms of counterclockwise bubbles and clockwise bubbles with fewer than $s$ dots.  The result then follows by induction.
\end{proof}

\begin{lem}[Bubble slide] \label{lem:bubble_slide}
  For $t \ge 0$, we have
  \[
    \begin{tikzpicture}[anchorbase]
      \draw[->] (-0.3,-0.6) to (-0.3,0.6);
      \draw [->](0,0) arc (180:360:0.4);
      \draw (0.8,0) arc (0:180:0.4);
      \bluedot{(0.7,-0.3)} node [anchor=north,color=black] {\dotlabel{d-1+t}};
    \end{tikzpicture}
    \ =\
    \begin{tikzpicture}[anchorbase]
      \draw [->](0,0) arc (180:360:0.4);
      \draw (0.8,0) arc (0:180:0.4);
      \bluedot{(0.7,-0.3)} node [anchor=north,color=black] {\dotlabel{d-1+t}};
      \draw[->] (1.3,-0.6) to (1.3,0.6);
    \end{tikzpicture}
    \ -\
    \sum_{b=0}^{t-2} (t-1-b)\
    \begin{tikzpicture}[anchorbase]
      \draw [->](0,0) arc (180:360:0.4);
      \draw (0.8,0) arc (0:180:0.4);
      \bluedot{(0.7,-0.3)} node [anchor=north,color=black] {\dotlabel{d-1+b}};
      \draw[->] (1.3,-0.6) to (1.3,0.6);
      \bluedot{(1.3,0)} node [anchor=west,color=black] {\dotlabel{t-2-b}};
    \end{tikzpicture}\ .
  \]
  (When $t \le 1$, we interpret the sum as being empty.)
\end{lem}

\begin{proof}
  We have
  \begin{multline*}
    \begin{tikzpicture}[anchorbase]
      \draw [->](0,0) arc (180:360:0.4);
      \draw (0.8,0) arc (0:180:0.4);
      \bluedot{(0.7,-0.3)} node [anchor=north,color=black] {\dotlabel{d-1+t}};
      \draw[->] (1.3,-0.6) to (1.3,0.6);
    \end{tikzpicture}
    \ \stackrel{\substack{\eqref{rel:s-squared} \\ \eqref{rel:up-down-double-crossing}, \eqref{eq:mult_dotslide}}}{=} \
    \begin{tikzpicture}[anchorbase]
      \draw[->] (-0.3,-0.6) to (-0.3,0.6);
      \draw [->](0,0) arc (180:360:0.4);
      \draw (0.8,0) arc (0:180:0.4);
      \bluedot{(0.7,-0.3)} node [anchor=north,color=black] {\dotlabel{d-1+t}};
    \end{tikzpicture}
    \ + \sum_{a=0}^{d-2+t}\
    \begin{tikzpicture}[anchorbase]
      \draw (0,0) .. controls (0,.5) and (.7,.5) .. (.9,0);
      \draw (0,0) .. controls (0,-.5) and (.7,-.5) .. (.9,0);
      \draw (1,-1) .. controls (1,-.5) .. (.9,0);
      \draw[->] (.9,0) .. controls (1,.5) .. (1,1);
      \bluedot{(.2,0.34)} node [anchor=south, color=black] {\dotlabel{d-2+t-a}};
      \bluedot{(.99,-.5)} node [anchor=west, color=black] {\dotlabel{a}};
    \end{tikzpicture}
    \stackrel{\eqref{eq:dotted_leftcurl}}{=} \
    \begin{tikzpicture}[anchorbase]
      \draw[->] (-0.3,-0.6) to (-0.3,0.6);
      \draw [->](0,0) arc (180:360:0.4);
      \draw (0.8,0) arc (0:180:0.4);
      \bluedot{(0.7,-0.3)} node [anchor=north,color=black] {\dotlabel{d-1+t}};
    \end{tikzpicture}
    \ + \sum_{a=0}^{t-2}\
    \begin{tikzpicture}[anchorbase]
      \draw (0,0) .. controls (0,.5) and (.7,.5) .. (.9,0);
      \draw (0,0) .. controls (0,-.5) and (.7,-.5) .. (.9,0);
      \draw (1,-1) .. controls (1,-.5) .. (.9,0);
      \draw[->] (.9,0) .. controls (1,.5) .. (1,1);
      \bluedot{(.2,0.34)} node [anchor=south, color=black] {\dotlabel{d-2+t-a}};
      \bluedot{(.99,-.5)} node [anchor=west, color=black] {\dotlabel{a}};
    \end{tikzpicture}
    \\
    \stackrel{\substack{\eqref{eq:mult_dotslide} \\ \eqref{eq:dotted_leftcurl}}}{=} \
    \begin{tikzpicture}[anchorbase]
      \draw[->] (-0.3,-0.6) to (-0.3,0.6);
      \draw [->](0,0) arc (180:360:0.4);
      \draw (0.8,0) arc (0:180:0.4);
      \bluedot{(0.7,-0.3)} node [anchor=north, color=black] {\dotlabel{d-1+t}};
    \end{tikzpicture}
    \ + \sum_{a=0}^{t-2} \sum_{b=0}^{t-2-a}\
    \begin{tikzpicture}[anchorbase]
      \draw [->](0,0) arc (180:360:0.4);
      \draw (0.8,0) arc (0:180:0.4);
      \bluedot{(0.7,-0.3)} node [anchor=north, color=black] {\dotlabel{d-1+b}};
      \draw[->] (1.3,-0.6) to (1.3,0.6);
      \bluedot{(1.3,0)} node[anchor=west, color=black] {\dotlabel{t-2-b}};
    \end{tikzpicture}
    \ = \
    \begin{tikzpicture}[anchorbase]
      \draw[->] (-0.3,-0.6) to (-0.3,0.6);
      \draw [->](0,0) arc (180:360:0.4);
      \draw (0.8,0) arc (0:180:0.4);
      \bluedot{(0.7,-0.3)} node [anchor=north, color=black] {\dotlabel{d-1+t}};
    \end{tikzpicture}
    \ + \sum_{b=0}^{t-2} (t-1-b)\
    \begin{tikzpicture}[anchorbase]
      \draw [->](0,0) arc (180:360:0.4);
      \draw (0.8,0) arc (0:180:0.4);
      \bluedot{(0.7,-0.3)} node [anchor=north, color=black] {\dotlabel{d-1+b}};
      \draw[->] (1.3,-0.6) to (1.3,0.6);
      \bluedot{(1.3,0)} node[anchor=west, color=black] {\dotlabel{t-2-b}};
    \end{tikzpicture} \ ,
  \end{multline*}
  where the last equality follows by counting, in the penultimate expression, how many times we get the term corresponding to a fixed $b$.  The term with $b=0$ appears once for all $a=0,\dotsc,t-2$,  the term with $b=1$ appears once for all $a=0,\dotsc,t-3$, etc.
\end{proof}

Let $\Pi :=\kk[y_1,y_2,y_3,\ldots]$ be a polynomial algebra in countably many variables and consider the homomorphism of algebras
\begin{equation} \label{eq:bubble-iso}
  \psi_0 \colon \Pi \to \End_{\tcH^\lambda}(\one),\quad
  y_k \mapsto
  \begin{tikzpicture}[anchorbase]
    \draw [->](0,0) arc (180:360:0.5);
    \draw (1,0) arc (0:180:0.5);
    \bluedot{(0.8,-0.4)} node[color=black,anchor=west] {\dotlabel{d+k}};
  \end{tikzpicture}\ .
\end{equation}

\begin{prop} \label{prop:bubble-iso}
  The homomorphism $\psi_0$ is an isomorphism.
\end{prop}

\begin{proof}
  It is clear from the above results, that any closed diagram can be converted into a linear combination of non-nested dotted counterclockwise bubbles.    Thus $\psi_0$ is surjective.  We will prove it is injective in Proposition~\ref{prop:psi0-injective}.
\end{proof}

Let $H_m$ be the degenerate affine Hecke algebra of rank $m$ (see Definition~\ref{def:degAHA}).  By \eqref{rel:braid}, \eqref{rel:s-squared}, and \eqref{rel:dotslide}, there exists a well-defined homomorphism
\begin{equation} \label{eq:Hn-elements-as-diagrams}
  H_m \to \End_{\tcH^\lambda}(\sQ_+^m)
\end{equation}
determined by sending $s_i$ to the diagram with only one crossing, between the $i$-th and \mbox{$(i+1)$-st} strand, and $x_j$ to a dot on the $j$-th strand of the identity endomorphism.  \emph{It is important to note that here we number strands from right to left}.  By placing bubbles in the far right region of a diagram, we obtain a homomorphism of algebras
\begin{equation} \label{eq:psi_m-def}
  \psi_m \colon H_m \otimes \Pi \to \End_{\tcH^\lambda}(\sQ_+^m).
\end{equation}

\begin{prop} \label{prop:DH-Q-iso}
  The homomorphism $\psi_m$ is an isomorphism.
\end{prop}

\begin{proof}
  Just as in~\cite[Prop.~4]{Kho14}, it is straightforward to show that $\psi_m$ is surjective, because any diagram in $\End_{\tcH^\lambda}(\sQ_+^m)$ is equal to a linear combination of permutation diagrams with dots at the top and dotted counterclockwise bubbles in the far right region.  We use \eqref{eq:dots-vs-right-curls} to eliminate right curls (the dots of Khovanov's Heisenberg category).  We will prove $\psi_m$ is injective in Proposition~\ref{prop:psim-injective}.
\end{proof}

Following \cite[Prop.~5]{Kho14}, we can give an explicit basis of the hom-spaces in $\tcH^\lambda$.

\begin{defin}
  For two sign sequences $\epsilon, \epsilon'$, let $B(\epsilon, \epsilon')$ be the set of planar diagrams obtained in the following manner:
  \begin{itemize}
    \item We write the sequence $\epsilon$ at the bottom of the plan strip $\R \times [0,1]$, and we write the sequence $\epsilon'$ at the top of this strip.

    \item We match the elements of $\epsilon$ and $\epsilon'$ by oriented segments embedded in the strip.  The orientation of the endpoints of each segment must agree with the sign (i.e.\ be oriented up at a $+$ sign and down at a $-$ sign).  No two segments can intersect more than once, and no self-intersection or triple intersections are allowed.

    \item We place some number (possibly zero) of dots on each segment near its out endpoint (i.e.\ between its out endpoint and any intersections with other intervals).

    \item In the rightmost region of the diagram, we draw a finite number of counterclockwise disjoint nonnested circles with at least $d+1$ dots each.
  \end{itemize}
  The diagrams in $B(\epsilon,\epsilon')$ are parameterized by $k!$ possible matchings of the $2k$ oriented endpoints, a sequence of $k$ nonnegative integers determining the number of dots on each interval, and by a finite sequence of nonnegative integers determining the number of counterclockwise circles with various numbers of dots.
\end{defin}

Below is an example of an element of $B(-++--,+--+-+-)$.
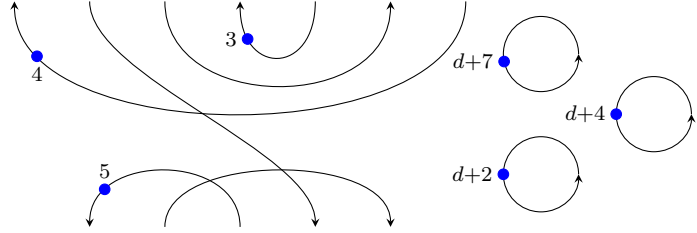
\[
  \begin{tikzpicture}[anchorbase]
    \draw[<-] (0,0) .. controls (0,1) and (2,1) .. (2,0);
    \draw[->] (1,0) .. controls (1,1) and (4,1) .. (4,0);
    \draw[->] (0,3) .. controls (0,2) and (3,1) .. (3,0);
    \draw[<-] (-1,3) .. controls (-1,1) and (5,1) .. (5,3);
    \draw[->] (1,3) .. controls (1,1.5) and (4,1.5) .. (4,3);
    \draw[<-] (2,3) .. controls (2,2) and (3,2) .. (3,3);
    \draw[->] (6.5,2.3) arc(0:360:.5);
    \draw[->] (6.5,.7) arc(0:360:.5);
    \draw[->] (8,1.5) arc(0:360:.5);
    \bluedot{(-.7,2.27)} node[anchor=north,color=black] {\dotlabel{4}};
    \bluedot{(.2,.5)} node[anchor=south,color=black] {\dotlabel{5}};
    \bluedot{(2.1,2.5)} node[anchor=east,color=black] {\dotlabel{3}};
    \bluedot{(5.5,.7)} node[anchor=east,color=black] {\dotlabel{d+2}};
    \bluedot{(5.51,2.2)} node[anchor=east,color=black] {\dotlabel{d+7}};
    \bluedot{(7,1.5)} node[anchor=east,color=black] {\dotlabel{d+4}};
  \end{tikzpicture}
\]

\begin{prop} \label{prop:hom-space-basis}
  For any sign sequences $\epsilon, \epsilon'$, the set $B(\epsilon,\epsilon')$ is a $\kk$-basis of $\Hom_{\cH^\lambda} (\sQ_\epsilon,\sQ_{\epsilon'})$.
\end{prop}

\begin{proof}
  The proof is almost identical to that of \cite[Prop.~5]{Kho14} and so will be omitted.  (See also \cite[Prop.~3.11]{LS13}.)
  \details{
    It is straightforward to check that, using the defining local relations of $\cH^\lambda$, any element of $\Hom_{\cH^\lambda} (\sQ_\epsilon, \sQ_{\epsilon'})$ can be reduced to a direct sum of elements of $B(\epsilon,\epsilon')$.  One uses \eqref{rel:s-squared}, \eqref{rel:up-down-double-crossing}, and \eqref{rel:down-up-double-crossing} to remove double crossings, \eqref{rel:dotslide} to move dots to the ends of intervals (modulo simpler diagrams), Lemma~\ref{lem:bubble_slide} to move circles to the rightmost region, etc.

    It remains to show that $B(\epsilon, \epsilon')$ is linearly independent.  Moving the lower endpoints of a diagram up using cup diagrams, or moving the upper endpoints of a diagram down using cap diagrams, yields canonical isomorphisms
    \[
      \Hom_{\cH^\lambda} (\sQ_\epsilon, \sQ_{\epsilon'})
      \cong \Hom_{\cH^\lambda}(\one, \sQ_{\bar \epsilon \epsilon'})
      \cong \Hom_{\cH^\lambda} (\one, \sQ_{\epsilon' \bar \epsilon}),
    \]
    where $\bar \epsilon$ is the sequence $\epsilon$ with the order and all signs reversed.  It thus suffices to show that $B(\varnothing, \epsilon)$ is linearly independent for any sequence $\epsilon$ with $k$ pluses and $k$ minuses.  We prove this by induction on $k$ and (for each $k$) by induction on the lexicographic order (where $+ < -$) among length $2k$ sequences.  Proposition~\ref{prop:DH-Q-iso} implies the base cases $k=0,1$ and $\epsilon = +^k-^k$ for any $k$.  Now assume $\epsilon = \epsilon_1 - +\, \epsilon_2$ for some sequences $\epsilon_1$ and $\epsilon_2$. By the inductive hypothesis, $B(\varnothing, \epsilon_1 \epsilon_2)$ and $B(\varnothing, \epsilon_1 + -\, \epsilon_2)$ are linearly independent.  Corollary~\ref{cor:Q+Q-commutation-relation} gives a canonical isomorphism
    \[
      \sQ_{\epsilon_1 + -\, \epsilon_2} \oplus \sQ_{\epsilon_1 \epsilon_2}^{\oplus d}
      \cong \sQ_{\epsilon_1 - +\, \epsilon_2}.
    \]
    This isomorphism maps the sets $B(\varnothing, \epsilon_1 + -\, \epsilon_2)$ and $B(\varnothing, \epsilon_1 \epsilon_2)$ to subsets $B_1$ and $B_2, B_3, \dotsc, B_{d+1}$ of $\Hom_{\cH^\lambda}(\sQ_{\epsilon_1 - + \epsilon_2})$.  Let $B = B_1 \cup \dotsb \cup B_{d+1}$.  It is straightforward to check that $B(\varnothing, \epsilon_1 - +\, \epsilon_2)$ is linearly independent if and only if $B$ is.  Since we know $B$ is linearly independent by induction, we are done.
  }
\end{proof}

\begin{prop} \label{prop:dot-shift-functor}
  Fix $j \in I$ and define $\mu = \sum_{i \in I} \mu_i \omega_i \in P_+$ by $\mu_i = \lambda_{i-j}$ for $i \in I$. Then the categories $\cH^\lambda$ and $\cH^\mu$ are isomorphic.
\end{prop}

\begin{proof}
  We claim that there is a functor $\Psi \colon \cH^\lambda \to \cH^\mu$ given by
  \begin{equation}\label{eq:imagedot}
    \begin{tikzpicture}[anchorbase]
      \draw[->] (0,0) -- (0,1);
      \bluedot{(0,0.5)};
    \end{tikzpicture}
    \ \mapsto\
    \begin{tikzpicture}[anchorbase]
      \draw[->] (0,0) -- (0,1);
      \bluedot{(0,0.5)};
    \end{tikzpicture}
    \ - j\
    \begin{tikzpicture}[anchorbase]
      \draw[->] (0,0) -- (0,1);
    \end{tikzpicture}
  \end{equation}
  and leaving all caps, cups and crossings unchanged.  It is clear that $\Psi$ is invertible and so it is enough to show that it is well-defined by proving that it preserves the local relations.  It preserves~\eqref{rel:braid}, \eqref{rel:s-squared}, \eqref{rel:up-down-double-crossing}, and~\eqref{rel:left-curl-zero}, because those relations do not involve dots (or dual dots).  Furthermore, it is straightforward to check that relation~\eqref{rel:dotslide} is also preserved.

  As for relation~\eqref{rel:cc-bubble}, note that
  \begin{equation}\label{eq:imagedots}
    \begin{tikzpicture}[anchorbase]
      \draw[->] (0,0) -- (0,1);
      \bluedot{(0,0.5)} node [anchor=east,color=black] {\dotlabel{t}};
    \end{tikzpicture}
    \ \mapsto \sum_{s=0}^{t}
    \begin{tikzpicture}[anchorbase]
      \draw[->] (0,0) -- (0,1);
      \bluedot{(0,0.5)} node [anchor=east,color=black] {\dotlabel{s}};
    \end{tikzpicture}
    \ (-1)^{t-s}\binom{t}{s} j^{t-s}.
  \end{equation}
  Therefore, we have
  \[
    \begin{tikzpicture}[anchorbase]
      \draw [->](0,0) arc (0:360:0.3);
      \bluedot{(-0.6,0)}  node [anchor=east,color=black] {\dotlabel{t}};
    \end{tikzpicture}
    \ \mapsto \ \sum_{s=0}^t
    \begin{tikzpicture}[anchorbase]
      \draw [->](0,0) arc (0:360:0.3);
      \bluedot{(-0.6,0)}  node [anchor=east,color=black] {\dotlabel{s}};
    \end{tikzpicture}
    \ (-1)^{t-s}\binom{t}{s} j^{t-s}
    \ = \
    \begin{cases}
      0 & \text{if } t < d-1, \\
      1 & \text{if } t=d-1, \\
      \sum_{i \in I} i \lambda_i & \text{if } t=d,
    \end{cases}
  \]
  where the case $t=d$ follows from
  \[
    \sum_{i\in I} i\mu_{i} - j d
    = \sum_{i\in I} i\lambda_{i-j} - j \sum_{i\in I}\lambda_i
    = \sum_{i\in I} (i+j)\lambda_{i} - j \sum_{i\in I}\lambda_i
    = \sum_{i\in I} i\lambda_i.
  \]
  (The other cases are immediate.)  This shows that~\eqref{rel:cc-bubble} is preserved by $\Psi$.

  It remains to prove that \eqref{rel:down-up-double-crossing} is preserved.  Let $V$ be the free $\Pi$-module spanned by upwards pointing strands with at most $d-1$ dots.  The $\Pi$-action is via placing bubbles to the right of the strand.  We have a $\Pi$-bilinear pairing $V \times V \to \Pi$ given by the usual vertical composition in our category (i.e.\ stacking diagrams), followed by closing off to the left.  If we choose the $\Pi$-basis
  \[
    \begin{tikzpicture}[anchorbase]
      \draw[->] (0,0) -- (0,1);
      \bluedot{(0,0.5)} node[anchor=east,color=black] {\dotlabel{i}};
    \end{tikzpicture}
    \ ,\quad i=0,1,\dotsc,d-1,
  \]
  then the matrix for the pairing is unitriangular.  By Lemma~\ref{lem:circle-duality}, the dual dots are uniquely obtained by inverting this matrix.  Since the pairing only involves \eqref{rel:cc-bubble}, which we have already shown is invariant under the functor, it follows that the sum in \eqref{rel:down-up-double-crossing} is preserved.  Hence \eqref{rel:down-up-double-crossing} is preserved under $\Psi$.

  \details{
    One can also verify via direct computation that \eqref{rel:down-up-double-crossing} is preserved as follows.  First of all, we claim that
    \begin{equation}\label{eq:imagec}
      c_s\mapsto \sum_{k=0}^{s} \binom{d-s+k}{k} c_{s-k}
    \end{equation}
    for any $0\leq s\leq d$. This can be proved by induction on $s$ after expanding the determinant which defines $c_s$, as in Lemma~\ref{lem:circle-det-expand}. We leave the details to the reader.

    Next note that, for any $0\leq t\leq d-1$, we have
    \begin{equation}\label{eq:recursivedualdots}
      \begin{tikzpicture}[anchorbase]
        \draw[->] (0,0) -- (0,1);
        \bluedot{(0,0.5)} node [anchor=east,color=black] {\dotlabel{t^\vee}};
      \end{tikzpicture}
      \ = \
      \begin{tikzpicture}[anchorbase]
        \draw[->] (0,0) -- (0,1);
      \end{tikzpicture}
      \ c_{d-(t+1)} \ + \
      \begin{tikzpicture}[anchorbase]
        \draw[->] (0,0) -- (0,1);
        \bluedot{(0,0.25)} node [anchor=west,color=black] {\dotlabel{(t+1)^\vee}};
        \bluedot{(0,0.65)};
      \end{tikzpicture}
    \end{equation}
    where in the second term on the righthand side the upper dot is not dual and, by convention, any diagram containing a dot labeled $d^\vee$ is zero.

    Now we will show by downwards induction on $t$ that, for any $0\leq t\leq d-1$, we have
    \begin{equation}\label{eq:imagedualdots}
      \begin{tikzpicture}[anchorbase]
        \draw[->] (0,0) -- (0,1);
        \bluedot{(0,0.5)} node [anchor=east,color=black] {\dotlabel{t^\vee}};
      \end{tikzpicture}
      \ \mapsto \sum_{s=t}^{d-1}
      \begin{tikzpicture}[anchorbase]
        \draw[->] (0,0) -- (0,1);
        \bluedot{(0,0.5)} node [anchor=east,color=black] {\dotlabel{s^\vee}};
      \end{tikzpicture}
      \ \binom{s}{t} j^{s-t}.
    \end{equation}
    By definition, we have
    \[
      \begin{tikzpicture}[anchorbase]
        \draw[->] (0,0) -- (0,1);
        \bluedot{(0,0.5)} node [anchor=east,color=black] {\dotlabel{(d-1)^\vee}};
      \end{tikzpicture}
      \ = \
      \begin{tikzpicture}[anchorbase]
        \draw[->] (0,0) -- (0,1);
      \end{tikzpicture}
      \ ,
    \]
    so the assertion holds for $t=d-1$.  By induction, we can then compute
    \begin{multline*}
      \begin{tikzpicture}[anchorbase]
        \draw[->] (0,0) -- (0,1);
      \end{tikzpicture}
      \ c_{d-(t+1)} \ + \
      \begin{tikzpicture}[anchorbase]
        \draw[->] (0,0) -- (0,1);
        \bluedot{(0,0.25)} node [anchor=west,color=black] {\dotlabel{(t+1)^\vee}};
        \bluedot{(0,0.65)};
      \end{tikzpicture}
      \ \mapsto \\
      \sum_{k=0}^{d-(t+1)} \binom{t+1+k}{k}
      \
        \begin{tikzpicture}[anchorbase]
          \draw[->] (0,0) -- (0,1);
        \end{tikzpicture}
        \ c_{d-(t+1+k)}
        \ + \
      \sum_{k=0}^{d-(t+1)} \binom{t+1+k}{k} \
     \left(
        \begin{tikzpicture}[anchorbase]
          \draw[->] (0,0) -- (0,1);
          \bluedot{(0,0.25)} node [anchor=east,color=black] {\dotlabel{(t+1+k)^\vee}};
          \bluedot{(0,0.65)} node [anchor=east,color=black] {};
        \end{tikzpicture}
        \ j^k - \
        \begin{tikzpicture}[anchorbase]
          \draw[->] (0,0) -- (0,1);
          \bluedot{(0,0.5)} node [anchor=east,color=black] {\dotlabel{(t+1+k)^\vee}};
        \end{tikzpicture}
        \ j^{k+1}
      \right)
    \end{multline*}
    which is obtained by applying~\eqref{eq:imagec} to $c_{d-(t+1)}$,~\eqref{eq:imagedot} to the dot, and~\eqref{eq:imagedualdots} to the dual dot.

    We then separate the two first terms with $k=0$ from the rest and reindex
    \[
      \sum_{k=0}^{d-(t+1)} \binom{t+1+k}{k}
      \begin{tikzpicture}[anchorbase]
        \draw[->] (0,0) -- (0,1);
        \bluedot{(0,0.5)} node [anchor=east,color=black] {\dotlabel{(t+1+k)^\vee}};
      \end{tikzpicture}
      \ j^{k+1} = \
      \sum_{k=1}^{d-(t+1)} \binom{t+k}{k-1}
      \begin{tikzpicture}[anchorbase]
        \draw[->] (0,0) -- (0,1);
        \bluedot{(0,0.5)} node [anchor=east,color=black] {\dotlabel{(t+k)^\vee}};
      \end{tikzpicture}
      \ j^{k}
    \]
    to obtain
    \begin{multline*}
      \begin{tikzpicture}[anchorbase]
        \draw[->] (0,0) -- (0,1);
      \end{tikzpicture}
      \ c_{d-(t+1)} + \
      \begin{tikzpicture}[anchorbase]
        \draw[->] (0,0) -- (0,1);
        \bluedot{(0,0.25)} node [anchor=east,color=black] {\dotlabel{(t+1)^\vee}};
        \bluedot{(0,0.65)} node [anchor=east,color=black] {};
      \end{tikzpicture}
      \ +
      \\
      \sum_{k=1}^{d-(t+1)} \binom{t+1+k}{k}
      \left(\
        \begin{tikzpicture}[anchorbase]
          \draw[->] (0,0) -- (0,1);
        \end{tikzpicture}
        \ c_{d-(t+1+k)} + \
        \begin{tikzpicture}[anchorbase]
          \draw[->] (0,0) -- (0,1);
          \bluedot{(0,0.25)} node [anchor=east,color=black] {\dotlabel{(t+1+k)^\vee}};
          \bluedot{(0,0.65)} node [anchor=east,color=black] {};
        \end{tikzpicture}
        \ j^k
      \right)
      \ - \sum_{k=1}^{d-(t+1)} \binom{t+k}{k-1}
      \begin{tikzpicture}[anchorbase]
        \draw[->] (0,0) -- (0,1);
        \bluedot{(0,0.5)} node [anchor=east,color=black] {\dotlabel{(t+k)^\vee}};
      \end{tikzpicture}
      \ j^{k}
      \\
      \stackrel{\eqref{eq:recursivedualdots}}{=}
      \begin{tikzpicture}[anchorbase]
        \draw[->] (0,0) -- (0,1);
        \bluedot{(0,0.5)} node [anchor=east,color=black] {\dotlabel{t^\vee}};
      \end{tikzpicture}
      \ + \sum_{k=1}^{d-(t+1)} \left(
      \binom{t+1+k}{k}- \binom{t+k}{k-1}\right)
      \begin{tikzpicture}[anchorbase]
        \draw[->] (0,0) -- (0,1);
        \bluedot{(0,0.5)} node [anchor=east,color=black] {\dotlabel{(t+k)^\vee}};
      \end{tikzpicture}
      \ j^k
    \end{multline*}
    By a well-known identity of binomial coefficients, we get
    \[
      \sum_{k=0}^{d-(t+1)} \binom{t+k}{k}
      \begin{tikzpicture}[anchorbase]
        \draw[->] (0,0) -- (0,1);
        \bluedot{(0,0.5)} node [anchor=east,color=black] {\dotlabel{(t+k)^\vee}};
      \end{tikzpicture}
      \ j^k = \sum_{s=t}^{d-1} \binom{s}{t}
      \begin{tikzpicture}[anchorbase]
        \draw[->] (0,0) -- (0,1);
        \bluedot{(0,0.5)} node [anchor=east,color=black] {\dotlabel{s^\vee}};
      \end{tikzpicture}
      \ j^{s-t},
    \]
    which proves~\eqref{eq:imagedualdots}.

    Now the fact that \eqref{rel:down-up-double-crossing} is preserved under the functor follows from the fact that
    \[
      \sum_{t=0}^{d-1}\
      \begin{tikzpicture}[anchorbase]
        \draw[->] (0,1) -- (0,0.9) arc (180:360:.25) -- (0.5,1);
        \draw[<-] (0,0) -- (0,0.1) arc (180:0:.25) -- (0.5,0);
        \bluedot{(0.45,0.75)} node [anchor=west,color=black] {\dotlabel{t}};
        \bluedot{(0.05,0.25)} node [anchor=east,color=black] {\dotlabel{t^\vee}};
      \end{tikzpicture}
      \ \mapsto \
      \ \sum_{t=0}^{d-1}\sum_{r=0}^{t}\sum_{s=t}^{d-1}
      \begin{tikzpicture}[anchorbase]
        \draw[->] (0,1) -- (0,0.9) arc (180:360:.25) -- (0.5,1);
        \draw[<-] (0,0) -- (0,0.1) arc (180:0:.25) -- (0.5,0);
        \bluedot{(0.45,0.75)} node [anchor=west,color=black] {\dotlabel{r}};
        \bluedot{(0.05,0.25)} node [anchor=east,color=black] {\dotlabel{s^\vee}};
      \end{tikzpicture}
      \ (-1)^{t-r}\binom{t}{r}\binom{s}{t}j^{s-r}
      \ =\
      \sum_{r=0}^{d-1}\
      \begin{tikzpicture}[anchorbase]
        \draw[->] (0,1) -- (0,0.9) arc (180:360:.25) -- (0.5,1);
        \draw[<-] (0,0) -- (0,0.1) arc (180:0:.25) -- (0.5,0);
        \bluedot{(0.45,0.75)} node [anchor=west,color=black] {\dotlabel{r}};
        \bluedot{(0.05,0.25)} node [anchor=east,color=black] {\dotlabel{r^\vee}};
      \end{tikzpicture}.
    \]
    The last equality follows from $\sum_{t=0}^{d-1} (-1)^{t-r} \binom{t}{r}\binom{s}{t}=\delta_{r,s}$, which can be shown  by using Newton's binomial formula twice to expand $((x-j)+j)^m$ in monomials of $x$ and $j$, for any $m\in\N$, and comparing it to $x^m$.
  }
\end{proof}

In light of Proposition~\ref{prop:dot-shift-functor}, it is natural to ask if the categories $\cH^\lambda$ and $\cH^\mu$ are isomorphic more generally for $\lambda$ and $\mu$ both of level $d$.  We do not currently know the answer to this question.  The explicit dependence on $\lambda$ appears only in the $j=d$ case of \eqref{rel:cc-bubble}.\footnote{After the current paper appeared, Brundan defined a Heisenberg category in \cite{Bru17} that omits the $j=d$ case of \eqref{rel:cc-bubble} and therefore depends only on the level $d$.}

%
\section{A filtration and the associated graded category} \label{sec:cat-filtration}
%

In Section~\ref{subsec:filtration-definition} we let $\kk$ be an arbitrary commutative ring, while in Section~\ref{subsec:symmetrizer-relations} we assume that $\kk$ is a field of characteristic zero.

\subsection{Definition} \label{subsec:filtration-definition}

We define a filtration on the hom-spaces of $\tcH^\lambda$ as follows.  We define
\[
  \deg\
  \begin{tikzpicture}[anchorbase]
    \draw[->] (0,-.5) arc (0:180:.5);
  \end{tikzpicture}
  \ = 1-d,
  \quad
  \deg\
  \begin{tikzpicture}[anchorbase]
    \draw[<-] (0,0) arc (180:360:.5);
  \end{tikzpicture}
  \ = d-1,
  \quad
  \deg\
  \begin{tikzpicture}[anchorbase]
    \draw[->] (0,0) -- (0,1);
    \bluedot{(0,.5)};
  \end{tikzpicture}
  \ = \deg\
  \begin{tikzpicture}[anchorbase]
    \draw[<-] (0,0) -- (0,1);
    \bluedot{(0,.5)};
  \end{tikzpicture}
  \ = 1.
\]
Right cups/caps and crossings are assigned degree zero.  This determines the degree of any planar diagram.  Then, for $\sx, \sy \in \Ob \cH^\lambda$ and $k \in \Z$, we define $\Hom_{\tcH^\lambda}^{\le k} (\sx,\sy)$ to be the span of the diagrams of degree less than or equal to $k$.  This yields a filtration
\[
  \dotsb \subseteq \Hom_{\tcH^\lambda}^{\le k-1}(\sx,\sy)
  \subseteq \Hom_{\tcH^\lambda}^{\le k}(\sx,\sy)
  \subseteq \Hom_{\tcH^\lambda}^{\le k+1}(\sx,\sy) \subseteq \dotsb
\]
of $\Hom_{\tcH^\lambda}(\sx,\sy)$.  This filtration of the hom-spaces clearly respects composition:
\[
  \Hom_{\tcH^\lambda}^{\le k}(\sy,\sz) \times \Hom_{\tcH^\lambda}^{\le \ell}(\sx,\sy)
  \to \Hom_{\tcH^\lambda}^{\le k + \ell}(\sx,\sz),\quad \sx,\sy,\sz \in \Ob \tcH^\lambda.
\]
We thus have the associated graded category $\tcH^\lambda_\gr$ defined by $\Ob \tcH^\lambda_\gr = \Ob \tcH^\lambda$, and
\[
  \Hom_{\tcH^\lambda_\gr} (\sx,\sy) = \gr \Hom_{\tcH^\lambda} (\sx,\sy),\quad \text{for all } \sx, \sy \in \Ob \tcH^\lambda_\gr,
\]
where $\gr \Hom_{\tcH^\lambda} (\sx,\sy)$ denotes the associated graded space.  The composition in $\tcH^\lambda_\gr$ is induced from the one in $\tcH^\lambda$.  We will use the symbol $\circeq$ to denote equality of morphisms in $\tcH^\lambda_\gr$.  The morphism spaces of $\tcH^\lambda_\gr$ are of course graded:
\[
  \Hom_{\tcH^\lambda_\gr}(\sx,\sy) = \bigoplus_{k \in \Z} \Hom_{\tcH^\lambda_\gr}^k(\sx,\sy),\qquad
  \Hom_{\tcH^\lambda_\gr}^k(\sx,\sy) := \Hom_{\tcH^\lambda}^{\le k}(\sx,\sy) / \Hom_{\tcH^\lambda}^{\le k-1}(\sx,\sy).
\]

\begin{lem} \label{lem:filtered-dimensions}
  If $d>1$, then
  \begin{gather}
    \End_{\tcH^\lambda}^{\le k}(\sQ_+^n \sQ_-^m) = 0 \quad \text{for } k < 0, \quad \text{and} \\
    \End_{\tcH^\lambda}^{\le 0}(\sQ_+^n \sQ_-^m) \cong \kk S_n \otimes \kk S_m.
  \end{gather}
\end{lem}

\begin{proof}
  This follows immediately from Proposition~\ref{prop:hom-space-basis}.
\end{proof}

\begin{lem}
  We have the following equalities in $\tcH^\lambda_\gr$:
  \begin{equation} \label{eq:cc-bubble-slide-graded}
    \begin{tikzpicture}[anchorbase]
      \draw [->](0,0) arc (180:360:0.4);
      \draw (0.8,0) arc (0:180:0.4);
      \bluedot{(0.7,-0.3)} node [anchor=north,color=black] {\dotlabel{j}};
      \draw[->] (1.3,-0.6) to (1.3,0.6);
    \end{tikzpicture}
    \ \circeq\
    \begin{tikzpicture}[anchorbase]
      \draw[->] (-0.3,-0.6) to (-0.3,0.6);
      \draw [->](0,0) arc (180:360:0.4);
      \draw (0.8,0) arc (0:180:0.4);
      \bluedot{(0.7,-0.3)} node [anchor=north,color=black] {\dotlabel{j}};
    \end{tikzpicture}
    \ ,\quad j \in \N,
  \end{equation}

  \noindent\begin{minipage}{0.5\linewidth}
    \begin{equation} \label{eq:dotslide-graded}
      \begin{tikzpicture}[anchorbase]
        \draw[->] (0,0) to (1,1);
        \draw[->] (1,0) to (0,1);
        \bluedot{(0.25,0.25)} node[color=black,anchor=east] {\dotlabel{j}};
      \end{tikzpicture}
      \ \circeq\
      \begin{tikzpicture}[anchorbase]
        \draw[->] (0,0) to (1,1);
        \draw[->] (1,0) to (0,1);
        \bluedot{(.75,.75)} node[color=black,anchor=west] {\dotlabel{j}};
      \end{tikzpicture}
      \ ,\quad j \in \N,
    \end{equation}
  \end{minipage}%
  \begin{minipage}{0.5\linewidth}
    \begin{equation} \label{eq:dual-dotslide-graded}
      \begin{tikzpicture}[anchorbase]
        \draw[->] (0,0) to (1,1);
        \draw[->] (1,0) to (0,1);
        \bluedot{(0.25,0.25)} node[color=black,anchor=east] {\dotlabel{j^\vee}};
      \end{tikzpicture}
      \ \circeq\
      \begin{tikzpicture}[anchorbase]
        \draw[->] (0,0) to (1,1);
        \draw[->] (1,0) to (0,1);
        \bluedot{(.75,.75)} node[color=black,anchor=west] {\dotlabel{j^\vee}};
      \end{tikzpicture}
      \ ,\quad 0 \le j \le d-1.
    \end{equation}
  \end{minipage}\par\vspace{\belowdisplayskip}
\end{lem}

\begin{proof}
  For $j \le d$, \eqref{eq:cc-bubble-slide-graded} follows immediately from \eqref{rel:cc-bubble}.  For $j > d$, it follows from Lemma~\ref{lem:bubble_slide} after noting that the terms in the sum there have degree $t-2$, while a counterclockwise circle with $d-1+t$ dots has degree $t$.

  The relation~\eqref{eq:dotslide-graded} follows from \eqref{eq:mult_dotslide}.  Then \eqref{eq:dual-dotslide-graded} follows from \eqref{eq:cc-bubble-slide-graded} and \eqref{eq:dotslide-graded}.
\end{proof}

\subsection{Symmetrizer relations} \label{subsec:symmetrizer-relations}

Through this subsection, we assume $\kk$ is a field of characteristic zero.

For $n \in \N_+$ and $\mu$ a partition of $n$, we have the corresponding minimal idempotent $e_\mu \in \kk S_n \subseteq H_n$.  In particular,
\[
  e_{(n)} = \frac{1}{n!} \sum_{w \in S_n} w
\]
is the complete symmetrizer.  Via the homomorphism $\psi_n$ of \eqref{eq:psi_m-def}, we view $e_\mu$ as an element of $\End_{\cH^\lambda} (\sQ_+^n)$ and, via adjunction, as an element of $\End_{\cH^\lambda}(\sQ_-^n)$.   In the case $\mu = (n)$, we denote these idempotents by a white box labeled $n$ across $n$ strands.  We sometimes use a dashed strand to denote multiple strands, when the number of strands is clear from the diagram.
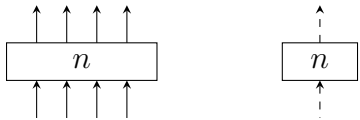
\[
  \begin{tikzpicture}[anchorbase]
    \draw (0,0) rectangle (2,.5) node[midway] {$n$};
    \draw[<-] (.4,0) to (.4,-0.5);
    \draw[<-] (.8,0) to (.8,-0.5);
    \draw[<-] (1.2,0) to (1.2,-0.5);
    \draw[<-] (1.6,0) to (1.6,-0.5);
    \draw[->] (.4,.5) to (.4,1);
    \draw[->] (.8,.5) to (.8,1);
    \draw[->] (1.2,.5) to (1.2,1);
    \draw[->] (1.6,.5) to (1.6,1);
  \end{tikzpicture}
  \qquad \qquad
  \begin{tikzpicture}[anchorbase]
    \draw (0,0) rectangle (1,.5) node[midway] {$n$};
    \draw[<-,dashed] (.5,0) to (.5,-0.5);
    \draw[->,dashed] (.5,.5) to (.5,1);
  \end{tikzpicture}
\]
Symmetrizers absorb crossings:
\begin{equation} \label{eq:symmetrizer-absorb-crossing}
  \begin{tikzpicture}[anchorbase]
    \draw (0,0) rectangle (2,.5) node[midway] {$n$};
    \draw[<-,dashed] (.4,0) to (.4,-0.5);
    \draw[<-] (.8,0) .. controls (.8,-0.25) and (1.2,-0.25) .. (1.2,-0.5);
    \draw[<-] (1.2,0) .. controls (1.2,-0.25) and (.8,-0.25) .. (.8,-0.5);
    \draw[<-,dashed] (1.6,0) to (1.6,-0.5);
    \draw[->,dashed] (.4,.5) to (.4,1);
    \draw[->] (.8,.5) to (.8,1);
    \draw[->] (1.2,.5) to (1.2,1);
    \draw[->,dashed] (1.6,.5) to (1.6,1);
  \end{tikzpicture}
  \ = \
  \begin{tikzpicture}[anchorbase]
    \draw (0,0) rectangle (2,.5) node[midway] {$n$};
    \draw[<-,dashed] (.4,0) to (.4,-0.5);
    \draw[<-] (.8,0) to (.8,-0.5);
    \draw[<-] (1.2,0) to (1.2,-0.5);
    \draw[<-,dashed] (1.6,0) to (1.6,-0.5);
    \draw[->,dashed] (.4,.5) to (.4,1);
    \draw[->] (.8,.5) to (.8,1);
    \draw[->] (1.2,.5) to (1.2,1);
    \draw[->,dashed] (1.6,.5) to (1.6,1);
  \end{tikzpicture}
  \ = \
  \begin{tikzpicture}[anchorbase]
    \draw (0,0) rectangle (2,.5) node[midway] {$n$};
    \draw[<-,dashed] (.4,0) to (.4,-0.5);
    \draw[<-] (.8,0) to (.8,-0.5);
    \draw[<-] (1.2,0) to (1.2,-0.5);
    \draw[<-,dashed] (1.6,0) to (1.6,-0.5);
    \draw[->,dashed] (.4,.5) to (.4,1);
    \draw[->] (.8,.5) .. controls (.8,.75) and (1.2,.75) .. (1.2,1);
    \draw[->] (1.2,.5) .. controls (1.2,.75) and (.8,.75) .. (.8,1);
    \draw[->,dashed] (1.6,.5) to (1.6,1);
  \end{tikzpicture}
\end{equation}
It also follows from \eqref{rel:braid} that symmetrizers pass through crossings (for either orientation of the solid strand):
\[
  \begin{tikzpicture}[anchorbase]
    \draw (0,0) rectangle (1,.5) node[midway] {$n$};
    \draw[<-,dashed] (.5,0) .. controls (0.5,-0.5) and (1.5,-0.5) .. (1.5,-1);
    \draw[->,dashed] (.5,.5) .. controls (0.5,1) and (-0.5,1) .. (-0.5,1.5);
    \draw (-0.5,-1) .. controls (-0.5,1.5) and (0,1) .. (1.5,1.5);
  \end{tikzpicture}
  \ =\
  \begin{tikzpicture}[anchorbase]
    \draw (0,0) rectangle (1,.5) node[midway] {$n$};
    \draw[<-,dashed] (.5,0) .. controls (0.5,-0.5) and (1.5,-0.5) .. (1.5,-1);
    \draw[->,dashed] (.5,.5) .. controls (0.5,1) and (-0.5,1) .. (-0.5,1.5);
    \draw (-0.5,-1) .. controls (1,-0.5) and (1.5,-1) .. (1.5,1.5);
  \end{tikzpicture}
\]

Corresponding to the idempotent $e_\mu$, for $\mu$ a partition of $n$, we have the object
\[
  \sQ_\pm^{\mu} := (\sQ_\pm^n, e_\mu) \in \Ob \cH^\lambda.
\]
By convention, we set $\sQ_\pm^\varnothing = \one$.

If $\bb = (b_1,b_2,\dotsc,b_\ell) \in \{0,1,\dotsc,d-1\}^\ell$, we define
\begin{gather*}
  \begin{tikzpicture}[anchorbase]
    \draw[<-,dashed] (0,1) -- (0,0);
    \bluedot{(0,0.5)} node[anchor=east, color=black] {\dotlabel{\bb}};
  \end{tikzpicture}
  \ = \
  \begin{tikzpicture}[anchorbase]
    \draw[<-] (0,1) -- (0,0);
    \bluedot{(0,0.5)} node[anchor=east, color=black] {\dotlabel{b_\ell}};
    \draw (0.7,0.5) node {$\cdots$};
    \draw[<-] (1.4,1) -- (1.4,0);
    \bluedot{(1.4,0.5)} node[anchor=west, color=black] {\dotlabel{b_2}};
    \draw[<-] (2.4,1) -- (2.4,0);
    \bluedot{(2.4,0.5)} node[anchor=west, color=black] {\dotlabel{b_1}};
  \end{tikzpicture},
  \qquad
  \begin{tikzpicture}[anchorbase]
    \draw[<-,dashed] (0,1) -- (0,0);
    \bluedot{(0,0.5)} node[anchor=east, color=black] {\dotlabel{\bb^\vee}};
  \end{tikzpicture}
  \ = \
  \begin{tikzpicture}[anchorbase]
    \draw[<-] (0,1) -- (0,0);
    \bluedot{(0,0.5)} node[anchor=east, color=black] {\dotlabel{b_\ell^\vee}};
    \draw (0.7,0.5) node {$\cdots$};
    \draw[<-] (1.4,1) -- (1.4,0);
    \bluedot{(1.4,0.5)} node[anchor=west, color=black] {\dotlabel{b_2^\vee}};
    \draw[<-] (2.4,1) -- (2.4,0);
    \bluedot{(2.4,0.5)} node[anchor=west, color=black] {\dotlabel{b_1^\vee}};
  \end{tikzpicture},
  \\
  S_\bb = \{w \in S_\ell \mid b_{w(i)} = b_i \text{ for all } 1 \le i \le \ell\}.
\end{gather*}

For $\ell \in \N_+$, let
\[
  \bB_\ell = \{(b_1,b_2,\dotsc,b_\ell) \mid 0 \le b_1 \le b_2 \le \dotsb b_\ell \le d-1\}.
\]
Note that
\begin{equation} \label{eq:dot-sum-counting}
  \sum_{\bb \in \{0,1,\dotsc,d-1\}^\ell}\
  \begin{tikzpicture}[anchorbase]
    \draw (-0.5,0.5) rectangle (0.5,1) node[midway] {$\ell$};
    \draw (-0.5,-0.5) rectangle (0.5,-1) node[midway] {$\ell$};
    \draw[->,dashed] (0,-0.5) -- (0,0.5);
    \bluedot{(0,0)} node[color=black, anchor=east] {\dotlabel{\bb}};
  \end{tikzpicture}
  \ =\
  \sum_{\bb \in \bB_\ell} \frac{\ell!}{|S_\bb|}\
  \begin{tikzpicture}[anchorbase]
    \draw (-0.5,0.5) rectangle (0.5,1) node[midway] {$\ell$};
    \draw (-0.5,-0.5) rectangle (0.5,-1) node[midway] {$\ell$};
    \draw[->,dashed] (0,-0.5) -- (0,0.5);
    \bluedot{(0,0)} node[color=black, anchor=east] {\dotlabel{\bb}};
  \end{tikzpicture}
\end{equation}

For $\bb = (b_1,\dotsc,b_\ell) \in \{0,1,\dotsc,d-1\}^\ell$, define
\begin{equation} \label{eq:alpha_bvee-def}
  \alpha_{\bb^\vee} =\
  \begin{tikzpicture}[anchorbase]
    \draw (-0.5,1) rectangle (-2,1.5) node[midway] {$m-\ell$};
    \draw (0.5,1) rectangle (2,1.5) node[midway] {$n-\ell$};
    \draw (-0.5,-1) rectangle (-2,-1.5) node[midway] {$n$};
    \draw (0.5,-1) rectangle (2,-1.5) node[midway] {$m$};
    \draw[<-,dashed] (-0.8,-1) arc (180:0:0.8);
    \bluedot{(-0.7,-0.6)} node[color=black,anchor=west] {\dotlabel{\bb^\vee}};
    \draw[<-,dashed] (-1.25,1) .. controls (-1.25,0.3) and (1.7,-0.3) .. (1.7,-1);
    \draw[->,dashed] (1.25,1) .. controls (1.25,0.3) and (-1.7,-0.3) .. (-1.7,-1);
  \end{tikzpicture}
  \ \colon  \sQ_-^{(n)} \otimes \sQ_+^{(m)} \to \sQ_+^{(m-\ell)} \otimes \sQ_-^{(n-\ell)}
\end{equation}
and
\begin{equation} \label{eq:beta_b-def}
  \beta_\bb =\
  \begin{tikzpicture}[anchorbase]
    \draw (-0.5,-1) rectangle (-2,-1.5) node[midway] {$m-\ell$};
    \draw (0.5,-1) rectangle (2,-1.5) node[midway] {$n-\ell$};
    \draw (-0.5,1) rectangle (-2,1.5) node[midway] {$n$};
    \draw (0.5,1) rectangle (2,1.5) node[midway] {$m$};
    \draw[->,dashed] (-0.8,1) arc (180:360:0.8);
    \bluedot{(+0.72,0.6)} node[color=black,anchor=east] {\dotlabel{\bb}};
    \draw[->,dashed] (-1.25,-1) .. controls (-1.25,-0.3) and (1.7,0.3) .. (1.7,1);
    \draw[<-,dashed] (1.25,-1) .. controls (1.25,-0.3) and (-1.7,0.3) .. (-1.7,1);
  \end{tikzpicture}
  \ \colon \sQ_+^{(m-\ell)} \otimes \sQ_-^{(n-\ell)} \to  \sQ_-^{(n)} \otimes \sQ_+^{(m)}.
\end{equation}

\begin{lem} \label{lem:alpha-beta-composition}
  Suppose $0 \le k,\ell \le \min \{m,n\}$.  Then, in $\cH^\lambda_\gr$, we have
  \[
    \alpha_{\bc^\vee} \circ \beta_\bb \circeq \delta_{\bc, \bb} \frac{|S_\bb| (n-k)! (m-k)!}{n!m!}    \id_{\sQ_+^{(m-k)} \sQ_-^{(n-k)}}
    \quad \text{for all } \bb \in \bB_\ell,\ \bc \in \bB_k.
  \]
\end{lem}

\begin{proof}
  Using Lemma~\ref{lem:circle-duality}, the proof is almost identical to the analogous argument in the proof of \cite[Th.~9.2]{RS17} and so will be omitted.
  \details{
    We work in the graded category $\cH^\lambda_\gr$.  We have
    \[
      \alpha_{\bc^\vee} \circ \beta_\bb \circeq
      \begin{tikzpicture}[anchorbase]
        \draw (-0.5,-1) rectangle (-2,-1.5) node[midway] {$m-\ell$};
        \draw (0.5,-1) rectangle (2,-1.5) node[midway] {$n-\ell$};
        \draw (-0.5,1) rectangle (-2,1.5) node[midway] {$n$};
        \draw (0.5,1) rectangle (2,1.5) node[midway] {$m$};
        \draw[->,dashed] (-0.8,1) arc (180:360:0.8);
        \bluedot{(+0.72,0.6)} node[color=black,anchor=east] {\dotlabel{\bb}};
        \draw[->,dashed] (-1.25,-1) .. controls (-1.25,-0.3) and (1.7,0.3) .. (1.7,1);
        \draw[<-,dashed] (1.25,-1) .. controls (1.25,-0.3) and (-1.7,0.3) .. (-1.7,1);
        \draw (-0.5,3.5) rectangle (-2,4) node[midway] {$m-k$};
        \draw (0.5,3.5) rectangle (2,4) node[midway] {$n-k$};
        \draw[<-,dashed] (-0.8,1.5) arc (180:0:0.8);
        \bluedot{(-0.72,1.9)} node[color=black,anchor=west] {\dotlabel{\bc^\vee}};
        \draw[<-,dashed] (-1.25,3.5) .. controls (-1.25,2.8) and (1.7,2.2) .. (1.7,1.5);
        \draw[->,dashed] (1.25,3.5) .. controls (1.25,2.8) and (-1.7,2.2) .. (-1.7,1.5);
      \end{tikzpicture}\ .
    \]
    If $k \ne \ell$, then, after expanding the symmetrizers, all diagrams involved will contain left curls and hence be zero by \eqref{rel:left-curl-zero}.  Now assume $k = \ell$ and consider the expansions of the symmetrizers
    \[
      e_{(n)} = \frac{1}{n!} \sum_{w \in S_n} w
      \quad \text{and} \quad
      e_{(m)} = \frac{1}{m!} \sum_{v \in S_m} v.
    \]
    If we replace the boxes labeled $n$ and $m$ by $w \in S_n$ and $v \in S_m$, respectively, then the resulting diagram will have a left curl unless
    \[
      w = (w_1,w_2) \in S_{n-k} \times S_k \subseteq S_n
      \quad \text{and} \quad
      v = (w_2^{-1}, v_2) \in S_k \times S_{m-k} \subseteq S_m,
    \]
    for some $w_1 \in S_{n-k}$, $w_2 \in S_k$, and $v_2 \in S_{m-k}$.  Furthermore, in light of Lemma~\ref{lem:circle-duality}, the diagram is zero if $\bb \ne \bc$.  If $\bb = \bc$, there are precisely $|S_\bb|$ choices for $w_2$ that match the $\bb$ and $\bb^\vee$ correctly, evaluating to one.  For all other choices, the diagram is zero.
  }
\end{proof}

\begin{prop} \label{prop:associated-graded-HB}
  The category obtained from $\tcH^\lambda_\gr$ by imposing the extra local relation that $d$ dots is equal to zero is isomorphic to the Heisenberg category $\cH_B'$ of \cite[\S6]{RS17} with $B = \kk[x]/(x^d)$ and trace map $\tr_B \colon B \to \kk$ given by $\tr(x^j) = \delta_{j,d-1}$.
\end{prop}

\begin{proof}
  Let $\cC$ be the category obtained from $\tcH^\lambda_\gr$ by imposing the extra relation that $d$ dots is equal to zero.  It follows from \eqref{eq:c-def} that $c_s = 0$ in $\cC$ for all $s \ge 1$.  Thus, \eqref{eq:dot-vee-def} gives
  \[
    \begin{tikzpicture}[anchorbase]
      \draw[->] (0,0) -- (0,1);
      \bluedot{(0,0.5)} node [anchor=east,color=black] {\dotlabel{j^\vee}};
    \end{tikzpicture}
    \ = \
    \begin{tikzpicture}[anchorbase]
      \draw[->] (0,0) -- (0,1);
      \bluedot{(0,0.5)} node [anchor=west,color=black] {\dotlabel{d-1-j}};
    \end{tikzpicture}
    \ .
  \]
  Therefore, the local relations in $\cC$ become precisely the defining local relations of $\cH_B'$ (see \cite[(6.9)--(6.16)]{RS17}), where the dot in $\cC$ corresponds to a dot labelled $x$ in $\cH_B'$.
\end{proof}

%
\section{The Grothendieck ring} \label{sec:Groth-group}
%

Throughout this section we assume that $\kk$ is a field of characteristic zero.

\subsection{The Heisenberg algebra} \label{subsec:Heisenberg-algebra}

Recall that $d = \sum_{i \in I} \lambda_i$ is a positive integer.  Let $\Sym_\Q^+$ and $\Sym_\Q^-$ be two copies of the Hopf algebra of symmetric functions over $\Q$.  Consider the Hopf pairing
\[
  \langle -, - \rangle_d \colon \Sym_\Q^- \times \Sym_\Q^+ \to \Q,\quad \langle p_n^-, p_m^+ \rangle_d = nd \delta_{n,m},
\]
where $p_n^\pm$ denotes the $n$-th power sum in $\Sym^\pm$.  (The Hopf pairing is uniquely determined by its values on these elements.)  One can show that the pairing of two complete symmetric functions is an integer.  (This follows, for example, by comparing the coefficients appearing in \cite[Th.~5.3]{Sua15} to \cite[(2.2)]{Sua15}.)  Thus we can restrict to obtain a bilinear form $\Sym_\Z^- \times \Sym_\Z^+ \to \Z$, where $\Sym_\Z^\pm$ are copies of the Hopf algebra of symmetric functions over $\Z$.

The \emph{level $d$ Heisenberg algebra} is the Heisenberg double $\fh_d := \Sym_\Z^+ \# \Sym_\Z^-$.  One can obtain a presentation by choosing any sets of generators of $\Sym_\Z^\pm$.  Choosing the complete symmetric functions $h_n^\pm$, $n \in \N$, we obtain that $\fh_d$ is the unital associative $\Z$-algebra generated by $h_n^\pm$, $n \in \N$, subject to the relations $h_0^+=h_0^-=1$ and
\begin{equation} \label{eq:heisenberg-presentation}
  h_n^+ h_m^+ = h_m^+ h_n^+,\quad
  h_n^- h_m^- = h_m^- h_n^-,\quad
  h_n^- h_m^+ = \sum_{r=0}^{\min\{ m,n\}} \binom{d+r-1}{r}h_{m-r}^+ h_{n-r}^- ,\quad
  n, m \in \N.
\end{equation}
See \cite[Th.~5.3]{Sua15} and \cite[Prop.~A.1]{LRS18}.  (Readers not familiar with the Heisenberg double construction may take the presentation \eqref{eq:heisenberg-presentation} as the definition of $\fh_d$.)

\subsection{Decomposition of the identity}

Fix $m,n \in \N$ and define $\beta_\bb$ as in \eqref{eq:beta_b-def}.

\begin{lem} \label{lem:beta-duals-exist}
  There exist
  \[
    \theta_\bb \colon \sQ_-^{(n)} \otimes \sQ_+^{(m)} \to \sQ_+^{(n-k)} \otimes \sQ_-^{(n-k)},\quad \bb \in \bB_k,\ 0 \le k \le \min \{m,n\},
  \]
  such that
  \begin{equation} \label{eq:theta-beta-duality}
    \theta_\bc \circ \beta_\bb = \delta_{\bc,\bb} \id_{\sQ_+^{(m-k)} \sQ_-^{(n-k)}}
    \quad \text{for all } \bb \in \bB_\ell,\ \bc \in \bB_k,\ 0 \le k,\ell \le \min \{m,n\}.
  \end{equation}
\end{lem}

\begin{proof}
  For all $\bb \in \bB_\ell$, define
  \[
    \theta'_\bb = \frac{n!m!}{|S_\bb| (n-k)! (m-k)!} \alpha_{\bb^\vee}.
  \]
  By Lemma~\ref{lem:alpha-beta-composition}, we have
  \[
    \theta'_\bc \circ \beta_\bb - \delta_{\bc,\bb} \id_{\sQ_+^{(m-k)} \sQ_-^{(n-k)}}
    \in \End_{\cH^\lambda}^{\le \deg \bb - \deg \bc - 1} \left( \sQ_+^{(m-k)} \sQ_-^{(n-k)} \right).
  \]
  It follows from Lemma~\ref{lem:filtered-dimensions} that $\theta'_\bc \circ \beta_\bb = \delta_{\bc,\bb} \id_{\sQ_+^{(m-k)} \sQ_-^{(n-k)}}$ if $\deg \bb \le \deg \bc$.  The desired $\theta_\bc$ then exist by the usual the usual process of inverting a unitriangular matrix.
\end{proof}

\begin{prop} \label{prop:key-id-decomp}
  In $\cH^\lambda$, we have
  \begin{equation} \label{eq:key-id-decomp}
    \id_{\sQ_-^{(n)} \sQ_+^{(m)}}
    = \sum_{k=0}^{\min \{m,n\}} \sum_{\bb \in \bB_k} \beta_\bb \circ \theta_\bb.
  \end{equation}
  Furthermore,
  \[
    \beta_\bb \circ \theta_\bb,\quad \bb \in \bB_k,\ 1 \le k \le \min \{m,n\},
  \]
  are orthogonal idempotents in $\End_{\cH^\lambda}(\sQ_-^{(n)} \sQ_+^{(m)})$.
\end{prop}

\begin{proof}
  Starting with
  \[
    \id_{\sQ_-^{(n)} \sQ_+^{(m)}}
    =\
    \begin{tikzpicture}[anchorbase]
      \draw (-0.2,0.5) rectangle (-1,1) node[midway] {$n$};
      \draw (0.2,0.5) rectangle (1,1) node[midway] {$m$};
      \draw (-0.2,-0.5) rectangle (-1,-1) node[midway] {$n$};
      \draw (0.2,-0.5) rectangle (1,-1) node[midway] {$m$};
      \draw[->,dashed] (-0.6,0.5) -- (-0.6,-0.5);
      \draw[<-,dashed] (0.6,0.5) -- (0.6,-0.5);
    \end{tikzpicture}
    \ ,
  \]
  we repeatedly use  \eqref{rel:down-up-double-crossing} to pull the upward strands left through the downward strands.  This yields a linear combination of diagrams of the form depicted in Figure~\ref{subfig1:key-id-decomp}, where
  \begin{itemize}
    \item the black boxes represent some linear combinations of diagrams involving crossings and dots, with dots only appearing on the strands involved in the cups, and
    \item where the gray boxes represent some linear combinations of diagrams involving crossings and dual dots.
  \end{itemize}
  Then, using \eqref{rel:dotslide} and the fact that crossings can be absorbed into symmetrizers, we can assume that the black boxes do not involve any crossings and the total number of dots on the cups weakly increases as we move down.  Thus, we see that $\id_{\sQ_-^{(n)} \sQ_+^{(m)}}$ can be written as a linear combination of diagrams of the form depicted in Figure~\ref{subfig2:key-id-decomp}, where $\bb \in \bB_k$ for some $0 \le k \le \min \{m,n\}$, and the gray boxes represent some linear combinations of diagrams involving crossings and dual dots.  Then, using the fact that smaller symmetrizers can be absorbed into larger ones, and that symmetrizers slide through crossings, we see that $\id_{\sQ_-^{(n)} \sQ_+^{(m)}}$ can be written as a linear combination of diagrams of the form depicted in Figure~\ref{subfig3:key-id-decomp}, where $0 \le k \le \min \{m,n\}$, $\bb \in \bB_k$, and the gray boxes represent some linear combinations of diagrams involving crossings and dual dots.

  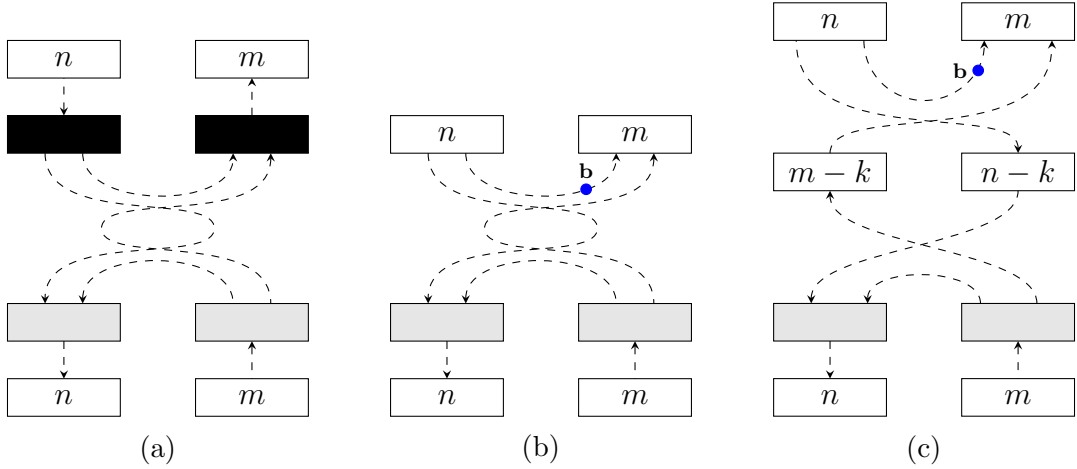
\begin{figure}[htb]
    \centering
    \begin{subfigure}[b]{0.3\linewidth}
      \centering
      \begin{tikzpicture}[anchorbase]
        \draw (-0.5,-2) rectangle (-2,-2.5) node[midway] {$n$};
        \draw (0.5,-2) rectangle (2,-2.5) node[midway] {$m$};
        \draw (-0.5,2) rectangle (-2,2.5) node[midway] {$n$};
        \draw (0.5,2) rectangle (2,2.5) node[midway] {$m$};
        \filldraw[fill=black] (-2,1) rectangle (-0.5,1.5);
        \filldraw[fill=black] (2,1) rectangle (0.5,1.5);
        \filldraw[fill=black!10] (-2,-1) rectangle (-0.5,-1.5);
        \filldraw[fill=black!10] (2,-1) rectangle (0.5,-1.5);
        \draw[dashed,->] (-1.25,-1.5) -- (-1.25,-2);
        \draw[dashed,<-] (1.25,-1.5) -- (1.25,-2);
        \draw[dashed,<-] (-1.25,1.5) -- (-1.25,2);
        \draw[dashed,->] (1.25,1.5) -- (1.25,2);
        \draw[dashed,->] (-1,1) .. controls (-1,0.25) and (1,0.25) .. (1,1);
        \draw[dashed,<-] (-1,-1) .. controls (-1,-0.25) and (1,-0.25) .. (1,-1);
        \draw[dashed,->] (-1.5,1) .. controls (-1.5,0) and (0.75,0.5) .. (0.75,0) .. controls (0.75,-0.5) and (-1.5,0) .. (-1.5,-1);
        \draw[dashed,<-] (1.5,1) .. controls (1.5,0) and (-0.75,0.5) .. (-0.75,0) .. controls (-0.75,-0.5) and (1.5,0) .. (1.5,-1);
      \end{tikzpicture}
      \caption{\label{subfig1:key-id-decomp}}
    \end{subfigure}
    \begin{subfigure}[b]{0.3\linewidth}
      \centering
      \begin{tikzpicture}[anchorbase]
        \draw (-0.5,-2) rectangle (-2,-2.5) node[midway] {$n$};
        \draw (0.5,-2) rectangle (2,-2.5) node[midway] {$m$};
        \draw (-2,1) rectangle (-0.5,1.5) node[midway] {$n$};
        \draw (2,1) rectangle (0.5,1.5) node[midway] {$m$};
        \filldraw[fill=black!10] (-2,-1) rectangle (-0.5,-1.5);
        \filldraw[fill=black!10] (2,-1) rectangle (0.5,-1.5);
        \draw[dashed,->] (-1.25,-1.5) -- (-1.25,-2);
        \draw[dashed,<-] (1.25,-1.5) -- (1.25,-2);
        \draw[dashed,->] (-1,1) .. controls (-1,0.25) and (1,0.25) .. (1,1);
        \draw[dashed,<-] (-1,-1) .. controls (-1,-0.25) and (1,-0.25) .. (1,-1);
        \draw[dashed,->] (-1.5,1) .. controls (-1.5,0) and (0.75,0.5) .. (0.75,0) .. controls (0.75,-0.5) and (-1.5,0) .. (-1.5,-1);
        \draw[dashed,<-] (1.5,1) .. controls (1.5,0) and (-0.75,0.5) .. (-0.75,0) .. controls (-0.75,-0.5) and (1.5,0) .. (1.5,-1);
        \bluedot{(0.6,0.52)} node[anchor=south,color=black] {\dotlabel{\bb}};
      \end{tikzpicture}
      \caption{\label{subfig2:key-id-decomp}}
    \end{subfigure}
    \begin{subfigure}[b]{0.3\linewidth}
      \centering
      \begin{tikzpicture}[anchorbase]
        \draw (-0.5,1.5) rectangle (-2,2) node[midway] {$m-k$};
        \draw (0.5,1.5) rectangle (2,2) node[midway] {$n-k$};
        \draw (-0.5,-1) rectangle (-2,-1.5) node[midway] {$n$};
        \draw (0.5,-1) rectangle (2,-1.5) node[midway] {$m$};
        \filldraw[fill=black!10] (-0.5,0) rectangle (-2,-0.5);
        \filldraw[fill=black!10] (0.5,0) rectangle (2,-0.5);
        \draw[dashed,->] (-1.25,-0.5) -- (-1.25,-1);
        \draw[dashed,<-] (1.25,-0.5) -- (1.25,-1);
        \draw[dashed,->] (0.75,0) .. controls (0.75,0.5) and (-0.75,0.5) .. (-0.75,0);
        \draw[dashed,->] (1.5,0) .. controls (1.5,0.75) and (-1.25,0.75) .. (-1.25,1.5);
        \draw[dashed,<-] (-1.5,0) .. controls (-1.5,0.75) and (1.25,0.75) .. (1.25,1.5);
        \draw (-0.5,3.5) rectangle (-2,4) node[midway] {$n$};
        \draw (0.5,3.5) rectangle (2,4) node[midway] {$m$};
        \draw[->,dashed] (-0.8,3.5) arc (180:360:0.8);
        \bluedot{(+0.72,3.1)} node[color=black,anchor=east] {\dotlabel{\bb}};
        \draw[->,dashed] (-1.25,2) .. controls (-1.25,2.7) and (1.7,2) .. (1.7,3.5);
        \draw[<-,dashed] (1.25,2) .. controls (1.25,2.7) and (-1.7,2) .. (-1.7,3.5);
      \end{tikzpicture}
      \caption{\label{subfig3:key-id-decomp}}
    \end{subfigure}
    \caption{Diagrams involved in the proof of Proposition~\ref{prop:key-id-decomp}. \label{fig:key-id-decomp}}
  \end{figure}

  The above shows that we can write
  \[
    \id_{\sQ_-^{(n)} \sQ_+^{(m)}}
    = \sum_{k=0}^{\min \{m,n\}} \sum_{\bb \in \bB_k} \beta_\bb \circ \theta''_\bb,
  \]
  for some $\theta''_\bb \in \Hom_{\cH^\lambda}(\sQ_-^{(n)} \sQ_+^{(m)}, \sQ_+^{(n-k)} \sQ_-^{(n-k)})$, $\bb \in \bB_k$, $0 \le k \le \min \{m,n\}$.  Composing on the left with $\theta_\bc$ and using Lemma~\ref{lem:beta-duals-exist} then yields that $\theta_\bb = \theta''_\bb$ for all $\bb$.  The final statement about orthogonal idempotents follows immediately from \eqref{eq:theta-beta-duality}.
\end{proof}

\subsection{Categorification of the higher level Heisenberg algebra}

\begin{prop} \label{prop:key-isom}
  Suppose $n,m \in \N$.  In $\cH^\lambda$, we have
  \begin{gather}
    \sQ_+^{(n)} \sQ_+^{(m)} \cong \sQ_+^{(m)} \sQ_+^{(n)}, \qquad
    \sQ_-^{(n)} \sQ_-^{(m)} \cong \sQ_-^{(m)} \sQ_-^{(n)}, \label{eq:Qpm-commute} \\
    \sQ_-^{(n)} \sQ_+^{(m)} \cong \sum_{k=0}^{\min \{m,n\}} \left(\sQ_+^{(m-k)} \sQ_-^{(n-k)}\right)^{\oplus {d+k-1 \choose k}}, \label{eq:key-isom}
  \end{gather}
  where, by convention, $\sQ_\pm^{(0)} = \one$.
\end{prop}

\begin{proof}
  Using the fact that symmetrizers slide through crossings, it is straightforward to verify that the morphisms
  
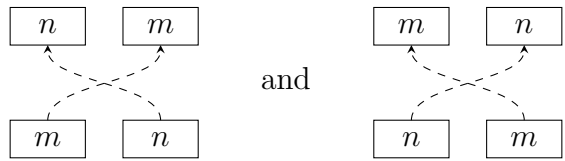
\[
    \begin{tikzpicture}[anchorbase]
      \draw (-0.25,1) rectangle (-1.25,.5) node[midway] {$n$};
      \draw (0.25,1) rectangle (1.25,.5) node[midway] {$m$};
      \draw (-0.25,-1) rectangle (-1.25,-.5) node[midway] {$m$};
      \draw (0.25,-1) rectangle (1.25,-.5) node[midway] {$n$};
      \draw[->,dashed] (-0.75,-0.5) .. controls (-0.75,0) and (0.75,0) .. (0.75,0.5);
      \draw[->,dashed] (0.75,-0.5) .. controls (0.75,0) and (-0.75,0) .. (-0.75,0.5);
    \end{tikzpicture}
    \qquad \text{and} \qquad
    \begin{tikzpicture}[anchorbase]
      \draw (-0.25,1) rectangle (-1.25,.5) node[midway] {$m$};
      \draw (0.25,1) rectangle (1.25,.5) node[midway] {$n$};
      \draw (-0.25,-1) rectangle (-1.25,-.5) node[midway] {$n$};
      \draw (0.25,-1) rectangle (1.25,-.5) node[midway] {$m$};
      \draw[->,dashed] (-0.75,-0.5) .. controls (-0.75,0) and (0.75,0) .. (0.75,0.5);
      \draw[->,dashed] (0.75,-0.5) .. controls (0.75,0) and (-0.75,0) .. (-0.75,0.5);
    \end{tikzpicture}
  \]
  are mutually inverse, giving the first isomorphism in \eqref{eq:Qpm-commute}.  The second is similar, simply reversing orientations of strands.

  The isomorphism \eqref{eq:key-isom} follows immediately from Proposition~\ref{prop:key-id-decomp}, after noting that $|\bB_k|$ is the dimension of the $k$-th symmetric power of a vector space of dimension $d$, and is therefore equal to ${d+k-1 \choose k}$.
\end{proof}

Let $K_0(\cH^\lambda)$ be the split Grothendieck group of $\cH^\lambda$.  The monoidal structure on $\cH^\lambda$ endows $K_0(\cH^\lambda)$ with the structure of a ring.

\begin{theo} \label{theo:main}
  We have an injective ring homomorphism
  \begin{equation} \label{eq:h-to-KH}
    \fh_d \hookrightarrow K_0(\cH^\lambda),\quad s_\mu^\pm \mapsto \sQ_\pm^\mu,\quad \mu \text{ a partition},
  \end{equation}
  where $s_\mu^\pm$ denotes the Schur function in $\Sym^\pm$ corresponding to the partition $\mu$.  In particular, this ring homomorphism maps $h_n^\pm$ to $\sQ_\pm^{(n)}$ for $n \in \N_+$.
\end{theo}

\begin{proof}
  Recall that $s_{(n)}^\pm = h_n^\pm$.  Comparing \eqref{eq:heisenberg-presentation} and Proposition~\ref{prop:key-isom}, we see that
  \[
    \fh_d \to K_0(\cH^\lambda),\quad h_n^\pm \mapsto \sQ_\pm^{(n)},\quad n \in \N_+.
  \]
  is a well-defined ring homomorphism.  We will show in Proposition~\ref{prop:h-KH-injective} that it is injective.  An argument analogous to the one in the proof of \cite[Th.~4.5]{LRS18}, based on the fact that the expression for $\sQ_\pm^{\mu}$ as a linear combination of products of $\sQ_\pm^{(n)}$ is given by the Giambelli rules in the ring of symmetric functions for expressing the Schur functions in terms of the complete homogeneous symmetric functions, then shows that \eqref{eq:h-to-KH} holds for all partitions $\mu$.
\end{proof}

\begin{conj} \label{conj:categorification}
  The map \eqref{eq:h-to-KH} is an isomorphism.
\end{conj}

If $d=1$, Conjecture~\ref{conj:categorification} reduces to \cite[Conj.~1]{Kho14}.  In \cite[Th.~10.5]{RS17}, it is possible to prove the analogue of Conjecture~\ref{conj:categorification} because of the presence of a nontrivial grading coming from a Frobenius algebra.  This is also the case in \cite[Th.~1]{CL12}.  The difficulty in the setting of the current paper is that we only have a filtration, and not a grading.

%
\section{Degenerate cyclotomic Hecke algebras} \label{sec:DCHA}
%

Our next goal is to define an action of the category $\cH^\lambda$ on categories of modules for degenerate cyclotomic Hecke algebras.  In this section, we recall these algebras and prove some results concerning them that we will need to show our action is well defined.  The action will then be defined in Section~\ref{sec:action}.

Throughout this section, $\kk$ is an arbitrary commutative ring.  We will cite \cite{Kle05} for various basic results.  Even though that reference works over a field of characteristic zero, one can check that, for all of the results we cite, the proofs go through in the more general setting of an arbitrary commutative ring.

\subsection{Definitions}

\begin{defin}[Degenerate affine Hecke algebra $H_n$] \label{def:degAHA}
  Let $n \in \N$. The \emph{degenerate affine Hecke algebra $H_n$} is the $\kk$-algebra generated by the elements $s_1,\dotsc, s_{n-1}$ and $x_1,\dotsc, x_n$, subject to the relations
    \begin{align}
      \label{eq:deg1} s_i^2 &= 1, \\
      \label{eq:deg2} s_is_j & =s_js_i, \quad |i-j|>1, \\
      \label{eq:deg3} s_is_{i+1}s_i &= s_{i+1}s_is_{i+1}, \\
      \label{eq:deg4} s_jx_i &= x_is_j,\quad i\ne j,j+1, \\
      \label{eq:deg5} s_ix_i &= x_{i+1}s_i-1, \\
      \label{eq:deg6} x_ix_j &= x_jx_i.
    \end{align}
  By convention, $H_0=\kk$ and $H_1=\kk[x_1]$.
\end{defin}

Any result for $H_n$ in the rest of the paper is understood to hold for all $n \in \N$ unless explicitly stated otherwise.

Note that in $H_n$ we also have, for $1\leq i\leq n$ and $t\geq 1$:
\begin{align}
  \label{eq:deg7} s_ix_{i+1} &= x_is_i+1, \\
  \label{eq:deg8} s_ix_i^t &= x_{i+1}^ts_i - \sum_{a+b=t-1} x_i^ax_{i+1}^b, \\
  \label{eq:deg9} s_ix_{i+1}^t &= x_{i}^ts_i + \sum_{a+b=t-1}x_i^ax_{i+1}^b.
\end{align}

Recall the weight $\lambda$ as defined at the beginning of Section~\ref{sec:diagrammatics}.

\begin{defin}[Cyclotomic quotient $H_n^\lambda$] \label{def:cyclotomic-quotient}
  Let $I_n^\lambda$ be the {\em cyclotomic ideal} of $H_n$, which is the ideal generated by $\prod_{i \in I} (x_1-i)^{\lambda_i}$. The \emph{cyclotomic quotient} of $H_n$ corresponding to $\lambda$ is defined to be
  \[
    H_n^{\lambda}:=H_n/I_n^\lambda.
  \]
By convention, $I_0^\lambda=\{0\}$ and so $H_0^\lambda=H_0$.
\end{defin}

One can define a cyclotomic quotient $H_n^f$ for any monic polynomial $f \in \kk[x_1,\dotsc,x_n]$ by replacing $I_n^\lambda$ in Definition~\ref{def:cyclotomic-quotient} with the ideal generated by $f$.  However, as explained in \cite[\S7.1]{Kle05}, it is sufficient to consider the cases where $f = \prod_{i \in I} (x_1-i)^{\lambda_i}$ for some $\lambda \in P_+$.  For this reason, we focus on such quotients in the current paper.

Note that $H_n \subseteq H_{n+1}$, for any $n \in \N$. Composing with the projection onto $H_{n+1}^\lambda$ gives a homomorphism of algebras $H_n\to H_{n+1}^\lambda$, whose kernel is exactly $I_n^\lambda$.  Therefore, we get an embedding of $H_n^\lambda$ into $H_{n+1}^\lambda$.  Via this embedding, we will view $H_n^\lambda$ as a subalgebra of $H_{n+1}^\lambda$ from now on.

\begin{lem}[{\cite[Lem.~7.6.1(i)]{Kle05}}]
 We have that $H_{n+1}^{\lambda}$ is a free left $H_n^{\lambda}$-module with basis
 \begin{equation} \label{eq:left-module-basis}
  \{s_n \dotsm s_jx_j^a \mid 0\leq a< d,\ 1\leq j\leq n+1\},
 \end{equation}
 and a free right $H_n^\lambda$-module with basis
 \begin{equation} \label{eq:right-module-basis}
  \{x_j^a s_j \dotsm s_n \mid 0\leq a< d,\ 1\leq j\leq n+1\}.
 \end{equation}
\end{lem}
By convention, $s_n\cdots s_j x_j^a$ and $x_j^a s_j \dotsm s_n$ are equal to $x_{n+1}^a$ when $j=n+1$.

\begin{lem}[{\cite[Lem.~7.6.1(ii,iii)]{Kle05}}] \label{lem:decomp(n+1)}
As $(H_{n}^{\lambda},H_{n}^{\lambda})$-bimodules, we have
  \begin{gather} \label{eq:decomp1}
    H_{n+1}^{\lambda}=
    H_n^{\lambda}s_n H_n^{\lambda}\oplus \bigoplus_{j=0}^{d-1}H_{n}^{\lambda}x_{n+1}^j,
    \\ \label{eq:decomp2}
    H_n^{\lambda}s_n H_n^{\lambda}\cong H_n^{\lambda}\otimes_{H_{n-1}^{\lambda}} H_n^{\lambda},\quad u s_n v \mapsto u \otimes v,
    \\ \label{eq:decomp3}
    H_n^{\lambda}x_{n+1}^a\cong H_n^{\lambda},\quad u x_{n+1}^a \mapsto u,\quad 0 \le a < d.
  \end{gather}
\end{lem}

\subsection{The trace map}

\begin{lem}[{\cite[Lem.~7.7.2]{Kle05}}] \label{lem:Frob}
  The algebra $H_{n+1}^{\lambda}$ is a Frobenius extension of $H_n^{\lambda}$, with the nondegenerate trace
  \[
    \tr_{n+1} \colon H_{n+1}^\lambda\to H_n^\lambda
  \]
  being defined by composing the projection onto $H_n^\lambda x_{n+1}^{d-1}$ in \eqref{eq:decomp1} with the isomorphism $H_n^\lambda x_{n+1}^{d-1} \cong H_n^\lambda$ in \eqref{eq:decomp3}.
\end{lem}

At the end of this section we will show that the trace is not symmetric.  Diagrammatically, the map $\tr_{n+1}$ corresponds to the operation of taking a diagram on $n$ upward strands and closing off the leftmost strand to the left:
\[
  \begin{tikzpicture}[anchorbase]
    \draw[dashed] (-0.75,-0.25) rectangle (0.75,0.25);
    \draw[->] (-0.6,0.25) -- (-0.6,0.75);
    \draw[->] (-0.4,0.25) -- (-0.4,0.75);
    \draw[->] (-0.2,0.25) -- (-0.2,0.75);
    \draw[->] (0,0.25) -- (0,0.75);
    \draw[->] (0.2,0.25) -- (0.2,0.75);
    \draw[->] (0.4,0.25) -- (0.4,0.75);
    \draw[->] (0.6,0.25) -- (0.6,0.75);
    \draw[<-] (-0.6,-0.25) -- (-0.6,-0.75);
    \draw[<-] (-0.4,-0.25) -- (-0.4,-0.75);
    \draw[<-] (-0.2,-0.25) -- (-0.2,-0.75);
    \draw[<-] (0,-0.25) -- (0,-0.75);
    \draw[<-] (0.2,-0.25) -- (0.2,-0.75);
    \draw[<-] (0.4,-0.25) -- (0.4,-0.75);
    \draw[<-] (0.6,-0.25) -- (0.6,-0.75);
  \end{tikzpicture}
  \quad \mapsto \quad
  \begin{tikzpicture}[anchorbase]
    \draw[dashed] (-0.75,-0.25) rectangle (0.75,0.25);
    \draw[->] (-0.6,0.25) -- (-0.6,0.35) arc (0:180:0.2) -- (-1,-0.35) arc (180:360:0.2) -- (-0.6,-0.25);
    \draw[->] (-0.4,0.25) -- (-0.4,0.75);
    \draw[->] (-0.2,0.25) -- (-0.2,0.75);
    \draw[->] (0,0.25) -- (0,0.75);
    \draw[->] (0.2,0.25) -- (0.2,0.75);
    \draw[->] (0.4,0.25) -- (0.4,0.75);
    \draw[->] (0.6,0.25) -- (0.6,0.75);
    \draw[<-] (-0.4,-0.25) -- (-0.4,-0.75);
    \draw[<-] (-0.2,-0.25) -- (-0.2,-0.75);
    \draw[<-] (0,-0.25) -- (0,-0.75);
    \draw[<-] (0.2,-0.25) -- (0.2,-0.75);
    \draw[<-] (0.4,-0.25) -- (0.4,-0.75);
    \draw[<-] (0.6,-0.25) -- (0.6,-0.75);
  \end{tikzpicture}
\]

\begin{lem} \label{lem:Nakayama-fix-x}
  We have
  \[
    \tr_{n+1}(x_{n+1}z) = \tr_{n+1}(z x_{n+1})
    \quad \text{for all } z \in H_{n+1}^\lambda.
  \]
\end{lem}

\begin{proof}
  Fix $z \in H_{n+1}^\lambda$.  By \eqref{eq:decomp1}, we have
  \[
    z = \sum_{j=0}^\ell h'_j s_n h''_j + \sum_{k=0}^{d-1} h_k x_{n+1}^k,
  \]
  for some $h'_j, h''_j, h_k \in H_n^\lambda$, $\ell \in \N$.  Then, since $x_{n+1}$ commutes with $H_n^\lambda$, we have
  \begin{multline*}
    \tr_{n+1} (x_{n+1} z) - \tr_{n+1} (z x_{n+1})
    = \tr_{n+1} (x_{n+1} z - z x_{n+1}) \\
    = \tr_{n+1} \left( \sum_{j=0}^\ell h'_j (x_{n+1} s_n - s_n x_{n+1}) h''_j \right)
    \stackrel{\substack{\eqref{eq:deg5} \\ \eqref{eq:deg7}}}{=} \tr_{n+1} \left( \sum_{j=0}^\ell h'_j (s_n x_n - x_n s_n) h''_j \right)
    = 0,
  \end{multline*}
  where the last equality follows from the fact that $h'_j (s_n x_n - x_n s_n) h''_j \subseteq H_n^\lambda s_n H_n^\lambda$ for all $j$.
\end{proof}

As noted in \cite[Cor.~7.7.4]{Kle05}, we can use Lemma~\ref{lem:Frob} recursively to see that $H_{n+1}^{\lambda}$ is a Frobenius algebra over $\kk$, with nondegenerate trace being defined by composing the projection onto $\kk x_1^{d-1} \dotsm x_{n+1}^{d-1}$ with the isomorphism of vector spaces $\kk x_1^{d-1} \dotsm x_{n+1}^{d-1}\cong \kk$. However, we will not need this fact here.

The trace induces an $H_n^{\lambda}$-valued bilinear form on $H_{n+1}^{\lambda}$, defined by
\begin{equation} \label{eq:bilinear-form}
  \langle u,v\rangle_{n+1}:=\tr_{n+1}(uv).
\end{equation}
Our next goal is to find a basis for $H_{n+1}^\lambda$ as an $H_n^\lambda$-module that is left dual to the basis \eqref{eq:left-module-basis} with respect to the form \eqref{eq:bilinear-form}.  For $n \in \N_+$ and $k \in \{0,\dotsc,d-1\}$, define
\begin{equation} \label{eq:y-def}
  y_{n,k} := \sum_{t=k}^{d-1} (-1)^{d-1-t} x_n^{t-k} \det \big( \tr_n ( x_n^{d+j-i} ) \big)_{i,j = 1,\dotsc,d-1-t},
\end{equation}
where we take the determinant above to be equal to one when $t=d-1$.  Under the action to be defined in Section~\ref{sec:action}, the elements $y_{n,k}$ will correspond to the dual dots in the category $\cH^\lambda$.

\begin{lem}
  The set
  \begin{equation}
    \{s_n \dotsm s_i y_{i,k} \mid i=1,\dotsc,n+1,\ k=0,\dotsc,d-1\}
  \end{equation}
  is a basis of $H_{n+1}^\lambda$ as a left $H_n^\lambda$-module.  Here we adopt the convention that, for $i=n+1$, the notation $s_n \dotsm s_i y_{i,k}$ means $y_{n+1,k}$ for $k=0,\dotsc,d-1$.
\end{lem}

\begin{proof}
  This follows from the fact that the transition matrix between the basis \eqref{eq:left-module-basis} and the set
  \[
    \{s_n \dotsm s_i y_{i,d-1-k} \mid i=1,\dotsc,n+1,\ k=0,\dotsc,d-1\}
  \]
  is triangular unipotent by \eqref{eq:y-def}.
\end{proof}

\begin{lem} \label{lem:zerotrace}
  We have
  \[
    \tr_{n+1}(gs_n)=\tr_{n+1}(s_ng)=0
    \quad \text{for all} \quad
    g \in \bigoplus_{m=0}^{d-1} H_n^\lambda x_{n+1}^m.
  \]
\end{lem}

\begin{proof}
  Let $g=\sum_{m=0}^{d-1}a_mx_{n+1}^m$, for some $a_m\in H_n^{\lambda}$.  We prove $\tr_{n+1}(s_ng)=0$, the proof of $\tr_{n+1}(gs_n)=0$  being similar.  Note that $x_{n+1}$ commutes with the $a_m$, $m \in \{0,\dotsc,m-1\}$, because the latter all belong to $H_n^{\lambda}$ by assumption. So we have
  \begin{multline*}
    s_n g
    = \sum_{m=0}^{d-1}s_na_mx_{n+1}^m
    = \sum_{m=0}^{d-1}s_nx_{n+1}^ma_m
    \stackrel{\eqref{eq:deg9}}{=} \sum_{m=0}^{d-1} \left( x_n^m s_n a_m + \sum_{a+b=m-1} x_n^a x_{n+1}^b a_m \right) \\
    = \sum_{m=0}^{d-1} \left( x_n^m s_n a_m + \sum_{a+b=m-1} x_n^a a_m x_{n+1}^b \right)
    \in H_n^\lambda s_n H_n^\lambda\oplus\bigoplus_{j=0}^{d-2}H_n^\lambda x_{n+1}^j.
  \end{multline*}
  Therefore the projection onto $H_n^\lambda x_{n+1}^{d-1}$ is zero and the result follows.
\end{proof}

\begin{lem} \label{lem:inclusions}
  We have
  \begin{gather*}
    s_n H_{n-1}^\lambda s_{n-1} H_{n-1}^\lambda s_n\subseteq H_n^\lambda s_n H_n^\lambda,\quad \text{and}
    \\
    s_n H_{n-1}^\lambda x_n^as_n \subseteq H_n^\lambda s_n H_n^\lambda \oplus \bigoplus_{0\leq b\leq a}  H_n^\lambda x_{n+1}^b,\quad 0\leq a\leq d-1.
  \end{gather*}
\end{lem}

\begin{proof}
  Since $s_n$ commutes with $H_{n-1}^\lambda$, the first inclusion follows from the braid relation \eqref{eq:deg3}. The second inclusion follows from the fact that $H_{n-1}^\lambda x_n^a s_n = x_n^a s_n H_{n-1}^\lambda$, and that
  \[
    s_nx_n^as_n
    = x_{n+1}^a-\sum_{u+v=a-1} x_n^ux_{n+1}^vs_n
    = x_{n+1}^a-\sum_{u+v=a-1} \left(x_n^us_nx_{n}^v+\sum_{s+t=v-1} x_n^{u+s}x_{n+1}^t\right).
  \]
  by using \eqref{eq:deg8} twice.
\end{proof}

\begin{cor} \label{cor:subseqtraces}
  For any $y\in H_n^{\lambda}$ we have
  \[
    \tr_{n+1}(s_nys_n)=\tr_n(y).
  \]
\end{cor}

\begin{proof}
  If $\tr_n(y)=0$, then $y\in H_{n-1}^\lambda s_{n-1}H_{n-1}^\lambda\oplus\bigoplus_{j=0}^{d-2}H_{n-1}^\lambda x_n^j$.  Therefore, Lemma~\ref{lem:inclusions} implies

  \[
    s_n y s_n \in H_n^\lambda s_n H_n^\lambda \oplus \bigoplus_{j=0}^{d-2} H_n^\lambda x_{n+1}^b,
  \]
  so $\tr_{n+1}(s_nys_n)=0$.

  The only remaining case to prove is when $y=y'x_{n}^{d-1}$, for some $y '\in H_{n-1}^\lambda$.
  In this case we have
  \begin{align*}
    s_n y s_n
    = s_n y' x_{n}^{d-1} s_n
    &\stackrel{\eqref{eq:deg9}}{=} y'x_{n+1}^{d-1} - \sum_{a+b=d-2} y' s_n x_n^a x_{n+1}^b \\
    &\stackrel{\eqref{eq:deg9}}{=} y'x_{n+1}^{d-1} - \sum_{a+b=d-2} \left( y' x_n^b s_n x_{n}^a + \sum_{s+t=b-1} y'x_n^{a+s} x_{n+1}^t \right).
  \end{align*}
  Therefore,
  \[
    \tr_{n}(y)=y'=\tr_{n+1}(s_nys_n). \qedhere
  \]
\end{proof}

\begin{prop} \label{prop:dual}
  The basis
  \[
    \{s_n \dotsm s_i y_{i,a} \mid i=1,\dotsc, n+1,\ a=0,\dotsc, d-1\}
  \]
  is left dual to the basis \eqref{eq:right-module-basis} with respect to $\langle -,- \rangle_{n+1}$.  More precisely,
  \[
    \langle s_n \dotsm s_i y_{i,a}, x_j^bs_j\dotsm s_n\rangle_{n+1}
    = \delta_{i,j} \delta_{a,b}, \quad i,j \in \{1,\dotsc,n+1\},\quad a,b \in \{0,\dotsc,d-1\}.
  \]
\end{prop}

\begin{proof}
  We first show that
  \begin{equation} \label{eq:xy-inner-prod}
    \langle y_{n+1,a},x_{n+1}^b \rangle_{n+1} = \delta_{a,b}
    \quad \text{for all } n \in \N,\ a,b \in \{0,\dotsc,d-1\}.
  \end{equation}
  We have
  \[
    \langle y_{n+1,a}, x_{n+1}^b \rangle_{n+1}
    = \sum_{t=a}^{d-1} (-1)^{d-1-t} \tr_{n+1} (x_{n+1}^{t+b-a}) \det \big( \tr_{n+1} ( x_{n+1}^{d+j-i} ) \big)_{i,j = 1,\dotsc,d-1-t}.
  \]
  If $b\leq a$, then $\tr_{n+1}(x_{n+1}^{t+b-a}) = 0$ unless $a=b$ and $t = d-1$, and so \eqref{eq:xy-inner-prod} follows immediately.  Now suppose $b>a$.  Setting $m=b-a$ and $s=d-1-t$ in the above sum, it suffices to prove that
  \[
    \sum_{s=0}^{d-1-a} (-1)^s \tr_{n+1} (x_{n+1}^{d-1+m-s}) \det \big( \tr_{n+1}(x_{n+1}^{d+j-i} ) \big)_{i,j=1,\dotsc,s} = 0.
  \]
  By~\eqref{eq:decomp1} and the definition of the trace in Lemma~\ref{lem:Frob}, the terms with $s > m$ are equal to zero. Since $b \le d-1$, we have $m \le d-1-a$.  Therefore, it is enough to prove
  \begin{equation} \label{eq:xy-dual-sum}
    \sum_{s=0}^m (-1)^s \tr_{n+1} (x_{n+1}^{d-1+m-s}) \det \big( \tr_{n+1}(x_{n+1}^{d+j-i} ) \big)_{i,j=1,\dotsc,s} = 0.
  \end{equation}
  But \eqref{eq:xy-dual-sum} follows from Lemma~\ref{lem:det-expansion} with $a_k = \tr_{n+1} (x_{n+1}^{d+k})$, considered as elements of the center of $H_n^\lambda$.  This completes the proof of \eqref{eq:xy-inner-prod}.

 The proof of the proposition will now proceed by induction with respect to $n \in \N$.  The $n=0$ case follows immediately from \eqref{eq:xy-inner-prod}.  Now suppose that $\{s_{n-1} \dotsm s_i y_{i,a} \mid i=1,\dotsc, n,\ a=0,\dotsc, d-1\}$ is left dual to $\{x_i^as_i\dotsm s_{n-1} \mid i=1,\dotsc, n,\ a=0,\dotsc, d-1\}$ with respect to $\langle -,- \rangle_n$, and consider the bases $\{s_{n} \dotsm s_i y_{i,a} \mid i=1,\dotsc, n+1,\ a=0,\dotsc, d-1\}$ and $\{x_i^a s_i\dotsm s_{n} \mid i=1,\dotsc, n+1,\ a=0,\dotsc, d-1\}$.

  By Corollary~\ref{cor:subseqtraces} we have
  \[
    \langle s_{n} \dotsm s_i y_{i,a}, x_j^bs_j\dotsm s_{n}\rangle_{n+1}
    = \langle s_{n-1} \dotsm s_i y_{i,a}, x_j^b s_j \dotsm s_{n-1} \rangle_n,
  \]
  for $1\leq i,j\leq n$. By induction the latter is equal to $\delta_{i,j}\delta_{a,b}$.  It remains to show the equalities
  \[
    \langle y_{n+1,a},x_j^b s_j \dotsm s_{n}\rangle_{n+1}=0
    \quad\text{and}\quad
    \langle s_{n}\dotsm s_j y_{j,a},x_{n+1}^b\rangle_{n+1}=0,
    \quad 1\leq j\leq n,\ 0 \le a,b \le d-1.
  \]
  The first one follows from
  \[
    \langle y_{n+1,a}, x_j^bs_j \dotsm s_{n}\rangle_{n+1}
    = \tr_{n+1}(y_{n+1,a}x_j^bs_j \dotsm s_{n})
    = 0,
  \]
  where the last equality holds by Lemma~\ref{lem:zerotrace}, with $g = y_{n+1,a} x_j^b s_j \dotsm s_{n-1}$. Similarly, we have
  \[
    \langle s_{n} \dotsm s_j y_{j,a},x_{n+1}^b\rangle_{n+1}
    = \tr_{n+1}(s_{n} \dotsm s_j y_{j,a} x_{n+1}^b)
    = 0,
  \]
  where the last equality holds by Lemma~\ref{lem:zerotrace}, with $g = s_{n-1} \dotsm s_j y_{j,a} x_{n+1}^b$.
\end{proof}

\begin{lem} \label{lem:tr_x^d}
  For $n \ge 1$, we have
  \[
    \tr_n (x_n^d) = \sum_{i \in I} i \lambda_i.
  \]
\end{lem}

\begin{proof}
  We prove the result by induction on $n$.  For $n=1$, note that the equation $\prod_i (x_1-i)^\lambda_i = 0$ allows us to write $x_1^d$ as a polynomial in $x_1$ of degree at most $d-1$.  Then $\tr_1(x_1^d)$ is the coefficient of $x_1^{d-1}$ in this polynomial.  This coefficient is clearly $\sum_{i \in I} i \lambda_i$.

  Now, for $n \ge 1$, by Corollary~\ref{cor:subseqtraces} and \eqref{eq:deg8}, we have
  \begin{align*}
    \tr_n(x_n^d)
    &= \tr_{n+1}(s_n x_n^d s_n) \\
    &= \tr_{n+1} \left( x_{n+1}^d - \sum_{a=0}^{d-1} x_n^{d-1-a} x_{n+1}^a s_n \right) \\
    &= \tr_{n+1} \left( x_{n+1}^d - \sum_{a=0}^{d-1} x_n^{d-1-a} \left( s_n x_n^a + \sum_{b=0}^{a-1} x_n^{a-1-b} x_{n+1}^b \right) \right) \\
    &= \tr_{n+1} \left( x_{n+1}^d \right) - \sum_{a=0}^{d-1} \tr_{n+1} \left( x_n^{d-1-a} s_n x_n^a \right) - \sum_{a=0}^{d-1} \sum_{b=0}^{a-1} \tr_{n+1} \left( x_n^{d-b-2} x_{n+1}^b \right) \\
    &= \tr_{n+1} \left( x_{n+1}^d \right).
  \end{align*}
  This completes the proof of the inductive step.
\end{proof}

Note that the trace is not symmetric in general.  For example, by \eqref{eq:deg8}, \eqref{eq:deg9}, and the fact that the trace is an $(H_n^\lambda,H_n^\lambda)$-bimodule map, we get
\[
  \tr_{n+1} \left( s_nx_{n+1}^{d+1} s_{n-1} \right)
  = \left( x_n + \sum_{i \in I} i \lambda_i \right) s_{n-1}
\]
and
\[
 \tr_{n+1} \left( x_{n+1}^{d+1} s_{n-1}s_n \right)
  = \tr_{n+1} \left( s_{n-1}x_{n+1}^{d+1} s_n \right)
  = s_{n-1} \left( x_n + \sum_{i \in I} i \lambda_i \right).
\]
\details{
  Indeed, we have
  \begin{align*}
    \tr_{n+1} \left( s_nx_{n+1}^{d+1} s_{n-1} \right)
    &= \tr_{n+1} \left( s_nx_{n+1}^{d+1} \right) s_{n-1} \\
    &\stackrel{\eqref{eq:deg9}}{=} \tr_{n+1} \left( x_n^{d+1} s_n + \sum_{a+b=d} x_n^a x_{n+1}^b \right) s_{n-1}\\
    & =  \tr_{n+1} \left( x_{n+1}^d + x_n x_{n+1}^{d-1} \right) s_{n-1}\\
    & = \left( \sum_{i \in I} i \lambda_i + x_n \right) s_{n-1},
  \end{align*}
  and
  \begin{align*}
    \tr_{n+1} \left( s_{n-1}x_{n+1}^{d+1} s_n \right)
    &= s_{n-1} \tr_{n+1} \left( x_{n+1}^{d+1} s_n \right) \\
    &\stackrel{\eqref{eq:deg8}}{=} s_{n-1} \tr_{n+1} \left( s_n x_n^{d+1} + \sum_{a+b=d} x_n^a x_{n+1}^b \right) \\
    &= s_{n-1} \tr_{n+1} \left( x_{n+1}^d + x_n x_{n+1}^{d-1} \right) \\
    &= s_{n-1} \left( \sum_{i \in I} i \lambda_i + x_n \right).
  \end{align*}
}
But
\[
  \left( x_n + \sum_{i \in I} i \lambda_i \right) s_{n-1}
  \ne s_{n-1} \left( x_n + \sum_{i \in I} i \lambda_i \right)
\]
because $x_n s_{n-1} = s_{n-1}x_{n-1} +1\ne s_{n-1} x_n$. If $s_{n-1}x_{n-1} +1 = s_{n-1} x_n$ were true,
multiplying all terms by $s_{n-1}$ would give $x_{n-1} + s_{n-1} = x_n$, which contradicts~\eqref{eq:decomp1}.  Note that the full trace map
\[
  \tr_1 \circ \tr_2 \circ \dotsb \circ \tr_n \colon H_n^\lambda \to \kk
\]
\emph{is} symmetric.  See \cite[Th.~A.2]{BK08}.

\subsection{Biadjointness of induction and restriction}

We use the notation in~\cite{Kho14,LS13} for bimodules, so ${}_n(n+1)$ denotes $H_{n+1}^{\lambda}$, viewed as an $(H_{n}^{\lambda},H_{n+1}^{\lambda})$-bimodule, and $(n+1)_n$ denotes $H_{n+1}^{\lambda}$ viewed as an $(H_{n+1}^{\lambda},H_{n}^{\lambda})$-bimodule.  We use juxtaposition to denote the tensor product, so that, for example, $(n+1)_n(n+1)$ is the $(H_{n+1}^{\lambda},H_{n+1}^{\lambda})$-bimodule $H_{n+1}^{\lambda} \otimes_{H_{n}^{\lambda}} H_{n+1}^{\lambda}$.

\begin{prop} \label{prop:adjunction-maps}
  The maps
  \begin{gather*}
    \varepsilon_\rR \colon (n+1)_n(n+1) \to (n+1),\quad \varepsilon_\rR(a \otimes b) = ab,\quad a \in (n+1)_n,\ b \in {_n}{(n+1)}{}, \\
    \eta_\rR \colon (n) \hookrightarrow {_n}(n+1)_n,\quad \eta_{\rR}(a) = a,\quad a \in (n), \\
    \varepsilon_\rL = \tr_{n+1} \colon {_n}(n+1)_n \to (n),\\
    \eta_\rL \colon (n+1) \to (n+1)_n (n+1),\\
    \eta_\rL(a) = a \sum_{\substack{i=1,\dotsc,n+1 \\ k=0,\dotsc,d-1}} x_i^k s_i \dotsm s_{n-1} s_n \otimes s_n s_{n-1} \dotsm s_i y_{i,k},\quad a \in (n+1),
  \end{gather*}
  are bimodule homomorphisms and satisfy the relations
  \begin{gather}
    (\varepsilon_\rR \otimes \id) \circ (\id \otimes \eta_\rR) = \id, \quad
    (\id \otimes \varepsilon_\rR) \circ (\eta_\rR \otimes \id) = \id, \label{eq:right-unit-counit} \\
    (\varepsilon_\rL \otimes \id) \circ (\id \otimes \eta_\rL) = \id, \quad
    (\id \otimes \varepsilon_\rL) \circ (\eta_\rL \otimes \id) = \id. \label{eq:left-unit-counit}
  \end{gather}
  In particular, $(n+1)_n$ is both left and right adjoint to ${_n}(n+1)$ in the 2-category of bimodules over rings.
\end{prop}

\begin{proof}
  The map $\varepsilon_\rR$ is multiplication and $\eta_\rR$ is the inclusion $H_n^\lambda \hookrightarrow H_{n+1}^\lambda$.  Thus, both are clearly bimodule homomorphisms.  The relations \eqref{eq:right-unit-counit} are standard.  They are the usual unit-counit equations making induction left adjoint to restriction.
  \details{
    For $a \in (n+1)_n$, $b \in \prescript{}{n}(n+1)$, we have
    \begin{gather*}
      (\varepsilon_\rR \otimes \id) \circ (\id \otimes \eta_\rR) (a) = (\varepsilon_\rR \otimes \id) \circ (\id \otimes \eta_\rR) (a \otimes 1_n) =  (\varepsilon_\rR \otimes \id)(a \otimes 1_{n+1}) = a, \\
      (\id \otimes \varepsilon_\rR) \circ (\eta_\rR \otimes \id) (b) = (\id \otimes \varepsilon_\rR) \circ (\eta_\rR \otimes \id) (1_n \otimes b) = (\id \otimes \varepsilon_\rR)(1_{n+1} \otimes b) = b.
    \end{gather*}
  }
  Similarly, that the maps $\varepsilon_\rL$ and $\eta_\rL$ are bimodule homomorphisms and satisfy equations \eqref{eq:left-unit-counit} follows from the standard proof that induction is right adjoint to restriction for Frobenius extensions.  See, for example, \cite[\S1.3]{Kad99}.
\end{proof}

%
\section{Action on categories of modules for degenerate cyclotomic Hecke algebras} \label{sec:action}
%

In this section, $\kk$ is an arbitrary commutative ring until Section~\ref{subsec:Fock-space}, where we assume it is a field of characteristic zero.

\subsection{Definition of the action functors}

\begin{defin}[The bimodule categories $\cB^\lambda_n$]
  For $n \in \N$, we let $\cB^\lambda_n$ be the category defined as follows.  The objects of $\cB^\lambda_n$ are direct sums of direct summands of bimodules of the form
  \begin{equation} \label{eq:general-bimodule}
    (n_k)_{m_k}(n_{k-1})_{m_{k-1}} \cdots{}_{m_2} (n_1)_{m_1} (n)
  \end{equation}
  with $k \in \N$ and $\{m_i,m_{i-1}\} = \{n_{i-1},n_{i-1}-1\}$ for all $i=1,\dotsc,k+1$, where we adopt the convention thats $n_0=m_0=n$ and $m_{k+1}=n_k$.  (When the $n_k$ do not agree, direct sums of bimodules are formal.)  Morphisms in $\cB^\lambda_n$ are bimodule homomorphisms (where the only morphism between bimodules of the form \eqref{eq:general-bimodule} for distinct $n_k$ is the zero morphism).
\end{defin}

For each $n \in \N$, we will define a functor $\bF_n \colon \tcH^\lambda \to \cB^\lambda_n$.  We define the functor $\bF_n$ on objects as follows.  The object $\sQ_+$ is sent to the bimodule $(k+1)_k$, and the object $\sQ_-$ is sent to the bimodule $\prescript{}{k}(k+1)$, where $k$ is uniquely determined by the fact that our functor should respect the tensor product and be a functor to the category $\cB^\lambda_n$ (i.e.\ the right action should be by $H_n^\lambda$).  For example,
\[
  \bF_n(\sQ_{++-+--+}) = (n+1)_n(n)_{n-1}(n)_n(n)_{n-1}(n)_n(n+1)_{n+1}(n+1)_n
\]
We define the $\bF_n$ to be zero on any object for which the indices $k$ become negative.  For example, $\bF_1(\sQ_{--})=0$.

We now define the functor on morphisms.  Consider a morphism in $\tcH^\lambda$ consisting of a single planar diagram.  We label the rightmost region of the diagram by $n$.  Then, we label all other regions of the diagram by integers such that, as we move from right to left across the diagram, labels increase by one when we cross upward pointing strands and decrease by one when we cross downward pointing strands.  It is clear that there is a unique way to label the regions of the diagram in this way.  For instance, the following diagram is labelled as indicated.
\[
  \begin{tikzpicture}[>=stealth]
    \draw[->] (0,0) .. controls (0,1) and (1,1) .. (1,2);
    \draw[->] (1,0) .. controls (1,1) and (3,1) .. (3,0);
    \draw[->] (2,0) .. controls (2,1.5) and (-1,1.5) .. (-1,0);
    \draw[->] (2,2) arc(180:360:1);
    \bluedot{(0.93,1.5)} node[anchor=west,color=black] {\dotlabel{2}};
    \bluedot{(-.3,0.97)};
    \draw (7,1) .. controls (7,1.5) and (6.3,1.5) .. (6.1,1);
    \draw (7,1) .. controls (7,.5) and (6.3,.5) .. (6.1,1);
    \draw (6,0) .. controls (6,.5) .. (6.1,1);
    \draw (6.1,1) .. controls (6,1.5) .. (6,2) [->];
    \bluedot{(6,0.3)} node[anchor=west,color=black] {\dotlabel{3}};
    \draw[->] (4.3,1) arc(180:-180:.6);
    \bluedot{(3.88,1.5)};
    \draw (7.5,1) node {\regionlabel{n}};
    \draw (6.6,1) node {\regionlabel{n-1}};
    \draw (3.7,.6) node {\regionlabel{n+1}};
    \draw (4.9,1) node {\regionlabel{n}};
    \draw (3,1.5) node {\regionlabel{n+2}};
    \draw (-.5,1.5) node {\regionlabel{n+2}};
    \draw (-.3,.6) node {\regionlabel{n+3}};
    \draw (.8,.7) node {\regionlabel{n+2}};
    \draw (1.5,.3) node {\regionlabel{n+1}};
    \draw (2.5,.3) node {\regionlabel{n}};
  \end{tikzpicture}
\]
Each diagram is a composition of dots, crossings, cups and caps.  Thus, we define the functors $\bF_n$, $n \in \N$, on such atoms.  Since, on these pieces, the functor $\bF_n$ will be independent of $n$, we will drop the index and describe the functor $\bF = \bigoplus_{n \in \N} \bF_n$.

We define $\bF$ on cups and caps by

\noindent\begin{minipage}{0.5\linewidth}
  \begin{equation} \label{eq:right-cap-cup-maps}
    \begin{tikzpicture}[anchorbase]
      \draw[<-] (0,-.5) arc (0:180:.5);
      \draw (0.5,-0.2) node {\regionlabel{n+1}};
    \end{tikzpicture}
    \quad \mapsto \varepsilon_\rR,
  \end{equation}
\end{minipage}%
\begin{minipage}{0.5\linewidth}
  \begin{equation}
    \begin{tikzpicture}[anchorbase]
      \draw[->] (0,0) arc (180:360:.5);
      \draw (1.3,-0.2) node {\regionlabel{n}};
    \end{tikzpicture}
    \quad \mapsto \eta_\rR,
  \end{equation}
\end{minipage}\par\vspace{\belowdisplayskip}

\noindent\begin{minipage}{0.5\linewidth}
  \begin{equation}
    \begin{tikzpicture}[anchorbase]
      \draw[->] (0,-.5) arc (0:180:.5);
      \draw (0.3,-0.3) node {\regionlabel{n}};
    \end{tikzpicture}
    \quad \mapsto \varepsilon_\rL,
  \end{equation}
\end{minipage}%
\begin{minipage}{0.5\linewidth}
  \begin{equation} \label{eq:left-cap-cup-maps}
    \begin{tikzpicture}[anchorbase]
      \draw[<-] (0,0) arc (180:360:.5);
      \draw (1.4,-0.3) node {\regionlabel{n+1}};
    \end{tikzpicture}
    \quad \mapsto \eta_\rL,
  \end{equation}
\end{minipage}\par\vspace{\belowdisplayskip}

\noindent where $\varepsilon_\rR$, $\eta_\rR$, $\varepsilon_\rL$, and $\eta_\rL$ are the bimodule homomorphisms of Proposition~\ref{prop:adjunction-maps}.

We define $\bF$ on dots by

\noindent\begin{minipage}{0.5\linewidth}
  \begin{equation} \label{eq:dot-maps-up}
    \begin{tikzpicture}[anchorbase]
      \draw[->] (0,0) -- (0,1);
      \bluedot{(0,.5)};
      \draw (0.3,0.2) node {\regionlabel{n}};
    \end{tikzpicture}
    \mapsto \pr{x_{n+1}},
  \end{equation}
\end{minipage}%
\begin{minipage}{0.5\linewidth}
  \begin{equation} \label{eq:dot-maps-down}
    \begin{tikzpicture}[anchorbase]
      \draw[<-] (0,0) -- (0,1);
      \bluedot{(0,.5)};
      \draw (-0.3,0.2) node {\regionlabel{n}};
    \end{tikzpicture}
    \mapsto \pl{x_{n+1}},
  \end{equation}
\end{minipage}\par\vspace{\belowdisplayskip}

\noindent where, for $a \in H_{n+1}^\lambda$, $\pr{a}$ and $\pl{a}$ denote the maps of right and left multiplication by $a$:
\begin{gather*}
  \pr{a} \colon H_{n+1}^\lambda \to H_{n+1}^\lambda,\quad b \mapsto ba,\\
  \pl{a} \colon H_{n+1}^\lambda \to H_{n+1}^\lambda,\quad b \mapsto ab.
\end{gather*}

We define $\bF$ on crossings as follows:

\noindent\begin{minipage}{0.5\linewidth}
  \begin{equation} \label{eq:up-crossing-map}
    \begin{tikzpicture}[anchorbase]
      \draw[->] (0,0) -- (1,1);
      \draw[->] (1,0) -- (0,1);
      \draw (1,0.5) node {\regionlabel{n}};
    \end{tikzpicture}
    \ \mapsto \pr{s_{n+1}},
  \end{equation}
\end{minipage}%
\begin{minipage}{0.5\linewidth}
  \begin{equation}
    \begin{tikzpicture}[anchorbase]
      \draw[<-] (0,0) -- (1,1);
      \draw[<-] (1,0) -- (0,1);
      \draw (1,0.5) node {\regionlabel{n}};
    \end{tikzpicture}
    \quad \mapsto \pl{s_{n-1}},
  \end{equation}
\end{minipage}\par\vspace{\belowdisplayskip}

\begin{equation} \label{eq:right-crossing-map}
  \begin{tikzpicture}[anchorbase]
    \draw[->] (0,0) -- (1,1);
    \draw[<-] (1,0) -- (0,1);
    \draw (1,0.5) node {\regionlabel{n}};
  \end{tikzpicture}
  \ \mapsto
  \Big( (n)_{n-1}(n) \to \prescript{}{n}(n+1)_n,\ a \otimes a' \mapsto a s_n a' \Big),
\end{equation}

\begin{equation} \label{eq:left-crossing-map}
  \begin{tikzpicture}[anchorbase]
    \draw[<-] (0,0) -- (1,1);
    \draw[->] (1,0) -- (0,1);
    \draw (1,0.5) node {\regionlabel{n}};
  \end{tikzpicture}
  \ \mapsto \Big( \prescript{}{n}(n+1)_n \to (n)_{n-1}(n),\quad s_n \mapsto 1_n \otimes 1_n,\ x_{n+1}^j \mapsto 0,\ j=0,\dotsc,d-1 \Big).
\end{equation}

In the remainder of this section, we prove that the functors $\bF_n$ are well-defined.

\subsection{Cyclicity}

We first prove some results that imply that the functors $\bF_n$ respect isotopy invariance.

\begin{lem}
  Under the functor $\bF$, we have

  \noindent\begin{minipage}{0.5\linewidth}
    \begin{equation} \label{eq:dot-vee-maps-up}
      \begin{tikzpicture}[anchorbase]
        \draw[->] (0,0) -- (0,1);
        \bluedot{(0,.5)} node[anchor=east,color=black] {\dotlabel{j^\vee}};
        \draw (0.3,0.2) node {\regionlabel{n}};
      \end{tikzpicture}
      \mapsto \pr{y_{n+1,j}},
    \end{equation}
  \end{minipage}%
  \begin{minipage}{0.5\linewidth}
    \begin{equation} \label{eq:dot-vee-maps-down}
      \begin{tikzpicture}[anchorbase]
        \draw[<-] (0,0) -- (0,1);
        \bluedot{(0,.5)} node[anchor=west,color=black] {\dotlabel{j^\vee}};
        \draw (-0.3,0.2) node {\regionlabel{n}};
      \end{tikzpicture}
      \mapsto \pl{y_{n+1,j}},
    \end{equation}
  \end{minipage}\par\vspace{\belowdisplayskip}

  \noindent for $j \in \{0,\dotsc,d-1\}$.
\end{lem}

\begin{proof}
  To prove \eqref{eq:dot-vee-maps-up}, we compute
  \begin{multline*}
    \bF \left(
      \begin{tikzpicture}[anchorbase]
        \draw[->] (0,0) -- (0,1);
        \bluedot{(0,.5)} node[anchor=east,color=black] {\dotlabel{j^\vee}};
        \draw (0.3,0.2) node {\regionlabel{n}};
      \end{tikzpicture}
    \right)
    \stackrel{\eqref{eq:dot-vee-def}}{=} \sum_{i=j}^{d-1} \pr{x_{n+1}^{i-j}} \bF(c_{d-1-i}) \\
    \stackrel{\eqref{eq:c-def}}{=} \sum_{i=j}^{d-1} (-1)^{d-1-i} \left( \pr{x_{n+1}^{i-j} } \right) \det \left( \tr_{n+1} (x_{n+1}^{d+b-a}) \right)_{a,b=1,\dotsc,d-1-i}
    \stackrel{\eqref{eq:y-def}}{=} \pr{y_{n+1,j}}.
  \end{multline*}
  The proof of \eqref{eq:dot-vee-maps-down} is analogous.
\end{proof}

\begin{lem} \label{lem:dot-cyclicity}
  The images of the diagrams
  
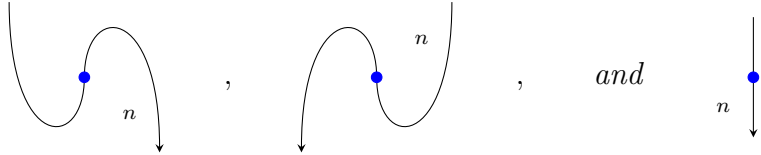
\[
    \begin{tikzpicture}[anchorbase]
      \draw[->] (-1,1) .. controls (-1,-1) and (0,-1) .. (0,0) .. controls (0,1) and (1,1) .. (1,-1);
      \bluedot{(0,0)};
      \draw (0.6,-0.5) node {\regionlabel{n}};
    \end{tikzpicture}
    \qquad, \qquad
    \begin{tikzpicture}[anchorbase]
      \draw[->] (1,1) .. controls (1,-1) and (0,-1) .. (0,0) .. controls (0,1) and (-1,1) .. (-1,-1);
      \bluedot{(0,0)};
      \draw (0.6,0.5) node {\regionlabel{n}};
    \end{tikzpicture}
    \qquad, \qquad \text{and} \qquad
    \begin{tikzpicture}[anchorbase]
      \draw[<-] (0,0) -- (0,1.6);
      \bluedot{(0,.8)};
      \draw (-0.4,0.4) node {\regionlabel{n}};
    \end{tikzpicture}
  \]
  under the functor $\bF$ are equal.
\end{lem}

\begin{proof}
  The image under $\bF$ of the leftmost diagram is the map
  \begin{gather*}
    {}_n(n+1)
    \cong (n)_n(n+1)
    \xrightarrow{\eta_\rR \otimes \id} {}_n(n+1)_n(n+1)
    \xrightarrow{\pr{x_{n+1}} \otimes \id}  {}_n(n+1)_n(n+1)
    \xrightarrow{\varepsilon_\rR} {}_n(n+1),
    \\
    a
    \mapsto 1_n \otimes a
    \mapsto 1_{n+1} \otimes a
    \mapsto x_{n+1} \otimes a
    \mapsto x_{n+1}
    = \pl{x_{n+1}} (a).
  \end{gather*}
  The image under $\bF$ of the second diagram is the map
  \[
    {}_n(n+1)
    \xrightarrow{\eta_\rL} {}_n(n+1)_n(n+1)
    \xrightarrow{\pr{x_{n+1}} \otimes \id} {}_n(n+1)_n(n+1)
    \xrightarrow{\varepsilon_\rL \otimes \id} (n)_n(n+1)
    \cong {}_n(n+1),
  \]
  \begin{align*}
    a &\mapsto \sum_{\substack{i=1,\dotsc,n+1 \\ k=0,\dotsc,d-1}} x_i^k s_i \dotsm s_{n-1} s_n \otimes s_n s_{n-1} \dotsm s_i y_{i,k} a \\
    &\mapsto \sum_{\substack{i=1,\dotsc,n+1 \\ k=0,\dotsc,d-1}} x_i^k s_i \dotsm s_{n-1} s_n x_{n+1} \otimes s_n s_{n-1} \dotsm s_i y_{i,k} a \\
    &\mapsto \sum_{\substack{i=1,\dotsc,n+1 \\ k=0,\dotsc,d-1}} \tr_{n+1} (x_i^k s_i \dotsm s_{n-1} s_n x_{n+1}) s_n s_{n-1} \dotsm s_i y_{i,k} a \\
    &= \sum_{\substack{i=1,\dotsc,n+1 \\ k=0,\dotsc,d-1}} \tr_{n+1} (x_{n+1} x_i^k s_i \dotsm s_{n-1} s_n) s_n s_{n-1} \dotsm s_i y_{i,k} a
    = x_{n+1} a
    = \pl{x_{n+1}} (a),
  \end{align*}
  where we have used Lemma~\ref{lem:Nakayama-fix-x} in the first equality, and Proposition~\ref{prop:dual} in the second equality.
  \details{
    Let $\cB$ be a $\kk$-basis of $H_{n+1}^\lambda$, with left dual basis $\cB^\vee$.  We claim that $\sum_{b \in \cB} \tr_{n+1}(ab)b^\vee = a$ for all $a \in H_{n+1}^\lambda$.  Indeed, write $a = \sum_{b \in \cB} a_b b^\vee$ with $a_b \in \kk$ for all $b \in \cB$.  Then we have
    \[
      \sum_{b \in \cB} \tr_{n+1}(ab)b^\vee
      = \sum_{b,c \in \cB} a_c \tr_{n+1} (c^\vee b) b^\vee
      = \sum_{c \in \cB} a_c c^\vee = a. \qedhere
    \]
  }
\end{proof}

\begin{lem} \label{lem:downcross-invariance}
  The images of the diagrams
  
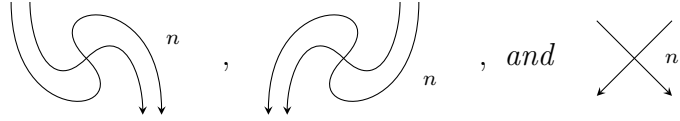
\[
    \begin{tikzpicture}[scale=0.5,>=stealth,baseline={([yshift=-.5ex]current bounding box.center)}]
      \draw[->] (-1.5,1.5) .. controls (-1.5,0.5) and (-1,-1) .. (0,0) .. controls (1,1) and (1.5,-0.5) .. (1.5,-1.5);
      \draw[->] (-2,1.5) .. controls (-2,-2) and (1.5,-1.5) .. (0,0) .. controls (-1.5,1.5) and (2,2) .. (2,-1.5);
      \draw (2.3,0.5) node {\regionlabel{n}};
    \end{tikzpicture}
    \quad, \quad
    \begin{tikzpicture}[scale=0.5,>=stealth,baseline={([yshift=-.5ex]current bounding box.center)}]
      \draw[->] (1.5,1.5) .. controls (1.5,0.5) and (1,-1) .. (0,0) .. controls (-1,1) and (-1.5,-0.5) .. (-1.5,-1.5);
      \draw[->] (2,1.5) .. controls (2,-2) and (-1.5,-1.5) .. (0,0) .. controls (1.5,1.5) and (-2,2) .. (-2,-1.5);
      \draw (2.3,-0.6) node {\regionlabel{n}};
    \end{tikzpicture}
    \quad, \text{ and } \quad
    \begin{tikzpicture}[anchorbase]
      \draw[<-] (0,0) -- (1,1);
      \draw[<-] (1,0) -- (0,1);
      \draw (1,0.5) node {\regionlabel{n}};
    \end{tikzpicture}
  \]
  under the functor $\bF$ are equal.
\end{lem}

\begin{proof}
  The proof of this lemma is completely analogous to that of \cite[Lem.~7.1]{RS17} and is therefore omitted.
  \details{
    The image of the first diagram is the composition
    \[
      \prescript{}{n-2}(n)
      \cong (n-2)_{n-2}(n)
      \xrightarrow{\eta_\rR^2 \otimes \id} \prescript{}{n-2}(n)_{n-2}(n)
      \xrightarrow{\pr{s_{n-1}} \otimes \id} \prescript{}{n-2}(n)_{n-2}(n)
      \xrightarrow{\varepsilon_\rR^2} \prescript{}{n-2}(n).
    \]
    This map is uniquely determined by the image of $1_n$ and we compute
    \[
      1_n
      \mapsto 1_{n-2} \otimes 1_n
      \mapsto 1_n \otimes 1_n
      \mapsto s_{n-1} \otimes 1_n
      \mapsto s_{n-1}
      = \pl{s_{n-1}}(1_n).
    \]
    On the other hand, the image of the second diagram is the composition
    \begin{multline*}
      \prescript{}{n-2}(n)
      \xrightarrow{\eta_\rL} \prescript{}{n-2}(n)_{n-1}(n)
      \cong \prescript{}{n-2}(n)_{n-1}(n-1)_{n-1}(n)
      \xrightarrow{\id \otimes \eta_\rL \otimes \id} \prescript{}{n-2}(n)_{n-1}(n-1)_{n-2}(n-1)_{n-1}(n) \\
      \cong \prescript{}{n-2}(n)_{n-2}(n)
      \xrightarrow{\pr{s_{n-1}} \otimes \id} \prescript{}{n-2}(n)_{n-2}(n)
      \xrightarrow{\varepsilon_\rL^2 \otimes \id} (n-2)_{n-2}(n)
      \cong \prescript{}{n-2}(n).
    \end{multline*}
    This map is also uniquely determined by the image of $1_n$ and we compute
    \begin{align*}
      1_n
      &\mapsto \sum_{\substack{i = 1,\dotsc,n \\ k=0,\dotsc,d-1}} x_i^k s_i \dotsm s_{n-1} \otimes s_{n-1} \dotsm s_i y_{i,k} \\
      &\mapsto \sum_{\substack{i = 1,\dotsc,n \\ j = 1,\dotsc,n-1 \\ k,\ell = 0,\dotsc,d-1}} x_i^k s_i \dotsm s_{n-1} \otimes x_j^\ell s_j \dotsm s_{n-2} \otimes s_{n-2} \dotsm s_j y_{j,\ell} \otimes s_{n-1} \dotsm s_i y_{i,k} \\
      &\mapsto \sum_{\substack{i = 1,\dotsc,n \\ j = 1,\dotsc,n-1 \\ k,\ell = 0,\dotsc,d-1}} x_i^k s_i \dotsm s_{n-1} x_j^\ell s_j \dotsm s_{n-2} \otimes s_{n-2} \dotsm s_j y_{j,\ell} s_{n-1} \dotsm s_i y_{i,k} \\
      &\mapsto \sum_{\substack{i = 1,\dotsc,n \\ j = 1,\dotsc,n-1 \\ k,\ell = 0,\dotsc,d-1}} x_i^k s_i \dotsm s_{n-1} x_j^\ell s_j \dotsm s_{n-2} s_{n-1} \otimes s_{n-2} \dotsm s_j y_{j,\ell} s_{n-1} \dotsm s_i y_{i,k} \\
      &\mapsto
      \sum_{\substack{i = 1,\dotsc,n \\ j = 1,\dotsc,n-1 \\ k,\ell = 0,\dotsc,d-1}} \tr_{n-1} \circ \tr_n \left( x_i^k s_i \dotsm s_{n-1} x_j^\ell s_j \dotsm s_{n-2} s_{n-1} \right) s_{n-2} \dotsm s_j y_{j,\ell} s_{n-1} \dotsm s_i y_{i,k} \\
      &= s_{n-1} = \pl{s_{n-1}}(1_n),
    \end{align*}
    where the second-to-last equality uses the fact that, by Proposition~\ref{prop:dual}, the set
    \[
      \left\{ x_i^k s_i \dotsm s_{n-1} x_j^\ell s_j \dotsm s_{n-2} \mid i=1,\dotsc,n,\ j=1,\dotsc,n-1,\ k,\ell = 0,\dotsc,d-1 \right\}
    \]
    is a basis for $H_n^\lambda$ as a right $H_{n-2}^\lambda$-module and, under the trace map $\tr_{n-1} \circ \tr_n \colon H_n^\lambda \to H_{n-2}^\lambda$, the corresponding left dual basis is
    \[
      \left\{ s_{n-2} \dotsm s_j y_{j,\ell} s_{n-1} \dotsm s_i y_{i,k} \mid i=1,\dotsc,n,\ j=1,\dotsc,n-1,\ k,\ell = 0,\dotsc,d-1 \right\}.
    \]
    Here we use the key property \cite[(3.1)]{RS17} of dual bases.  Note that the notation $\vee$ indicates \emph{right} dual basis in \cite{RS17}.
  }
\end{proof}

\begin{lem} \label{lem:rightcross-invariance}
  The images of the diagrams
  
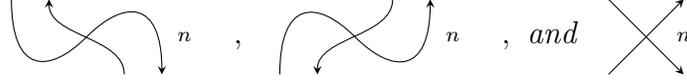
\[
    \begin{tikzpicture}[anchorbase]
    \draw[->] (1,0) .. controls (1,0.5) and (0,0.5) .. (0,1);
    \draw[->] (-0.5,1) .. controls (-0.5,0) and (0,0) .. (0.5,0.5) .. controls (1,1) and (1.5,1) .. (1.5,0);
    \draw (1.8,0.5) node {\regionlabel{n}};
    \end{tikzpicture}
    \quad, \quad
    \begin{tikzpicture}[anchorbase]
      \draw[<-] (-1,0) .. controls (-1,0.5) and (0,0.5) .. (0,1);
      \draw[<-] (0.5,1) .. controls (0.5,0) and (0,0) .. (-0.5,0.5) .. controls (-1,1) and (-1.5,1) .. (-1.5,0);
      \draw (0.8,0.5) node {\regionlabel{n}};
    \end{tikzpicture}
    \quad, \text{ and} \quad
    \begin{tikzpicture}[anchorbase]
      \draw[->] (0,0) -- (1,1);
      \draw[<-] (1,0) -- (0,1);
      \draw (1,0.5) node {\regionlabel{n}};
    \end{tikzpicture}
  \]
  under the functor $\bF$ are equal.
\end{lem}

\begin{proof}
  This follows from a straightforward computation, which will be omitted.  (See also \cite[Lem.~7.2]{RS17}.)
  \details{
    The image under $\bF$ of the first diagram is the composition
    \[
      (n)_{n-1}(n) \xrightarrow{\eta_\rR \otimes \id} \prescript{}{n}(n+1)_{n-1}(n) \xrightarrow{\pr{s_n} \otimes \id} \prescript{}{n}(n+1)_{n-1}(n) \xrightarrow{\id \otimes \varepsilon_\rR} \prescript{}{n}(n+1)_n.
    \]
    For $a \in (n)_{n-1}$ and $a' \in \prescript{}{n-1}(n)$, we have
    \[
      a \otimes a' \mapsto a \otimes a' \mapsto as_n \otimes a' \mapsto a s_n a'.
    \]
    Similarly, the image under $\bF$ of the second diagram is the composition
    \[
      (n)_{n-1}(n) \xrightarrow{\id \otimes \eta_\rR} (n)_{n-1}(n+1)_n \xrightarrow{\id \otimes \pl{s_n}} (n)_{n-1}(n+1)_n \xrightarrow{\varepsilon_\rR \otimes \id} \prescript{}{n}(n+1)_n.
    \]
    For $a \in (n)_{n-1}$ and $a' \in \prescript{}{n-1}(n)$, we have
    \[
      a \otimes a' \mapsto a \otimes a' \mapsto a \otimes s_n a' \mapsto a s_n a'. \qedhere
    \]
  }
\end{proof}

\begin{lem} \label{lem:leftcross-invariance}
  The images of the diagrams
  \[
    \begin{tikzpicture}[anchorbase]
      \draw[<-] (1,0) .. controls (1,0.5) and (0,0.5) .. (0,1);
      \draw[<-] (-0.5,1) .. controls (-0.5,0) and (0,0) .. (0.5,0.5) .. controls (1,1) and (1.5,1) .. (1.5,0);
      \draw (2,0.5) node {$n$};
    \end{tikzpicture}
    \quad, \quad
    \begin{tikzpicture}[anchorbase]
      \draw[->] (-1,0) .. controls (-1,0.5) and (0,0.5) .. (0,1);
      \draw[->] (0.5,1) .. controls (0.5,0) and (0,0) .. (-0.5,0.5) .. controls (-1,1) and (-1.5,1) .. (-1.5,0);
      \draw (1,0.5) node {$n$};
    \end{tikzpicture}
    \quad, \text{ and} \quad
    \begin{tikzpicture}[anchorbase]
      \draw[<-] (0,0) -- (1,1);
      \draw[->] (1,0) -- (0,1);
      \draw (1,0.5) node {\regionlabel{n}};
    \end{tikzpicture}
  \]
  under the functor $\bF$ are equal.
\end{lem}

\begin{proof}
  By Lemma~\ref{lem:decomp(n+1)}, ${}_n(n+1)_n$ is generated, as an $(H_n^\lambda,H_n^\lambda)$-bimodule, by $s_n$ and $x_{n+1}^j$, $j = 0,\dotsc,d-1$.  Thus, it suffices to compute the images of these elements.  The image under $\bF$ of the first diagram is the composition
  \[
    \prescript{}{n}(n+1)_n \cong (n)_n(n+1)_n \xrightarrow{\eta_\rL \otimes \id} (n)_{n-1}(n+1)_n \xrightarrow{\id \otimes \pl{s_n}} (n)_{n-1}(n+1)_n \xrightarrow{\id \otimes \varepsilon_\rL} (n)_{n-1}(n).
  \]
  For $j \in \{0,\dotsc,d-1\}$, we have
  \begin{multline*}
    x_{n+1}^j
    \mapsto \sum_{\substack{i=1,\dotsc,n \\ k=0,\dotsc,d-1}} x_i^k s_i \dotsm s_{n-2} s_{n-1} \otimes s_{n-1} s_{n-2} \dotsm s_i y_{i,k} x_{n+1}^j \\
    \mapsto \sum_{\substack{i=1,\dotsc,n \\ k=0,\dotsc,d-1}} x_i^k s_i \dotsm s_{n-2} s_{n-1} \otimes s_n s_{n-1} s_{n-2} \dotsm s_i y_{i,k} x_{n+1}^j
    \mapsto 0,
  \end{multline*}
  and
  \begin{multline*}
    s_n
    \mapsto \sum_{\substack{i=1,\dotsc,n \\ k=0,\dotsc,d-1}} x_i^k s_i \dotsm s_{n-2} s_{n-1} \otimes s_{n-1} s_{n-2} \dotsm s_i y_{i,k} s_n \\
    \mapsto \sum_{\substack{i=1,\dotsc,n \\ k=0,\dotsc,d-1}} x_i^k s_i \dotsm s_{n-2} s_{n-1} \otimes s_n s_{n-1} s_{n-2} \dotsm s_i y_{i,k} s_n
    \mapsto 1_n \otimes 1_n,
  \end{multline*}
  where we used Proposition~\ref{prop:dual} in the last map.

  The image under $\bF$ of the second diagram is the composition
  \[
    \prescript{}{n}(n+1)_n
    \cong \prescript{}{n}(n+1)_n(n)
    \xrightarrow{\id \otimes \eta_\rL} \prescript{}{n}(n+1)_{n-1}(n)
    \xrightarrow{\pr{s_n} \otimes \id} \prescript{}{n}(n+1)_{n-1}(n)
    \xrightarrow{\varepsilon_\rL \otimes \id} (n)_{n-1}(n).
  \]
  For $j = \{0,\dotsc,d-1\}$, we have
  \begin{multline*}
    x_{n+1}^j
    \mapsto x_{n+1}^j \sum_{\substack{i=1,\dotsc,n \\ k=0,\dotsc,d-1}} x_i^k s_i \dotsm s_{n-2} s_{n-1} \otimes s_{n-1} s_{n-2} \dotsm s_i y_{i,k} \\
    \mapsto x_{n+1}^j \sum_{\substack{i=1,\dotsc,n \\ k=0,\dotsc,d-1}} x_i^k s_i \dotsm s_{n-2} s_{n-1} s_n \otimes s_{n-1} s_{n-2} \dotsm s_i y_{i,k}
    \mapsto 0,
  \end{multline*}
  and
  \begin{multline*}
    s_n
    \mapsto s_n \sum_{\substack{i=1,\dotsc,n \\ k=0,\dotsc,d-1}} x_i^k s_i \dotsm s_{n-2} s_{n-1} \otimes s_{n-1} s_{n-2} \dotsm s_i y_{i,k} \\
    \mapsto s_n \sum_{\substack{i=1,\dotsc,n \\ k=0,\dotsc,d-1}} x_i^k s_i \dotsm s_{n-2} s_{n-1} s_n \otimes s_{n-1} s_{n-2} \dotsm s_i y_{i,k}
    \mapsto 1_n \otimes 1_n,
  \end{multline*}
  where, in the last map, we used the fact that
  \[
    \tr_{n+1} \left( s_n x_i^k s_i \dotsm s_{n-2} s_{n-1} s_n \right)
    = \tr_n \left( x_i^k s_i \dotsm s_{n-2} s_{n-1} \right)
  \]
  by Corollary~\ref{cor:subseqtraces}, and that
  \begin{multline*}
    \sum_{\substack{i=1,\dotsc,n \\ k=0,\dotsc,d-1}} \tr_n \left( x_i^k s_i \dotsm s_{n-2} s_{n-1} \right) \otimes_{H_{n-1}^\lambda} s_{n-1} s_{n-2} \dotsm s_i y_{i,k} \\
    = 1_n \otimes_{H_{n-1}^\lambda} \sum_{\substack{i=1,\dotsc,n \\ k=0,\dotsc,d-1}} \tr_n \left( x_i^k s_i \dotsm s_{n-2} s_{n-1} \right) s_{n-1} s_{n-2} \dotsm s_i y_{i,k}
    = 1_n \otimes_{H_{n-1}^\lambda} 1_n
  \end{multline*}
  by Proposition~\ref{prop:dual}.
  \details{
    We use here the property \cite[(3.1)]{RS17} of dual bases.
  }
\end{proof}

\subsection{The action functors are well-defined}

We now prove one of our main results: that the functors $\bF_n$ are well defined.  We assume some knowledge of cyclic biadjointness and its relation to planar diagrammatics for bimodules.  We refer the reader to \cite{Kho10} for an overview of this topic.

\begin{theo} \label{theo:action-functor}
  The above maps give a well-defined additive functor
  \[
    \bF_n \colon \cH^\lambda \to \cB_n^\lambda
  \]
  for each $n \in \N$ and hence define an action of $\cH^\lambda$ on $\bigoplus_{n \in \N} H_n^\lambda\md$.
\end{theo}

\begin{proof}
  The action comes from the standard action of bimodules on categories of modules, via the tensor product.  See, for example, \cite[(7.1)]{RS17}.

  By definition, the category $\cB_n^\lambda$ is idempotent complete for all $n \in \N$.  Thus, any functor from $\tcH^\lambda$ to $\cB_n^\lambda$ naturally induces a functor $\cH^\lambda \to \cB^\lambda$.  Therefore, it suffices to consider the category $\tcH^\lambda$.

  By Proposition~\ref{prop:adjunction-maps}, the images $(n+1)_n$ and $\prescript{}{n}(n+1)$ of the objects $\sQ_+$ and $\sQ_-$ under $\bF$ are biadjoint.  Thus, the zigzag identities
  \[
    \begin{tikzpicture}[anchorbase]
      \draw[->] (-1,1) .. controls (-1,-1) and (0,-1) .. (0,0) .. controls (0,1) and (1,1) .. (1,-1);
    \end{tikzpicture}
    \ =\
    \begin{tikzpicture}[anchorbase]
      \draw[->] (1,1) .. controls (1,-1) and (0,-1) .. (0,0) .. controls (0,1) and (-1,1) .. (-1,-1);
    \end{tikzpicture}
    \ =\
    \begin{tikzpicture}[anchorbase]
      \draw[->] (0,1) -- (0,-1);
    \end{tikzpicture}
    \quad, \qquad
    \begin{tikzpicture}[anchorbase]
      \draw[<-] (-1,1) .. controls (-1,-1) and (0,-1) .. (0,0) .. controls (0,1) and (1,1) .. (1,-1);
    \end{tikzpicture}
    \ =\
    \begin{tikzpicture}[anchorbase]
      \draw[<-] (1,1) .. controls (1,-1) and (0,-1) .. (0,0) .. controls (0,1) and (-1,1) .. (-1,-1);
    \end{tikzpicture}
    \ =\
    \begin{tikzpicture}[anchorbase]
      \draw[<-] (0,1) -- (0,-1);
    \end{tikzpicture}
    \quad ,
  \]
  are preserved by $\bF$.  The fact that $\bF$ preserves invariance under local isotopies then follows from Lemmas~\ref{lem:dot-cyclicity}, \ref{lem:downcross-invariance}, \ref{lem:rightcross-invariance}, and \ref{lem:leftcross-invariance}.

  It remains to show that $\bF$ preserves the relations \eqref{rel:braid}--\eqref{rel:left-curl-zero}.  The fact that $\bF$ preserves relations \eqref{rel:braid}, \eqref{rel:s-squared}, and \eqref{rel:dotslide} follows immediately from \eqref{eq:deg3}, \eqref{eq:deg1}, and \eqref{eq:deg5}, respectively.

  If the rightmost region is labeled $n$, then the image under $\bF$ of the left side of \eqref{rel:left-curl-zero} is the map
  \begin{gather*}
    (n+1)_n
    \xrightarrow{\eta_\rR} {}_{n+1}(n+2)_n
    \xrightarrow{\pr{s_{n+1}}} {}_{n+1}(n+2)_n
    \xrightarrow{\varepsilon_\rL} (n+1)_n,\\
    a
    \mapsto a
    \mapsto a s_{n+1}
    \mapsto 0.
  \end{gather*}
  Thus, $\bF$ preserves \eqref{rel:left-curl-zero}.

  If the outside region is labeled $n$, then the image under $\bF$ of the left side of \eqref{rel:cc-bubble} is the map
  \begin{gather*}
    (n)
    \xrightarrow{\eta_\rR} {}_n(n+1)_n
    \xrightarrow{\pl{\left(x_{n+1}^j\right)}}  {}_n(n+1)_n
    \xrightarrow{\varepsilon_\rL} (n),\\
    a
    \mapsto a
    \mapsto x_{n+1}^j a = a x_{n+1}^j
    \mapsto
    \begin{cases}
      0 & \text{if } j < d-1, \\
      a & \text{if } j=d-1, \\
      a \sum_{i \in I} i \lambda_i & \text{if } j=d,
    \end{cases}
  \end{gather*}
  where the last case follows from Lemma~\ref{lem:tr_x^d}.  Thus, $\bF$ preserves \eqref{rel:cc-bubble}.

  Now, if the rightmost region is labeled $n$, then the image under $\bF$ of the first term on the right side of \eqref{rel:down-up-double-crossing} (i.e.\ the double crossing) is the map
  \[
    {}_n(n+1)_n
    \xrightarrow{\eqref{eq:left-crossing-map}} (n)_{n-1}(n)
    \xrightarrow{\eqref{eq:right-crossing-map}} {}_n(n+1)_n.
  \]
  This map is uniquely determined by the images of $s_n$ and $x_{n+1}^k$, $k=0,\dotsc,d-1$.  We compute
  \[
    s_n
    \mapsto 1_n \otimes 1_n
    \mapsto s_n
    \quad \text{and} \quad
    x_{n+1}^k
    \mapsto 0.
  \]
  On the other hand, the image under $\bF$ of the sum in \eqref{rel:down-up-double-crossing} of diagrams over $j$, is (using \eqref{eq:dot-vee-maps-down})
  \[
    \sum_{j=0}^{d-1} \pr{x_{n+1}^j} \circ \eta_\rR \circ \varepsilon_\rL \circ \pl{y_{n+1,j}}.
  \]
  This acts as
  \[
    s_n
    \mapsto \sum_{j=0}^{d-1} \tr_{n+1} (y_{n+1,j} s_n) x_{n+1}^j
    = 0,
    \quad
    x_{n+1}^k
    \mapsto \sum_{j=0}^{d-1} \tr_{n+1} \left( y_{n+1,j} x_{n+1}^k \right) x_{n+1}^j
    = x_{n+1}^k,
  \]
  where the equalities follow from Proposition~\ref{prop:dual}.  It follows that $\bF$ preserves \eqref{rel:down-up-double-crossing}.

  Finally, if the rightmost region is labeled $n$, then the image under $\bF$ of the left side of \eqref{rel:up-down-double-crossing} is the map
  \begin{gather*}
    (n)_{n-1}(n)
    \xrightarrow{\eqref{eq:right-crossing-map}} {}_n(n+1)_n
    \xrightarrow{\eqref{eq:left-crossing-map}} (n)_{n-1}(n),
    \\
    a \otimes b
    \mapsto a s_n b
    \mapsto a \otimes b.
  \end{gather*}
  Thus, $\bF$ preserves \eqref{rel:up-down-double-crossing}.
\end{proof}

\subsection{Categorification of Fock space} \label{subsec:Fock-space}

We assume in this subsection that $\kk$ is a field of characteristic zero.  For $n \in \N$, let $K_0(H_n^\lambda\pmd)$ be the split Grothendieck group of the category $H_n^\lambda\pmd$ of finitely-generated projective $H_n^\lambda$-modules.  By Theorems~\ref{theo:main} and~\ref{theo:action-functor}, we have ring homomorphisms
\[
  \fh_d \to K_0(\cH^\lambda) \xrightarrow{\bigoplus_{n \in \N} K_0(\bF_n)} \End \left( \bigoplus_{n \in \N} K_0(H_n^\lambda\pmd) \right),
\]
yielding an action of $\fh_d$ on $\bigoplus_{n \in \N} K_0(H_n^\lambda\pmd)$.  (Note that we do not use the injectivity statement in Theorem~\ref{theo:main}, which we have yet to prove.)

If $\kk_0$ denotes the trivial one-dimensional $H_0^\lambda$-module, then, for any partition $\mu$, we have
\[
  \bF_0([\sQ_-^{\mu}]) \cdot [\kk_0] = 0,\qquad
  \bF_0([\sQ_+^{\mu}]) \cdot [\kk_0] = [H_n^\lambda e_\mu].
\]
It follows from the Stone--von Neumann Theorem that the submodule of $\bigoplus_{n \in \N} K_0(H_n^\lambda\pmd)$ generated by $[\kk_0]$ is the Fock space representation of $\fh_d$.  (See \cite[Th.~2.11(c)]{SY15} for a version of the Stone--von Neumann Theorem in the general setting of Heisenberg doubles.)

\begin{prop} \label{prop:h-KH-injective}
  The ring homomorphism $\fh_d \to K_0(\cH^\lambda)$ of Theorem~\ref{theo:main} is injective.
\end{prop}

\begin{proof}
  This follows from the above discussion and the fact that the Fock space representation is faithful.  (See, for example, \cite[Th.~2.11(d)]{SY15} for this statement in the setting of Heisenberg doubles.)
\end{proof}

The action of $i$-induction and $i$-restriction on $\bigoplus_{n \in \N} K_0(H_n^\lambda\pmd)$ realizes the irreducible highest weight module $V(\lambda)$ of $\fg$ of highest weight $\lambda$.  (See, for example, \cite[Ch.~9]{Kle05}.)  In general, the $\fh_d$-module $\bigoplus_{n \in \N} K_0(H_n^\lambda\pmd)$ is a direct sum of Fock space representations.  This corresponds to the fact that the restriction of $V(\lambda)$ to the principal Heisenberg subalgebra of $\fg$ is a direct sum of Fock spaces.  If $\lambda$ is of level one, then $V(\lambda)$ remains irreducible as a module over the Heisenberg algebra, and we are in the setting of \cite{QSY18}.  See Section~\ref{subsec:truncations} for further comments in this direction.  In general, the projective modules $H_n^\lambda e_\mu$ are not indecomposable.  An explicit description of the indecomposable projective $H_n^\lambda$-modules is not currently known for arbitrary level.  (See \cite{Bru12} for the level two case.)

%
\section{Properties of the action functors} \label{sec:action-properties}
%

In this section, $\kk$ is an arbitrary commutative ring until after Proposition~\ref{prop:cc-power-sums}.  At that point, we assume $\kk$ is a field of characteristic zero for the remainder of the section.

Note that $\bF_n$ maps closed diagrams, which are endomorphisms of the identity object $\one$, to $\End_{\cB_n^\lambda} (n)$, the algebra of bimodule endomorphisms of $(n)$.  We have a natural isomorphism of algebras
\[
  Z(H_n^\lambda) \xrightarrow{\cong} \End_{\cB_n^\lambda} (n),\quad a \mapsto \pl{a},
\]
where we recall that $\pl{a}$ denotes left multiplication by $a$.  In what follows, we will often identify $Z(H_n^\lambda)$ and $\End_{\cB_n^\lambda} (n)$ via this isomorphism.

\begin{prop} \label{prop:cc-power-sums}
  For all $n \in \N$ and $t \ge 3$, we have
  \[
    \bF_n
    \left(
      \begin{tikzpicture}[anchorbase]
        \draw [->](0,0) arc (180:360:0.4);
        \draw (0.8,0) arc (0:180:0.4);
        \bluedot{(0.7,-0.3)} node [anchor=north,color=black] {\dotlabel{d-1+t}};
      \end{tikzpicture}
    \right)
    = x_1^{t-2} + x_2^{t-2} + \dotsb + x_n^{t-2} + \text{l.o.t.},
  \]
  where ``l.o.t.'' is polynomial in $x_1,\dotsc,x_n$ of degree less than or equal to $t-3$.  In other words, $\bF_n$ maps a counterclockwise circle with $d-1+t$ dots, $t \ge 3$, to the $(t-2)$-nd power sum in the generators $x_1,\dotsc,x_n$, up to terms of lower degree.
\end{prop}

\begin{proof}
  It follows from Lemma~\ref{lem:bubble_slide} that, for $t \ge 3$,
  \[
    \begin{tikzpicture}[anchorbase]
      \draw [->](0,0) arc (180:360:0.4);
      \draw (0.8,0) arc (0:180:0.4);
      \bluedot{(0.7,-0.3)} node [anchor=north,color=black] {\dotlabel{d-1+t}};
      \draw[->] (1.3,-0.6) to (1.3,0.6);
    \end{tikzpicture}
    \ =\
    \begin{tikzpicture}[anchorbase]
      \draw[->] (-0.3,-0.6) to (-0.3,0.6);
      \draw [->](0,0) arc (180:360:0.4);
      \draw (0.8,0) arc (0:180:0.4);
      \bluedot{(0.7,-0.3)} node [anchor=north,color=black] {\dotlabel{d-1+t}};
    \end{tikzpicture}
    \ +\
    \begin{tikzpicture}[anchorbase]
      \draw[->] (0,-0.6) to (0,0.6);
      \bluedot{(0,0)} node [anchor=west,color=black] {\dotlabel{t-2}};
    \end{tikzpicture}
    \ +\
    \sum_{k=0}^{t-3} \
    \begin{tikzpicture}[anchorbase]
      \draw[->] (0,-0.6) to (0,0.6);
      \bluedot{(0,0)} node [anchor=east,color=black] {\dotlabel{k}};
    \end{tikzpicture}
    \ \alpha_k,
  \]
  where $\alpha_k \in \End_{\cH^\lambda}(\one)$, $k \in \{0,\dotsc,t-3\}$, is a linear combination of counterclockwise circles with at most $d-1+t-2$ dots.  It follows that
  \begin{equation} \label{eq:power-sums-from-bubbles}
    \begin{tikzpicture}[anchorbase]
      \draw [->](0,0) arc (180:360:0.4);
      \draw (0.8,0) arc (0:180:0.4);
      \bluedot{(0.7,-0.3)} node [anchor=north,color=black] {\dotlabel{d-1+t}};
      \draw[->] (1.3,-0.6) to (1.3,0.6);
      \draw[->] (1.6,-0.6) to (1.6,0.6);
      \node at (2,0) {$\cdots$};
      \draw[->] (2.4,-0.6) to (2.4,0.6);
    \end{tikzpicture}
    \ =\
    \begin{tikzpicture}[anchorbase]
      \draw[->] (0,-0.6) to (0,0.6);
      \bluedot{(0,0)} node [anchor=west,color=black] {\dotlabel{t-2}};
      \draw[->] (0.8,-0.6) to (0.8,0.6);
      \node at (1.2,0) {$\cdots$};
      \draw[->] (1.6,-0.6) to (1.6,0.6);
    \end{tikzpicture}
    \ +\
    \begin{tikzpicture}[anchorbase]
      \draw[->] (-0.3,-0.6) to (-0.3,0.6);
      \draw[->] (0,-0.6) to (0,0.6);
      \bluedot{(0,0)} node [anchor=west,color=black] {\dotlabel{t-2}};
      \draw[->] (0.8,-0.6) to (0.8,0.6);
      \node at (1.2,0) {$\cdots$};
      \draw[->] (1.6,-0.6) to (1.6,0.6);
    \end{tikzpicture}
    \ + \dotsb + \
    \begin{tikzpicture}[anchorbase]
      \draw[->] (0.8,-0.6) to (0.8,0.6);
      \node at (1.2,0) {$\cdots$};
      \draw[->] (1.6,-0.6) to (1.6,0.6);
      \draw[->] (1.9,-0.6) to (1.9,0.6);
      \bluedot{(1.9,0)} node [anchor=west,color=black] {\dotlabel{t-2}};
    \end{tikzpicture}
    \ + \text{l.o.t.},
  \end{equation}
  where ``l.o.t.'' denotes a linear combination of lower order terms---diagrams consisting of upwards pointing strands carrying fewer than $t-2$ total dots, with a closed diagram in the rightmost region.

  Now, consider the commutative diagram
  \[
    \xymatrix{
      \End_{\cH^\lambda} \one \ar[r]^{\bF_n} \ar[d] & \End_{\cB_n^\lambda} (n) \ar@{^{(}->}[d]^\iota \\
      \End_{\cH^\lambda} \sQ_+^n \ar[r]^{\bF_0} & \End_{\cB_0^\lambda} (n)_0
    }
  \]
  where the leftmost vertical arrow is given by horizontal composition on the right by the identity morphism of $\sQ_+^n$, and $\iota$ is the natural inclusion.  (Diagrammatically, the leftmost vertical arrow takes a closed diagram and places $n$ upwards pointing arrows to its right.)  It follows from \eqref{eq:power-sums-from-bubbles} that
  \[
    \iota \circ \bF_n \left(
    \begin{tikzpicture}[anchorbase]
      \draw [->](0,0) arc (180:360:0.4);
      \draw (0.8,0) arc (0:180:0.4);
      \bluedot{(0.7,-0.3)} node [anchor=north,color=black] {\dotlabel{d-1+t}};
    \end{tikzpicture}
    \right)
    = x_1^{t-2} + x_2^{t-2} + \dotsb + x_n^{t-2} + \text{l.o.t.},
  \]
  where ``l.o.t.'' is polynomial in $x_1,\dotsc,x_n$ of degree less than or equal to $t-3$.
\end{proof}

For the remainder of the section, we assume that $\kk$ is a field of characteristic zero.

\begin{cor} \label{cor:center-surjection}
  For all $n \in \N$, the functor $\bF_n$ induces a surjective homomorphism of algebras $\End_{\cH^\lambda} \one \twoheadrightarrow \End_{\cB_n^\lambda} (n) \cong Z(H_n^\lambda)$.
\end{cor}

\begin{proof}
  Since the power sums generate the ring of symmetric polynomials over a field of characteristic zero, the result follows from Proposition~\ref{prop:cc-power-sums} and \cite[Th.~1]{Bru08}, which states that $Z(H_n^\lambda)$ consists of all symmetric polynomials in the $x_1,\dotsc,x_n$.
\end{proof}

\begin{prop} \label{prop:psi0-injective}
  The homomorphism $\psi_0$ of \eqref{eq:bubble-iso} is injective.
\end{prop}

\begin{proof}
  Define a grading on $\Pi$ by setting $\deg y_i = i$ for all $i \in \N_+$.  Suppose $f \in \Pi$ is a nonzero element of top degree $a$ (i.e.\ $f$ is a sum of a nonzero element of degree $a$ and elements of lower degree).  Then $\bF_n(\psi_0(f)) \in Z(H_n^\lambda)$ and so, by \cite[Th.~1]{Bru08}, $\bF_n(\psi_0(f))$ is a symmetric polynomial in $x_1,\dotsc,x_n$.  By Proposition~\ref{prop:cc-power-sums}, the top degree of the monomials appearing in $\bF_n(\psi_0(f))$ is $a$.  Then, by \cite[Th.~3.2]{Bru08}, $\bF_n(\psi_0(f))$ is nonzero for sufficiently large $n$.
  \details{
    By \cite[Th.~3.2]{Bru08}, the elements
    \[
      p_n(\mu) := \sum_{\nu \sim \mu} x_1^{\nu_1} \dotsm x_n^{\nu_n},\quad
      \ell(\mu) + |\mu/d| \le n,
    \]
    form a basis of $Z(H_n^\lambda)$, where the sum is over all distinct rearrangements of $\mu$ (so that the polynomials above are simply the monomial symmetric polynomials).  In the above, if $\mu = (\mu_1 \ge \mu_2 \ge \dotsb)$, then $\mu/d = (\lfloor \mu_1/d \rfloor, \lfloor \mu_2/d \rfloor,\dotsc)$.  Hence $|\mu/d| \le |\mu|/d$.  So the elements
    \[
      p_n(\mu),\quad \ell(\mu) + |\mu|/d \le n,
    \]
    are linearly independent.  Now, the symmetric polynomials of degree $\le a$ are spanned by a finite number of the $p_n(\mu)$.  Hence $\bF_n(\psi_0(f))$ is a nonzero linear combination of the $p_n(\mu)$, and these are linearly independent for sufficiently large $n$.  Thus $\bF_n(\psi_0(f)) \ne 0$ for sufficiently large $n$.
  }
  It follows that $\psi_0(f) \ne 0$.
\end{proof}

Consider the $\kk$-algebra
\[
  (H_{n+k}^\lambda)^{H_n^{\lambda}} := \{ a \in H_{n+k}^\lambda \mid ha=ah \text{ for all } h \in H_n^\lambda \}.
\]
This algebra is canonically isomorphic to the endomorphism ring of the bimodule ${_n}(n+k)$, and, therefore, to the endomorphism ring of the restriction functor, via the map that assigns to $a\in (H_{n+k}^\lambda)^{H_n^{\lambda}}$ the endomorphism $\pl{a}$.  Likewise, the opposite algebra of $(H_{n+k}^\lambda)^{H_n^{\lambda}}$ is canonically isomorphic to the endomorphism ring of the bimodule $(n+k)_n$ and, therefore, to that of the induction functor, via the map that assigns to $a\in (H_{n+k}^\lambda)^{H_n^{\lambda}}$ the endomorphism $\pr{a}$.

\begin{prop} \label{prop:psim-injective}
  The homomorphism $\psi_m$ of \eqref{eq:psi_m-def} is injective.
\end{prop}

\begin{proof}
  Fix $j \in I$ such that $\lambda_j \ne 0$.  Define $\mu = \sum_{i \in I} \mu_i \omega_i \in P_+$ by $\mu_i = \lambda_{i-j}$ for $i \in I$.  Thus $\mu_0 \ne 0$.  We have an isomorphism of algebras $H_n^\lambda \cong H_n^\mu$ that fixes $s_1,\dotsc,s_{n-1}$ and maps $x_i$ to $x_i-j$ for $i \in \{1,2,\dotsc,n\}$. We also have a surjective homomorphism of algebras $H_n^\mu \twoheadrightarrow \kk S_n$ given by mapping $x_i$ to the $i$-th Jucys--Murphy element.

  For $n \in \N$, consider the composition of algebra homomorphisms
  \[
    H_m \otimes \Pi
    \xrightarrow{\psi_m} \End_{\cH^\lambda} (\sQ_+^m)
    \xrightarrow{\bF_n} \End (m+n)_n
    \hookrightarrow \left( H_{m+n}^\lambda \right)^\op
    \cong \left( H_{m+n}^\mu \right)^\op
    \twoheadrightarrow (\kk S_{m+n})^\op.
  \]
  This composition is precisely the homomorphism $\psi_{m,n}'$ of \cite[\S4]{Kho14}.  It is shown there that these maps are asymptotically injective, in the sense that if $z$ is some nonzero element of $H_m \otimes \Pi$, then $\psi_{m,n}'(z) \ne 0$ for sufficiently large $n$.  It follows that $\psi_m$ is injective.
\end{proof}

\begin{prop} \label{prop:full-Q+k}
  For all $n,k \in \N$, the functor $\bF_n$ induces a surjective homomorphism of algebras
  \[
    \End_{\cH^\lambda} \sQ_+^k \twoheadrightarrow \End_{\cB_n^\lambda} (n+k)_n.
  \]
\end{prop}

\begin{proof}
  We proceed by induction on $k$.  For $k=0$, the result is Corollary~\ref{cor:center-surjection}.

  Let $m=n+k$.  Assume that the result holds for some $m \ge n$.  We must prove that $\pr{a}$ is in the image of $\bF_n$ for all $a \in (H_{m+1}^\lambda)^{H_n^\lambda}$.  Now, it follows from the bimodule decomposition \eqref{eq:decomp1} that
  \[
    (H_{m+1}^\lambda)^{H_n^\lambda}
    = \left( H_m^\lambda s_m H_m^\lambda \right)^{H_n^\lambda} \oplus \bigoplus_{j=0}^{d-1} (H_m^\lambda)^{H_n^\lambda} x_{m+1}^j.
  \]
  Therefore, it suffices to consider the cases where $a \in (H_m^\lambda)^{H_n^\lambda} x_{m+1}^j$ for some $j \in \{0,\dotsc,d-1\}$ and $a \in \left( H_m^\lambda s_m H_m^\lambda \right)^{H_n^\lambda}$.

  First suppose that $a = a' x_{m+1}^j = x_{m+1}^j a'$ for some $a' \in (H_m^\lambda)^{H_n^\lambda}$ and $j \in \{0,\dotsc,d-1\}$.  By the inductive hypothesis, there exists a $D \in \End_{\cH^\lambda} \sQ_+^{m-n}$ such that $\bF_n(D) = a'$.  Then
  \[
    \bF_n \left(
    \begin{tikzpicture}[anchorbase]
      \draw[->] (0,0) -- (0,1);
      \bluedot{(0,0.5)} node[anchor=east, color=black] {\dotlabel{j}};
      \draw (0.5,0.5) node {$D$};
    \end{tikzpicture}
    \right) = \pl{a}.
  \]

  Now suppose $a \in \left( H_m^\lambda s_m H_m^\lambda \right)^{H_n^\lambda}$.  By \eqref{eq:decomp2}, $a$ is the image under the map
  \[
    (m)_{m-1}(m) \to {}_m(m+1)_m,\quad u \otimes v \mapsto u s_m v,
  \]
  of some element $a' \in (m)_{m-1}(m)$ satisfying $ba'=a'b$ for all $b \in H_n^\lambda$.  Consider the bimodule homomorphism
  \[
    \pl{a'} = \pr{a'} \colon (n) \to {}_n(m)_{m-1}(m)_n.
  \]
  By adjunction and the induction hypothesis, we have an element of $\Hom_{\cH^\lambda} (\one, \sQ_-^{m-n} \sQ_+ \sQ_- \sQ_+^{m-n})$, which we will depict as
  \[
    \begin{tikzpicture}[anchorbase]
      \draw[<-] (0,0.1) -- (0,0) arc (180:360:.3) -- (0.6,0.1);
      \draw[->,dashed] (-0.2,0.1) -- (-0.2,0) arc (180:360:.5) -- (0.8,0.1);
      \filldraw[green,xshift=-2pt,yshift=-5pt] (0.3,-0.4) rectangle ++(4pt,10pt);
    \end{tikzpicture}
    \ \colon \one \to \sQ_-^{m-n} \sQ_+ \sQ_- \sQ_+^{m-n},
  \]
  (where the dashed line represents $m-n$ lines) such that
  \[
    \bF_n \left(\
    \begin{tikzpicture}[anchorbase]
      \draw[<-] (0,0.1) -- (0,0) arc (180:360:.3) -- (0.6,0.1);
      \draw[->,dashed] (-0.2,0.1) -- (-0.2,0) arc (180:360:.5) -- (0.8,0.1);
      \filldraw[green,xshift=-2pt,yshift=-5pt] (0.3,-0.4) rectangle ++(4pt,10pt);
    \end{tikzpicture}\
    \right) = \pl{a'}.
  \]
  \details{
    The bimodule ${}_n(m)_{m-1}$ is biadjoint to the bimodule $_{m-1}(m)_n$, so
    \[
      \Hom \big( (n), {}_n(m)_{m-1}(m)_n \big)
      \cong \Hom \big( {}_{m-1}(m)_n, {}_{m-1}(m)_n \big)
      \subseteq \Hom \big( (m)_n, (m)_n \big).
    \]
    Since the adjunction maps are the images under $\bF_n$ of the cups and caps (for an appropriate value of $n$), it follows from the induction hypothesis that the functor is surjective on $\Hom \big( (n), {}_n(m)_{m-1}(m)_n \big)$.
  }
  It follows that
  \[
    \bF_n \left(\
    \begin{tikzpicture}[anchorbase]
      \draw[->] (0,0) -- (0,0.8) arc (180:0:0.6) -- (1.2,0.6) arc (360:180:0.3) .. controls (0.6,1) and (1.2,1) .. (1.2,2);
      \draw[->,dashed] (0.15,0) -- (0.15,0.8) arc (180:0:0.15) -- (0.45,0.6) arc (180:360:0.45) -- (1.35,2);
      \filldraw[green,xshift=-2pt,yshift=-5pt] (0.9,0.25) rectangle ++(4pt,10pt);
    \end{tikzpicture}\
    \right) = \pl{a}.
  \]
  This completes the proof.
\end{proof}

\begin{theo} \label{theo:Fn-full}
  For all $n \in \N$, the functor $\bF_n$ is full.
\end{theo}

\begin{proof}
  It follows from Corollary~\ref{cor:Q+Q-commutation-relation} that
  every object of $\cH^\lambda$ is a direct sum of objects of the form $\sQ_-^k \sQ_+^\ell$ for some $k,\ell \in \N$.  By adjunction, we have
  \[
    \Hom_{\cH^\lambda} (\sQ_-^k \sQ_+^\ell, \sQ_-^a \sQ_+^b)
    \cong \Hom_{\cH^\lambda} (\sQ_+^\ell, \sQ_+^k \sQ_-^a \sQ_+^b).
  \]
  Thus, it suffices to consider hom-spaces with domain $\sQ_+^\ell$, $\ell \in \N$.  Again, by adjunction, we have
  \[
    \Hom_{\cH^\lambda} (\sQ_+^\ell, \sQ_-^a \sQ_+^b)
    \cong \Hom_{\cH^\lambda} (\sQ_+^{\ell+a}, \sQ_+^b).
  \]
  Since this hom-space is zero unless $\ell + a = b$, the theorem follows from Proposition~\ref{prop:full-Q+k}.
\end{proof}

It follows from Theorem~\ref{theo:Fn-full} that the category $\cH^\lambda$ yields a graphical calculus for the representation theory of the $H_n^\lambda$, $n \in \N$.  In the remainder of this section, we use this fact to deduce algebraic properties of these algebras.

There is an natural inclusion of algebras $H_n \otimes H_k \hookrightarrow H_{n+k}$.  This induces an inclusion of algebras
\begin{equation} \label{eq:degCHA-inclusion}
  H_n^\lambda \otimes H_k \hookrightarrow H_{n+k}^\lambda,
\end{equation}

\begin{cor} \label{cor:centralizers}
  Under the inclusion \eqref{eq:degCHA-inclusion}, the centralizer of $H_n^\lambda$ in $H_{n+k}^\lambda$ is generated by $H_k$ and the center of $H_n^\lambda$.
\end{cor}

\begin{proof}
  The opposite algebra of the centralizer of $H_n^\lambda$ in $H_{n+k}^\lambda$ is precisely $\End \big( (n+k)_n \big)$.  Thus, by Theorem~\ref{theo:Fn-full} and Proposition~\ref{prop:DH-Q-iso}, the centralizer of $H_n^\lambda$ in $H_{n+k}^\lambda$ is generated by $\bF_n(\psi_k(H_k \otimes \Pi))$.  Since $\bF_n(\psi_k(\Pi))$ is the center of $H_n^\lambda$ and $\bF_n(\psi_k(H_k))$ is the image of $H_k$ in $H_{n+k}^\lambda$ under the inclusion \eqref{eq:degCHA-inclusion}, the result follows.
\end{proof}

The level one case of Corollary~\ref{cor:centralizers}, describing centralizers for group algebras of symmetric groups was proved by Olshanski (see \cite{Ols87}, or \cite[Th.~3.2.6]{CST10}).  Corollary~\ref{cor:centralizers}, which is a generalization of Olshanski's result, seems to be new.  The analogue for the degenerate affine Hecke algebra (before taking cyclotomic quotients) is mentioned in \cite[p.~577]{Che87}.

\begin{cor}
  The four bimodule maps $\varepsilon_\rR$, $\eta_\rR$, $\varepsilon_\rL$, and $\eta_\rL$ of Proposition~\ref{prop:adjunction-maps} turn the induction and restriction functors $\Ind_{H_n^\lambda}^{H_{n+1}^\lambda}$ and $\Res_{H_n^\lambda}^{H_{n+1}^\lambda}$ into a cyclic biadjoint pair.
\end{cor}

\begin{proof}
  This follows immediately from the isotopy invariance in $\cH^\lambda$ and Theorem~\ref{theo:Fn-full}.
\end{proof}

%
\section{Further directions} \label{sec:further-directions}
%

The results of the current paper suggest some natural interesting directions of future research.
In this final section, we briefly outline some of these directions.

\subsection{Truncations and categorified quantum groups} \label{subsec:truncations}

In \cite{QSY18}, it is shown that a certain truncation of a 2-category version of Khovanov's Heisenberg category is equivalent to a truncation of the Khovanov--Lauda categorified quantum group of type $A_\infty$, introduced in \cite{KL3}.  (See also the related work \cite{Rou08}.)  This equivalence is a categorification of the principal realization of the basic representation (i.e.\ the irreducible representation of highest weight $\omega_0$) of $\fsl_\infty$.  It yields an explicit action of categorified quantum groups on categories of modules for symmetric groups.

Replacing Khovanov's Heisenberg category by the more general categories $\cH^\lambda$ of the current paper should yield a generalization of the results of \cite{QSY18}.  The 2-category version $\mathscr{H}^\lambda$ of $\cH^\lambda$ has objects labelled by integers.  For $k,\ell \in \Z$, the category $\Mor_{\mathscr{H}^\lambda} (k,\ell)$ of morphisms from $k$ to $\ell$ is the full subcategory of $\cH^\lambda$ on direct sums of summands of objects $\sQ_+^m \sQ_-^n$ with $\ell - k = m - n$.  The horizontal composition in $\mathscr{H}^\lambda$ comes from the monoidal structure of $\cH^\lambda$.  One can truncate $\mathscr{H}^\lambda$ by killing all morphisms factoring through negative integers.  Then taking the idempotent completion as defined in \cite[Def.~5.1]{QSY18} yields a 2-category $\mathscr{H}^{\lambda,\textup{tr}}$.  We expect that the 2-category $\mathscr{H}^{\lambda,\textup{tr}}$ is equivalent to the truncation of the Khovanov--Lauda categorified quantum group obtained by killing all morphisms factoring through weights not appearing in the irreducible highest weight representation of $\fg$ of highest weight $\lambda$. This would yield actions of categorified quantum groups on categories of modules for degenerate cyclotomic Hecke algebras as in \cite{BK09b}.

\subsection{Wreath product algebras} \label{subsec:wreath-generalization}

In \cite{RS17}, a general Heisenberg category $\cH_B$ was introduced that depends on a graded Frobenius superalgebra $B$.  When $B$ is the ground field, this category reduces to the Heisenberg category of Khovanov.  The inspiration behind the definition of the categories $\cH_B$ is the passage from the group algebra of the symmetric group to the wreath product algebras $B^{\otimes n} \rtimes S_n$.  Just as the degenerate affine Hecke algebra appears naturally in the endomorphism algebra of $\sQ_+^n$ in Khovanov's category, the endomorphism algebra of $\sQ_+^n$ in $\cH_B$ contains an algebra denoted $D_n$ in \cite[Def.~8.12]{RS17}.  These algebras have been studied in \cite{Sav17}, where they are called \emph{affine wreath product algebras}, and denoted $\mathcal{A}_n(B)$.  In particular, cyclotomic quotients $\mathcal{A}_n^\mathbf{C}(B)$ of these algebras are defined \cite[\S6]{Sav17}, and it is shown that level one cyclotomic quotients are isomorphic to $B^{\otimes n} \rtimes S_n$ (see \cite[Cor.~6.13]{Sav17}).    On the other hand, taking $B = \kk$ yields $H_n^\lambda$, where $\lambda$ depends on the quotient parameter $\mathbf{C}$.
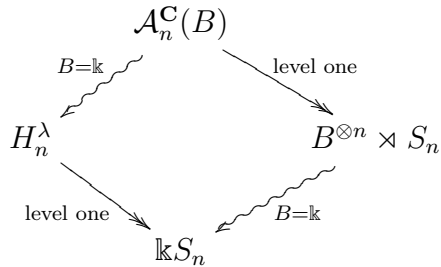
\[
  \xymatrix{
    & \mathcal{A}_n^\mathbf{C}(B) \ar@{~>}[dl]_{B = \kk} \ar@{->>}[dr]^{\text{ level one}} & \\
    H_n^\lambda \ar@{->>}[dr]_{\text{level one}} & & B^{\otimes n} \rtimes S_n \ar@{~>}[dl]^{B = \kk} \\
    & \kk S_n &
  }
\]
We expect that examination of the representation theory of the algebras $\mathcal{A}_n(B)$, together with their cyclotomic quotients, should lead to even more general Heisenberg categories $\cH_B^\mathbf{C}$ that specialize to both the categories $\cH^\lambda$ of the current paper and the categories $\cH_B$ of \cite{RS17}.\footnote{Since the writing of the current paper, these categories have been defined in \cite{Sav18}.}  One of the advantages of the presence of the graded Frobenius superalgebra $B$ is that, provided it is not concentrated in degree zero, it allows one to use grading arguments to prove the analogue of Conjecture~\ref{conj:categorification}.

\subsection{Deformations}

A $q$-deformed version of Khovanov's Heisenberg algebra was introduced in \cite{LS13}.  This deformation corresponds to replacing group algebras of symmetric groups with Iwahori--Hecke algebras of type $A$.  We expect that the results of the current paper can also be $q$-deformed by replacing cyclotomic degenerate Hecke algebras by cyclotomic Hecke algebras (otherwise known as Ariki--Koike algebras).  This would lead to a higher level generalization of the results of \cite{LS13}.

\subsection{Traces and diagrammatic pairings}

In \cite{CLLS15,CLLSS16}, the traces of Khovanov's Heisenberg category and the $q$-deformed version introduced in \cite{LS13} have been related to $W$-algebras and the elliptic Hall algebra.  The trace decategorification has also been used in \cite{LRS18} to categorify the Jack inner pairing on symmetric functions.  We expect that these results have higher level versions based on the categories $\cH^\lambda$ of the current paper or the more general categories $\cH_B^\lambda$ mentioned above.


\bibliographystyle{alphaurl}
\bibliography{MackaaySavage-biblist}

\newcommand{\etalchar}[1]{$^{#1}$}
\newcommand{\arxiv}[1]{\href{http://arxiv.org/abs/#1}{\tt
  arXiv:\nolinkurl{#1}}}
\begin{thebibliography}{CSST10}

\bibitem[Ari96]{Ari96}
S.~Ariki.
\newblock On the decomposition numbers of the {H}ecke algebra of {$G(m,1,n)$}.
\newblock {\em J. Math. Kyoto Univ.}, 36(4):789--808, 1996.
\newblock \href {http://dx.doi.org/10.1215/kjm/1250518452}
  {\path{doi:10.1215/kjm/1250518452}}.

\bibitem[BK08]{BK08}
J.~Brundan and A.~Kleshchev.
\newblock Schur-{W}eyl duality for higher levels.
\newblock {\em Selecta Math. (N.S.)}, 14(1):1--57, 2008.
\newblock \arxiv{math/0605217}.
\newblock \href {http://dx.doi.org/10.1007/s00029-008-0059-7}
  {\path{doi:10.1007/s00029-008-0059-7}}.

\bibitem[BK09a]{BK09}
J.~Brundan and A.~Kleshchev.
\newblock Blocks of cyclotomic {H}ecke algebras and {K}hovanov-{L}auda
  algebras.
\newblock {\em Invent. Math.}, 178(3):451--484, 2009.
\newblock \arxiv{0808.2032}.
\newblock \href {http://dx.doi.org/10.1007/s00222-009-0204-8}
  {\path{doi:10.1007/s00222-009-0204-8}}.

\bibitem[BK09b]{BK09b}
J.~Brundan and A.~Kleshchev.
\newblock Graded decomposition numbers for cyclotomic {H}ecke algebras.
\newblock {\em Adv. Math.}, 222(6):1883--1942, 2009.
\newblock \arxiv{0901.4450}.
\newblock \href {http://dx.doi.org/10.1016/j.aim.2009.06.018}
  {\path{doi:10.1016/j.aim.2009.06.018}}.

\bibitem[BK10]{BK10}
J.~Brundan and A.~Kleshchev.
\newblock The degenerate analogue of {A}riki's categorification theorem.
\newblock {\em Math. Z.}, 266(4):877--919, 2010.
\newblock \arxiv{0901.0057}.
\newblock \href {http://dx.doi.org/10.1007/s00209-009-0603-y}
  {\path{doi:10.1007/s00209-009-0603-y}}.

\bibitem[Bru]{Bru17}
J.~Brundan.
\newblock On the definition of {H}eisenberg category.
\newblock \arxiv{1709.06589}.

\bibitem[Bru08]{Bru08}
J.~Brundan.
\newblock Centers of degenerate cyclotomic {H}ecke algebras and parabolic
  category {$\mathscr{O}$}.
\newblock {\em Represent. Theory}, 12:236--259, 2008.
\newblock \arxiv{math/0607717}.
\newblock \href {http://dx.doi.org/10.1090/S1088-4165-08-00333-6}
  {\path{doi:10.1090/S1088-4165-08-00333-6}}.

\bibitem[Bru12]{Bru12}
J.~Brundan.
\newblock An orthogonal form for level two {H}ecke algebras with applications.
\newblock In {\em Algebraic groups and quantum groups}, volume 565 of {\em
  Contemp. Math.}, pages 29--53. Amer. Math. Soc., Providence, RI, 2012.
\newblock \arxiv{1104.0448}.
\newblock \href {http://dx.doi.org/10.1090/conm/565/11158}
  {\path{doi:10.1090/conm/565/11158}}.

\bibitem[Che87]{Che87}
I.~V. Cherednik.
\newblock A new interpretation of {G}el$\prime$fand-{T}zetlin bases.
\newblock {\em Duke Math. J.}, 54(2):563--577, 1987.
\newblock \href {http://dx.doi.org/10.1215/S0012-7094-87-05423-8}
  {\path{doi:10.1215/S0012-7094-87-05423-8}}.

\bibitem[CL12]{CL12}
S.~Cautis and A.~Licata.
\newblock Heisenberg categorification and {H}ilbert schemes.
\newblock {\em Duke Math. J.}, 161(13):2469--2547, 2012.
\newblock \arxiv{1009.5147}.
\newblock \href {http://dx.doi.org/10.1215/00127094-1812726}
  {\path{doi:10.1215/00127094-1812726}}.

\bibitem[CLL{\etalchar{+}}]{CLLSS16}
S.~Cautis, A.~D. Lauda, A.~Licata, P.~Samuelson, and J.~Sussan.
\newblock The elliptic {H}all algebra and the deformed {K}hovanov {H}eisenberg
  category.
\newblock \arxiv{1609.03506}.

\bibitem[CLLS]{CLLS15}
S.~Cautis, A.~D. Lauda, A.~Licata, and J.~Sussan.
\newblock {$W$}-algebras from {H}eisenberg categories.
\newblock {\em J. Inst. Math. Jussieu}.
\newblock \arxiv{1501.00589}.
\newblock \href {http://dx.doi.org/10.1017/S1474748016000189}
  {\path{doi:10.1017/S1474748016000189}}.

\bibitem[CSST10]{CST10}
T.~Ceccherini-Silberstein, F.~Scarabotti, and F.~Tolli.
\newblock {\em Representation theory of the symmetric groups}, volume 121 of
  {\em Cambridge Studies in Advanced Mathematics}.
\newblock Cambridge University Press, Cambridge, 2010.
\newblock The Okounkov-Vershik approach, character formulas, and partition
  algebras.
\newblock \href {http://dx.doi.org/10.1017/CBO9781139192361}
  {\path{doi:10.1017/CBO9781139192361}}.

\bibitem[Kad99]{Kad99}
L.~Kadison.
\newblock {\em New examples of {F}robenius extensions}, volume~14 of {\em
  University Lecture Series}.
\newblock American Mathematical Society, Providence, RI, 1999.
\newblock \href {http://dx.doi.org/10.1090/ulect/014}
  {\path{doi:10.1090/ulect/014}}.

\bibitem[Kho10]{Kho10}
M.~Khovanov.
\newblock Categorifications from planar diagrammatics.
\newblock {\em Jpn. J. Math.}, 5(2):153--181, 2010.
\newblock \arxiv{1008.5084}.
\newblock \href {http://dx.doi.org/10.1007/s11537-010-0925-x}
  {\path{doi:10.1007/s11537-010-0925-x}}.

\bibitem[Kho14]{Kho14}
M.~Khovanov.
\newblock Heisenberg algebra and a graphical calculus.
\newblock {\em Fund. Math.}, 225(1):169--210, 2014.
\newblock \arxiv{1009.3295}.
\newblock \href {http://dx.doi.org/10.4064/fm225-1-8}
  {\path{doi:10.4064/fm225-1-8}}.

\bibitem[KL09]{KL09}
M.~Khovanov and A.~D. Lauda.
\newblock A diagrammatic approach to categorification of quantum groups. {I}.
\newblock {\em Represent. Theory}, 13:309--347, 2009.
\newblock \arxiv{0803.4121}.
\newblock \href {http://dx.doi.org/10.1090/S1088-4165-09-00346-X}
  {\path{doi:10.1090/S1088-4165-09-00346-X}}.

\bibitem[KL10]{KL3}
M.~Khovanov and A.~D. Lauda.
\newblock A categorification of quantum {${\rm sl}(n)$}.
\newblock {\em Quantum Topol.}, 1(1):1--92, 2010.
\newblock \href {http://dx.doi.org/10.4171/QT/1} {\path{doi:10.4171/QT/1}}.

\bibitem[Kle05]{Kle05}
A.~Kleshchev.
\newblock {\em Linear and projective representations of symmetric groups},
  volume 163 of {\em Cambridge Tracts in Mathematics}.
\newblock Cambridge University Press, Cambridge, 2005.
\newblock \href {http://dx.doi.org/10.1017/CBO9780511542800}
  {\path{doi:10.1017/CBO9780511542800}}.

\bibitem[LRS18]{LRS18}
A.~Licata, D.~Rosso, and A.~Savage.
\newblock A graphical calculus for the {J}ack inner product on symmetric
  functions.
\newblock {\em J. Combin. Theory Ser. A}, 155:503--543, 2018.
\newblock \arxiv{1610.01862}.
\newblock \href {http://dx.doi.org/10.1016/j.jcta.2017.11.020}
  {\path{doi:10.1016/j.jcta.2017.11.020}}.

\bibitem[LS13]{LS13}
A.~Licata and A.~Savage.
\newblock Hecke algebras, finite general linear groups, and {H}eisenberg
  categorification.
\newblock {\em Quantum Topol.}, 4(2):125--185, 2013.
\newblock \arxiv{1101.0420}.
\newblock \href {http://dx.doi.org/10.4171/QT/37} {\path{doi:10.4171/QT/37}}.

\bibitem[Ols87]{Ols87}
G.~I. Olshanski\u\i.
\newblock Extension of the algebra {$U(\mathfrak{g})$} for infinite-dimensional
  classical {L}ie algebras {$\mathfrak{g},$} and the {Y}angians
  {$Y(\mathfrak{gl}(m))$}.
\newblock {\em Dokl. Akad. Nauk SSSR}, 297(5):1050--1054, 1987.

\bibitem[QSY18]{QSY18}
H.~Queffelec, A.~Savage, and O.~Yacobi.
\newblock An equivalence between truncations of categorified quantum groups and
  {H}eisenberg categories.
\newblock {\em J. \'Ec. polytech. Math.}, 5:197--238, 2018.
\newblock \arxiv{1701.08654}.
\newblock \href {http://dx.doi.org/10.5802/jep.68} {\path{doi:10.5802/jep.68}}.

\bibitem[Rou]{Rou08}
R.~Rouquier.
\newblock 2-{K}ac-{M}oody algebras.
\newblock \arxiv{0812.5023v1}.

\bibitem[RS17]{RS17}
D.~Rosso and A.~Savage.
\newblock A general approach to {H}eisenberg categorification via wreath
  product algebras.
\newblock {\em Math. Z.}, 286(1-2):603--655, 2017.
\newblock \arxiv{1507.06298}.
\newblock \href {http://dx.doi.org/10.1007/s00209-016-1776-9}
  {\path{doi:10.1007/s00209-016-1776-9}}.

\bibitem[Sava]{Sav17}
A.~Savage.
\newblock Affine wreath product algebras.
\newblock \arxiv{1709.02998}.

\bibitem[Savb]{Sav18}
A.~Savage.
\newblock Frobenius {H}eisenberg categorification.
\newblock \arxiv{1802.01626}.

\bibitem[Sua17]{Sua15}
D.~B. Suarez.
\newblock Integral presentations of quantum lattice {H}eisenberg algebras.
\newblock In {\em Categorification and Higher Representation Theory}, volume
  683 of {\em Contemp. Math.}, pages 247--259. Amer. Math. Soc., Providence,
  RI, 2017.
\newblock \arxiv{1509.02211}.
\newblock \href {http://dx.doi.org/10.1090/conm/683/13706}
  {\path{doi:10.1090/conm/683/13706}}.

\bibitem[SY15]{SY15}
A.~Savage and O.~Yacobi.
\newblock Categorification and {H}eisenberg doubles arising from towers of
  algebras.
\newblock {\em J. Combin. Theory Ser. A}, 129:19--56, 2015.
\newblock \arxiv{1309.2513}.
\newblock \href {http://dx.doi.org/10.1016/j.jcta.2014.09.002}
  {\path{doi:10.1016/j.jcta.2014.09.002}}.

\end{thebibliography}

\end{document}